\documentclass[11pt,a4paper]{amsart} 

\usepackage[french,english]{babel}	%langue française
\usepackage[T1]{fontenc}		%codage de caractère T1
\usepackage[utf8]{inputenc}		%codage de fichier utf-8
\usepackage{enumerate}    %énumération
\usepackage{lmodern}			%police de caractère vectorielle
\usepackage{dsfont}			%majuscules ensemblistes
\usepackage{amsmath, amsthm, amssymb, amsthm,bbm,bm,mathrsfs}%,mathtools}	%environnements mathématiques
\usepackage[pdfborder={0 0 0}]{hyperref}%hyperliens
\usepackage[toc,page]{appendix} %annexe
\usepackage{epsfig,color}
\usepackage[all]{xy}
\usepackage{framed}
\numberwithin{equation}{section}

\title{Matricial model for the free multiplicative convolution }
\author{Guillaume Cébron}
\dedicatory{Laboratoire de Probabilités et Modèles Aléatoires (LPMA)}

\date{2014}
\email{guillaume.cebron@upmc.fr}
\subjclass[2010]{Primary 15B52, 60B15; secondary 46L54, 60E07.}
\keywords{Random matrices, free probability, infinitely divisible distributions.}
\address{LPMA, UMR 7599,
Université Paris 6
, 4 place Jussieu 
, 75252 Paris Cedex 05 (FRANCE)}
\newtheorem{theorem}{Theorem}[section] 
\newtheorem{lemma}[theorem]{Lemma}
\newtheorem{proposition}[theorem]{Proposition}
\newtheorem{propdef}[theorem]{Proposition-Definition}

\newtheorem{corollary}[theorem]{Corollary}

\newtheorem{theo}{Theorem}
\newtheorem*{theo*}{Theorem}
\theoremstyle{definition}
\newtheorem{definition}[theorem]{Definition}
\theoremstyle{remark}
\newtheorem{remark}[theorem]{Remark}
\newtheorem{example}[theorem]{Example}

\newcommand{\diff}{\mathrm{d}}
\newcommand{\Tr}{\operatorname{Tr}}

\newcommand{\Log}{\operatorname{Log}}
\newcommand{\e}{\mathrm{\mathbf{e}}}

\newcommand{\Id}{\mathrm{Id}}

\newcommand{\Poiss}{\mathrm{Poiss}}
\newcommand{\inv}{\mathrm{inv}}

\newcommand{\End}{\mathrm{End}}
\newcommand{\ID}{\mathcal{ID}}
\newcommand{\Wg}{\mathrm{Wg}}
\newcommand{\Haar}{\mathrm{Haar}}
\newcommand{\R}{\mathbb{R}}
\newcommand{\C}{\mathbb{C}}
\newcommand{\N}{\mathbb{N}}
\newcommand{\Z}{\mathbb{Z}}
\newcommand{\E}{\mathbb{E}}
\newcommand{\U}{\mathbb{U}}

\newcommand{\Mt}{\mathcal{M}_{\mathbb{U}}}
\newcommand{\Ms}{\mathcal{M}_{\ast}}

\begin{document}

\maketitle
\begin{abstract}
This paper investigates homomorphisms à la Bercovici-Pata between additive and multiplicative convolutions. We also consider their matricial versions which are associated with measures on the space of Hermitian matrices and on the unitary group. The previous results combined with a matricial model of Benaych-Georges and Cabanal-Duvillard allows us to define and study the large $N$ limit of a new matricial model on the unitary group for free multiplicative Lévy processes. 
\end{abstract}

%\tableofcontents
%\newpage

\section{Introduction}
%\addtocontents{toc}{\protect\setcounter{tocdepth}{-1}}
The classical convolution $\ast$ on $\mathbb{R}$ and the classical multiplicative convolution $\circledast$ on the unit circle $\mathbb{U}=\{z\in \mathbb{C}:|z|=1\}$, which correspond respectively to the addition and to the product of independent random variables, have analogues in free probability. Indeed, replacing the concept of classical independence by the concept of freeness, Voiculescu defined the free convolution $\boxplus$ on $\mathbb{R}$, and the free multiplicative convolution $\boxtimes$ on $\U$ (we refer the reader to~\cite{Voiculescu1992} for an introduction to free convolutions). 
A probability measure $\mu$ on $\R$ is said to be $\ast$-\textit{infinitely divisible} 
if, for all $n\in \mathbb{N}^*$, there exists a probability measure $\mu_n $ such that $\mu_n^{\ast n}=\mu$. The set of {$\ast$-infinitely} divisible probability measures endowed with the operation $\ast$ is a semigroup which we will denote by $\ID(\R,\ast)$, and we consider analogously the sets $\ID(\U,\circledast)$, $\ID(\mathbb{R},\boxplus)$ and $\ID(\U,\boxtimes)$.

In~\cite{Bercovici1999}, Bercovici and Pata identified a isomorphism of semigroup~$\Lambda$ between $\ID(\mathbb{R},\ast)$ and $\ID(\mathbb{R},\boxplus)$ which has a good behaviour with respect to limit theorems: for all $\mu\in\ID(\mathbb{R},\ast)$ and all sequence $(\mu_n)_{n\in \N}$ of probability measures on $\R$,
\[\begin{array}{ccc}\mu_n^{\ast n}  \overset{(w)}{\underset{n\to+\infty}{\longrightarrow}}  \mu&  \Longleftrightarrow &  \mu_n^{\boxplus n}  \overset{(w)}{\underset{n\to+\infty}{\longrightarrow}}   \Lambda(\mu)\end{array}\]
where the convergence is  the weak convergence of measures.
%$$\left.\begin{array}{ccccccc}\underbrace{\mu_n\ast\cdots \ast \mu_n}_{n \text{ times}} &  \overset{(w)}{\underset{n\to+\infty}{\longrightarrow}} & \mu & \Longleftrightarrow & \underbrace{\mu_n\boxplus\cdots \boxplus \mu_n}_{n \text{ times}} & \overset{(w)}{\underset{n\to+\infty}{\longrightarrow}} &  \Lambda(\mu).\end{array}\right.$$
% the measure $\mu_n^{\ast n}$ converges weakly to $\mu$ if and only if $\mu_n^{\boxplus n}$ converges weakly to $\Lambda(\mu)$ as $n$ tends to $\infty$.
Unfortunately, the situation is not as symmetric in the multiplicative case. Let $\Ms$ denote the set of probability measures $\mu$ on $\U$ such that $\int_\U \zeta \diff \mu (\zeta)\neq0$. In~\cite{Chistyakov2008}, Chistyakov and Götze proved that, given a sequence $(\mu_n)_{n\in \N}$ of probability measures on $\U$, the weak convergence of $\mu_n^{\boxtimes n}$ to any measure of $\Ms$ implies the weak convergence of $\mu_n^{\circledast n}$; but they also proved that the converse is false. It is thus only possible to define a homomorphism of semigroup~$\Gamma$ between $\ID(\U,\boxtimes)$ and $\ID(\U,\circledast)$ (see Definition~\ref{gammadef}) such that, for all $\mu\in\ID(\U,\boxtimes)\cap \Ms$ and all sequence $(\mu_n)_{n\in \N}$ of probability measures on $\U$,
$$\begin{array}{ccc}\mu_n^{\boxtimes n}  \overset{(w)}{\underset{n\to+\infty}{\longrightarrow}}  \mu&  \Longrightarrow &  \mu_n^{\circledast n}  \overset{(w)}{\underset{n\to+\infty}{\longrightarrow}}   \Gamma(\mu).\end{array}$$
Finally, the homomorphism $\e:x\mapsto e^{ix}$ from $(\R,+)$ to $(\U,\times)$ induces a homomorphism of semigroup~$\e_\ast$ between $\ID(\mathbb{R},\ast)$ and $\ID(\U,\circledast)$, given by the push-forward of measures, which enjoys a similar property: for all $\mu\in\ID(\mathbb{R},\ast)$ and all sequence $(\mu_n)_{n\in \N}$ of probability measures on $\R$,
$$\begin{array}{ccc}\mu_n^{\ast n}  \overset{(w)}{\underset{n\to+\infty}{\longrightarrow}}  \mu&  \Longrightarrow &  \e_\ast(\mu_n)^{\circledast n}  \overset{(w)}{\underset{n\to+\infty}{\longrightarrow}}   \e_\ast(\mu).\end{array}$$
The first aim of this work is to complete the picture which we just sketched. In Definition~\ref{edef}, we shall introduce a new homomorphism of semigroup~$\e_{\boxplus}$ between $\ID(\mathbb{R},\boxplus) $ and $\mathcal{ID}(\mathbb{U},\boxtimes)$, and which is linked to the previous homomorphisms in the following way.
% (see Definition~\ref{edef}, Theorem~\ref{thone}, Definition~\ref{gammadef} and Proposition~\ref{thtwo}):
 \begin{theo}[see Prop.~\ref{thtwo} and Thm.~\ref{thone}]The \label{un}map $\e_{\boxplus}:\ID(\mathbb{R},\boxplus) \to\mathcal{ID}(\mathbb{U},\boxtimes)$ is such that:
 \begin{enumerate}
 \item For all $\mu\in\ID(\mathbb{R},\boxplus)$ and all sequence $(\mu_n)_{n\in \N}$ of probability measures on $\R$,
\begin{equation*}\begin{array}{ccc}\mu_n^{\boxplus n}  \overset{(w)}{\underset{n\to+\infty}{\longrightarrow}}  \mu&  \Longrightarrow &  \e_\ast(\mu_n)^{\boxtimes n}  \overset{(w)}{\underset{n\to+\infty}{\longrightarrow}}   \e_\boxplus(\mu);\end{array}  \label{cwu}
\end{equation*}
 
 \item  The following diagram commutes:
\begin{equation}
%\begin{array}{ccc} \ID(\mathbb{R},\ast) & \xrightarrow{\Lambda} &\ID(\mathbb{R},\boxplus) \\ \downarrow{\e_\ast} &  & \downarrow{\e_\boxplus} \\ \mathcal{ID}(\mathbb{U},\circledast) &  \xleftarrow{\Gamma} &\mathcal{ID}(\mathbb{U},\boxtimes). \end{array}
\xymatrix{
    \ID(\mathbb{R},\ast) \ar[r]^-*+{\Lambda} \ar[d]_-*+{\e_\ast} &  \ID(\mathbb{R},\boxplus) \ar[d]^-*+{\e_\boxplus} \\
    \mathcal{ID}(\mathbb{U},\circledast)  & \mathcal{ID}(\mathbb{U},\boxtimes) \ar[l]^-*+{\Gamma}.
  }
  \label{diag}
\end{equation}
 \end{enumerate}
 \end{theo}

In the highly non-commutative theory of Lie groups, there is a well-known process which connects additive infinitely divisible laws with multiplicative ones. It consists in passing to the limit the product of multiplicative little increments which are built from additive increments using the exponential map (see~\cite{Estrade1992}). A natural question is whether there exists a matrix approximation of $\e_\boxplus$ which arises from this procedure.

Our starting point is a matricial model for $\ID(\mathbb{R},\boxplus)$ which has been constructed simultaneously by Benaych-Georges and Cabanal-Duvillard in~\cite{Benaych-Georges2005} and~\cite{Cabanal-Duvillard2005}. For all $N\in \mathbb{N}$, let us consider the classical convolution~$\ast$ on the set of Hermitian matrices $\mathcal{H}_N$, and denote by $\ID_\inv(\mathcal{H}_N,\ast)$ the set of infinitely divisible probability measures on $\mathcal{H}_N$ which are invariant under conjugation. For all ${\mu \in\ID(\mathbb{R},\boxplus)}$, Benaych-Georges and Cabanal-Duvillard proved that there exists an element of $\ID_\inv(\mathcal{H}_N,\ast)$, which we shall denote by $\Pi_N(\mu)$  (see Section~\ref{pin}), such that:
\begin{enumerate}
\item For all $\mu\in\ID(\mathbb{R},\boxplus)$, the spectral measure of a random matrix with distribution $\Pi_N(\mu)$ converges weakly to $\mu$ in probability as $N$ tends to infinity;
\item $\Pi_N : \ID(\mathbb{R},\boxplus)\rightarrow \ID_\inv(\mathcal{H}_N,\ast)$ is a homomorphism of semigroup.
\end{enumerate}
On the other hand, the map $\e:H\mapsto e^{iH}$ from $\mathcal{H}_N$ to the unitary group $U(N)$ induces, with some care, a homomorphism of semigroup from $\ID_\inv(\mathcal{H}_N,\ast)$ to the set $\ID_\inv(U(N),\circledast)$ of infinitely divisible measures on $U(N)$ which are invariant under conjugation. Indeed, for all $\mu\in \ID_\inv(\mathcal{H}_N,\ast)$, the sequence $(\e_{\ast}(\mu^{\ast 1/n})^{\circledast n})_{n\in \N^*}   $ converges weakly to a measure $\mathcal{E}_N(\mu)\in \ID_\inv(U(N),\circledast)$ (see Proposition-Definition~\ref{ththree}). The situation can be summed up in the following diagram:
\begin{equation}
%\begin{array}{ccc} \ID(\mathbb{R},\boxplus) & \xrightarrow{\Pi_N} & \ID_\inv(\mathcal{H}_N,\ast)\\  \downarrow{\e_\boxplus}&  & \downarrow{ \mathcal{E}_N} \\ \mathcal{ID}(\mathbb{U},\boxtimes)&  & \ID_\inv(U(N),\ast). \end{array} \label{diagN}
\xymatrix{
    \ID(\mathbb{R},\boxplus)\ar[r]^-*+{\Pi_N} \ar[d]_-*+{\e_\boxplus} & \ID_\inv(\mathcal{H}_N,\ast) \ar[d]^-*+{\mathcal{E}_N} \\
    \mathcal{ID}(\mathbb{U},\boxtimes)&\ID_\inv(U(N),\circledast).
  }\label{diagN}
  \end{equation}
When $N=1$, we have $\Pi_1=\Lambda^{-1}$, $ \mathcal{E}_1=\e_\ast$, and consequently the diagram~\eqref{diagN} is exactly the top part of the diagram~\eqref{diag}. The second main result of this work is the definition of a homomorphism of semigroup $\Gamma_N : \mathcal{ID}(\mathbb{U},\boxtimes)\rightarrow \ID_\inv(U(N),\circledast) $ which completes the picture as follows (see Section~\ref{gamman}).
\begin{theo}[see Prop.~\ref{thfive} and Cor.~\ref{thsix}]
The map $\Gamma_N$ is such that:\label{convwe}\label{trois}
\begin{enumerate}
\item for all $\mu\in\mathcal{ID}(\mathbb{U},\boxtimes)$, the spectral measure of a random matrix $U^{(N)}$ with distribution $\Gamma_N(\mu)$ converges weakly to $\mu$ in expectation, in the sense that, for each continuous function $f$ on $\U$, one has the convergence
$$\lim_{N\to \infty}\frac{1}{N}\E\left[\Tr\left(f\left(U^{(N)}\right)\right)\right]=\int_\U f \diff \mu; $$
\item The following diagram commutes
\begin{equation}
\xymatrix{
    \ID(\mathbb{R},\boxplus)\ar[r]^-*+{\Pi_N} \ar[d]_-*+{\e_\boxplus} & \ID_\inv(\mathcal{H}_N,\ast) \ar[d]^-*+{\mathcal{E}_N} \\
    \mathcal{ID}(\mathbb{U},\boxtimes)\ar[r]_-*+{\Gamma_N}&\ID_\inv(U(N),\circledast).
  } \label{ding}
\end{equation}
\end{enumerate}
\end{theo}
This result can be expressed by saying that the map $\e_\boxplus$ is the limit of the map $ \mathcal{E}_N$ as $N$ tends to infinity. The first assertion of the theorem above is a generalisation of a result of Biane: in~\cite{Biane1997a}, he proved that the spectral measure of a Brownian motion on $U(N)$ with adequately chosen speed converges to the distribution of a free unitary Brownian motion at each fixed time. The distribution of a Brownian motion is indeed an infinitely divisible measure at each time, and this convergence can be viewed as a particular case of Theorem~\ref{convwe}.

The proof itself of Theorem~\ref{trois} is interesting at least for two reasons. It is the first time that the free log-cumulants, originated in~\cite{Mastnak2010},  are used for proving an asymptotic result of random matrices. Secondly, the proof relies upon a key object, the symmetric group $\mathfrak{S}_n$, which is linked to both the combinatorics of free probability theory, and the computation of conjugate-invariant measures on $U(N)$. More precisely, in~\cite{LEVY2008}, Lévy established that the asymptotic distribution of a Brownian motion on the unitary group is closely related to the counting of paths in the Caley graph of $\mathfrak{S}_n$. Similarly, for all $\mu\in\mathcal{ID}(\mathbb{U},\boxtimes)$, the asymptotic distribution of a random matrix with law $\Gamma_N(\mu)$ involves the counting of paths in $\mathfrak{S}_n$, each step of which is given by the following generator (see Lemma~\ref{superlemme})$$T(\sigma)  =n L\kappa_1(\mu)\cdot \sigma+\displaystyle\sum_{\substack{2\leq m \leq n\\c\ m\textup{-cycle of }\mathfrak{S}_n\\c\sigma\preceq \sigma}} L\kappa_{m}\left(\mu \right)\cdot c\sigma, $$
where $(L\kappa_n(\mu))_{n\in \N^*}$ are the free log-cumulants of $\mu$.

In fact, Biane proved in~\cite{Biane1997a} a stronger result: the convergence of all finite dimensional distributions of the Brownian motion on $U(N)$ to the distribution of a free unitary Brownian motion. Similarly, a classical result of Voiculescu allows us to strengthen the previous asymptotic results of Theorem~\ref{convwe} as follows (see Section~\ref{wholly} for details).

\begin{theo}Let $( U_t)_{t\in\R_+}$ be a free unitary multiplicative Lévy process with marginal distributions $( \mu_t)_{t\in\R_+}$ in $\Ms$. For all $N\in \N^*$, let $(U_t^{(N)})_{t\in \R_+}$ be a Lévy process with marginal distributions $(\Gamma_N(\mu_t))_{t\in \R_+}$. Then, $(U_t^{(N)})_{t\in \R_+}$ converges to $( U_t)_{t\in\R_+}$ in non-commutative distribution. In other words, for each integer $n\geq 1$, for each non-commutative polynomial $P$ in $n$ variables, and each choice of $n$ non-negative reals $t_1,\ldots,t_n$, one has the convergence\label{whole}
$$\lim_{N\to \infty}\frac{1}{N}\E\left[\Tr\left(P\left(U_{t_1}^{(N)},\ldots,U_{t_n}^{(N)} \right)\right)\right]=\tau(P\left(U_{t_1},\ldots,U_{t_n} \right)). $$
Moreover, independent copies of $(U_t^{(N)})_{t\in \R_+}$ converge to freely independent copies of $( U_t)_{t\in\R_+}$.\label{ssss}\label{quatre}
\end{theo}

The rest of the paper is organized as follows. In Section~\ref{secone}, we give an overview of the theory of infinitely divisible measures. In Section~\ref{secthree}, we define $\e_\boxplus$ and $\Gamma$ and we prove Theorem~\ref{un}. Section~\ref{flc} is devoted to the notion of free log-cumulants, which is an important tool for the proof of the asymptotic results of this paper. Section~\ref{secfour} presents a description of convolution semigroups on the unitary group, and studies more precisely those which are invariant by conjugation. Section~\ref{secfive} links together  the measures on the Hermitian matrices with the measures on the unitary matrices through the stochastic exponentiation $ \mathcal{E}_N$. Finally, Section~\ref{secsix} provides the definition of the random matrix models $\Pi_N$ and $\Gamma_N$, and the proof of Theorem~\ref{trois} and Theorem~\ref{quatre}.

%\addtocontents{toc}{\protect\setcounter{tocdepth}{2}}

\section{Infinite divisibility for unidimensional convolutions}\label{secone}
In this section, we give the necessary background concerning $\ID(\mathbb{R},\ast)$, $\ID(\U,\circledast)$, $\ID(\mathbb{R},\boxplus)$ and $\ID(\U,\boxtimes)$. In particular, we give a description of the characteristic pair and the characteristic triplet of an infinitely divisible measure in each case.

% cf.~\cite{Barndorff2005,Bercovici1999,Bercovici1993,Parthasarathy1967,Sato1999}.

%Let us denote respectively by $$\mathcal{P}(\R)$$, $\Mr$ and $\Mt$ the set of compactly supported Borel probability measures on respectively $\C$, $\R$ and $\U$. We also denote by $\Ms$ the subspace composed of $\mu\in \Mg$ such that $\int_\C z\diff \mu (z)\neq0$. Finally, we denote by $\mathcal{P}(\R)$ the set of Borel probability measures on $\R$.
We say that a sequence of finite measures $(\mu_n)_{n\in \N}$ on $\C$ \textit{converges weakly} to a measure $\mu$ if for all continuous and bounded complex function $f$, $\lim_{n\rightarrow \infty}\int_{\C}f\diff \mu_n = \int_{\C}f\diff \mu. $

\subsection{Classical infinite divisibility on $\R$}\label{adconv}

Let $\mu\in \ID(\R,\ast)$. There exists a sequence $(\mu_n)_{n\in \N^*}$ of probability measures such that, for all $n\in \N^*$,  $\mu_n^{\ast n}=\mu$. The important fact is that the measures
$$\diff \sigma_n(x)=n\frac{ x^2}{x^2+1}\mu_n(\diff x)$$
converge weakly to a measure $\sigma$ and the reals
$$\gamma_n=n\int_{\mathbb{R}} \frac{x}{x^2+1}\mu_n(\diff x)$$ converge to a constant $\gamma\in \mathbb{R}$.
% 
%
%
%In this case, the measure $n\frac{ x^2}{x^2+1}\mu_n(\diff x)$ converges weakly to a finite measure $\sigma$ on $\mathbb{R}$, and $n\int_{\mathbb{R}} \frac{x}{x^2+1}\mu_n(\diff x)$ converges to a constant $\gamma\in \mathbb{R}$ as $n$ tends to $\infty$. 
The pair $(\gamma, \sigma)$ is known as the $\ast$-\textit{characteristic pair} for $\mu$ and it is uniquely determined by $\mu$. More generally, we have the following characterization. 
\begin{theorem}[\cite{Bercovici1999}]\label{critc}Let $\mu\in \ID(\R,\ast)$ with $\ast$-characteristic pair $(\gamma, \sigma)$.
Let $k_1<k_2<\cdots$ be natural numbers and  $(\mu_n)_{n\in \N^*}$ be a sequence of probability measures on $\R$. The following assertions are equivalent:
\begin{enumerate}
\item the measures $\underbrace{\mu_n\ast\cdots \ast \mu_n}_{k_n \text{ times}}$ converge weakly to $\mu$;
\item the measures  $$\diff \sigma_n(x)=k_n\frac{ x^2}{x^2+1}\mu_n(\diff x)$$
converge weakly to $\sigma$ and
$$\lim_{n\to \infty}k_n\int_{\mathbb{R}} \frac{x}{x^2+1}\mu_n(\diff x)=\gamma.$$
\end{enumerate}
%\begin{equation}
%\mu_n ^{\ast k_n} \wn\mu \Longleftrightarrow k_n\frac{x^2}{x^2+1}\mu_n(\diff x)\wn\sigma\text{ and } k_n\int_\R \frac{x}{x^2+1}\mu_n(\diff x)\underset{n\to+\infty}{\longrightarrow} \gamma.\label{critc}
%\end{equation}
\end{theorem}
In addition to~\cite{Bercovici1999}, we refer the reader to the very complete lecture notes~\cite{Barndorff2005}. We present now two additional properties of the \mbox{$\ast$-characteristic} pairs. Firstly, there is an one-to-one correspondence between {$\ast$-infinitely} divisible probability measures and pairs $(\gamma, \sigma)$. Indeed, for all finite measure $\sigma$ on $\mathbb{R}$, and all constant $\gamma\in \mathbb{R}$, there exists a unique {$\ast$-infinitely} divisible probability measure such that $(\gamma, \sigma)$ is the \mbox{$\ast$-characteristic} pair for $\mu$. Secondly, the \mbox{$\ast$-characteristic} pairs linearize the convolution: let $\mu_1$ and $\mu_2$ be two {$\ast$-infinitely} divisible measures with respective $\ast$-characteristic pairs $(\gamma_1, \sigma_1)$ and $(\gamma_2, \sigma_2)$. The measure $\mu_1\ast\mu_2$ is a \mbox{$\ast$-infinitely} divisible measure with $\ast$-characteristic pair $(\gamma_1+\gamma_2, \sigma_1+\sigma_2)$.
%\begin{remark}
%The parametrization of additive infinitely divisible measures is subject to variations. Here, the function $x\mapsto \frac{x^2}{x^2+1}$ does not play a crucial role. If we replace this function by another function bounded far from $0$ and equivalent to $x^2$ around $0$, it leads to another possible parametrization.\end{remark}

Let us review another, perhaps more classical, characterization of infinitely divisible measures. Let $\mu$ be {$\ast$-infinitely} divisible and $(\gamma, \sigma)$ be its $\ast$-characteristic pair. We set
\begin{equation}
a=\sigma(\{0\}),\ \rho(\diff x)=\frac{1+x^2}{x^2}\cdot 1_{\mathbb{R}\setminus\{0\}}(x) \sigma (\diff x),\ \text{and}\ \eta=\gamma+\int_{\mathbb{R}}x\left(1_{[-1,1]}(x)-\frac{1}{1+x^2}\right)\rho(\diff x).\label{rel}
\end{equation}
The triplet $(\eta,a,\rho)$ is called the $\ast$-\textit{characteristic triplet} for $\mu$. Observe that $\rho$ is such that the function $x\mapsto \min(1,x^2)$ is $\rho$-integrable and $\rho(\{0\})=0$. Such a measure is called a \textit{Lévy measure} on $\R$. Conversely, for all $(\eta,a,\rho)$ with $\eta\in \R$, $a\geq 0$ and $\rho$ a Lévy measure on $\R$, there exists a unique {$\ast$-infinitely} divisible probability measure such that $(\eta,a,\rho)$ is the $\ast$-characteristic triplet for $\mu$.

\begin{framed}
\begin{example}Here are three important classes of {$\ast$-infinitely} divisible measures:\label{exac}
\begin{enumerate}
\item For any constant $\eta$ in $\R$, the Dirac distribution $\delta_\eta$ is in $\mathcal{ID}(\mathbb{R},\ast)$, and its \mbox{$\ast$-characteristic} triplet is $(\eta,0,0)$;
\item For any constant $a>0$, the Gaussian distribution of variance $a$ is $$\mathcal{N}_a(\diff x)=\frac{1}{\sqrt{2\pi a}}e^{-\frac{x^2}{2a}}\diff x\in \mathcal{ID}(\mathbb{R},\ast)$$ whose $\ast$-characteristic triplet is $(0,a,0)$;
\item For any constant $\lambda>0 $ and any probability measure $\rho\in \mathcal{P}(\R)$, the compound Poisson distribution with rate $\lambda$ and jump distribution $\rho$ is $$\Poiss^*_{\lambda,\rho}=e^{-\lambda}\sum_{n\in \N}\frac{\lambda^n}{n!}\rho^{\ast n}\in \mathcal{ID}(\mathbb{R},\ast)$$ whose $\ast$-characteristic triplet is $(\lambda \int_{[-1,1]}x\rho(\diff x),0,\lambda \rho_{|\R\setminus \{0\}})$. One important particular case is when $\rho=\delta_1$: the Poisson distribution $\Poiss^*_{\lambda}$ of mean $\lambda$ is 
$$\Poiss^*_\lambda(\diff x)=\Poiss^*_{\lambda,\delta_1}(\diff x)=e^{-\lambda}\sum_{n\in \N}\frac{\lambda^n}{n!}\delta_n\in \mathcal{ID}(\mathbb{R},\ast).$$
%\item For any constant $\lambda \in \R$, the Poisson distribution with mean $\lambda$ is $$\Poiss^*_\lambda(\diff x)=e^{-\lambda}\sum_{n\in \N^*}\frac{\lambda^n}{n!}\delta_n\in \mathcal{ID}(\mathbb{R},\ast)$$ whose characteristic triplet is $(0,0,\lambda \delta_1)$.
\end{enumerate}
\end{example}
\end{framed}
\subsection{The Bercovici-Pata bijection}

In~\cite{Bercovici1999}, Bercovici and Pata proved that all results of the previous section stay true if one replaces the classical convolution $\ast$  by the free additive convolution $\boxplus$. This leads to the \textit{Bercovici-Pata bijection} $\Lambda$ from $\mathcal{ID}(\mathbb{R},\ast)$ to $\mathcal{ID}(\mathbb{R},\boxplus)$ which maps a \mbox{$\ast$-infinitely} divisible measure with \mbox{$\ast$-characteristic} pair $(\gamma, \sigma)$ to the {$\boxplus$-infinitely} divisible measure with \mbox{$\boxplus$-characteristic} pair $(\gamma, \sigma)$. Its importance is due to the following theorem.
\begin{theorem}[\cite{Bercovici1999}]
The Bercovici-Pata bijection $\Lambda$ has the following properties:\label{bercopata}
\begin{enumerate}
\item For all $\mu,\nu\in \mathcal{ID}(\mathbb{R},\ast)$, $\Lambda(\mu\ast \nu)=\Lambda(\mu)\boxplus \Lambda(\nu)$;
\item For all natural numbers $k_1<k_2<\cdots$,  all sequence $(\mu_n)_{n\in \N^*}$ of probability measures on $\R$ and all {$\ast$-infinitely} divisible measure $\mu$, the measures $\mu_n^{\ast k_n}$ converge weakly to $\mu$ if and only if the measures $\mu_n^{\boxplus k_n}$ converge weakly to $\Lambda(\mu)$.
\end{enumerate}
\end{theorem}

\begin{framed}
\begin{example}Here are the free analogues of the measures presented in Example~\ref{exac}:\label{exaf}
\begin{enumerate}
\item For any constant $\eta$ in $\R$, we have $\Lambda(\delta_\eta)=\delta_\eta\in \mathcal{ID}(\mathbb{R},\boxplus)$, and its $\boxplus$-characteristic triplet is $(\eta,0,0)$.
\item For any constant $a>0$, the semi-circular distribution of variance $a$ is $$\mathcal{S}_a(\diff x)=\frac{1}{2\pi a}\sqrt{4a-x^2}\cdot 1_{[-2\sqrt{a},2\sqrt{a}]}(x)\diff x\in \mathcal{ID}(\mathbb{R},\boxplus)$$ whose characteristic triplet is $(0,a,0)$. We have $\Lambda(\mathcal{N}_a)=\mathcal{S}_a$.
\item For any constant $\lambda >0$, the free Poisson distribution with mean $\lambda$, also called the Marçenko-Pastur distribution, is $$\Poiss^\boxplus_{\lambda,\delta_1}(\diff x)=\left\{\begin{array}{ll}\displaystyle (1-\lambda)\delta_0+\frac{1}{2\pi x}\sqrt{(x-a)(b-x)}1_{a\leq x\leq b}\diff x & \text{if }0\leq \lambda \leq 1, \\ \displaystyle\frac{1}{2\pi x}\sqrt{(x-a)(b-x)} 1_{a\leq x\leq b}\diff x& \text{if }\lambda >1,\end{array}\right.$$where $a=(1-\sqrt{\lambda})^2$ and $b=(1+\sqrt{\lambda})^2$. Its $\boxplus$-characteristic triplet is $(\lambda,0,\lambda \delta_1)$. More generally, for any constant $\lambda>0 $ and probability measure $\rho\in \mathcal{P}(\R)$, the free compound Poisson distribution with rate $\lambda$ and jump distribution $\rho$ is the measure $\Poiss^\boxplus_{\lambda,\rho}\in \mathcal{ID}(\mathbb{R},\boxplus)$ whose \mbox{$\boxplus$-characteristic} triplet is $(\lambda \int_{[-1,1]}x\rho(\diff x),0,\lambda \rho)$. We have $\Lambda(\Poiss^\ast_{\lambda,\rho})=\Poiss^\boxplus_{\lambda,\rho}$.
%\item For any constant $\lambda \in \R$, the Poisson distribution with mean $\lambda$ is $$\Poiss^*_\lambda(\diff x)=e^{-\lambda}\sum_{n\in \N^*}\frac{\lambda^n}{n!}\delta_n\in \mathcal{ID}(\mathbb{R},\ast)$$ whose characteristic triplet is $(0,0,\lambda \delta_1)$.
\end{enumerate}
\end{example}
\end{framed}
%We finish this section with two technical facts. The first one is the Fourier transform of a {$\ast$-infinitely} divisible measure (see~\cite{Bercovici1999}, or~\cite{Sato1999}). Let $\mu\in \mathcal{ID}(\mathbb{R},\ast)$ and $(\eta,a,\rho)$ be its $\ast$-characteristic triplet. For all $\theta\in \R$, we have
%\begin{equation}
%\int_\R e^{i\theta x}\diff\mu(x)=\exp\left(i\eta \theta -\frac{1}{2}a\theta^2+\int_\R (e^{i\theta x}-1-i\theta x 1_{[-1,1]}(x))\diff \rho(x)\right).\label{char}
%\end{equation}
%The second one is a straightforward reformulation of Theorem~\ref{critc}, using the relation given by~\eqref{rel}.

We finish this section with a technical lemma, which is a straightforward reformulation of Theorem~\ref{critc}, using the relation given by~\eqref{rel}.

\begin{lemma}\label{lemconv}Let $\mu\in \mathcal{ID}(\mathbb{R},\boxplus)$ and $(\eta,a,\rho)$ be its $\boxplus$-\textit{characteristic triplet}. Let $k_1<k_2<\cdots$ be natural numbers and $(\mu_n)_{n\in \N^*}$ a sequence of probability measures on $\R$ such that the measures $\mu_n ^{\boxplus k_n} $ converge weakly to $\mu$. Then, for all $f:\R\to \C$ continuous, bounded, and such that $f(x)\sim_{x\to0} f_0 x^2$, we have
$$\lim_{n\to \infty}k_n\int_\R f\diff\mu_n = \int_\R f \diff \rho+a f_0,\text{ and }\lim_{n\to \infty} k_n\int x1_{[-1,1]}(x)\diff \mu_n(x)=\eta.$$
\end{lemma}

\subsection{Classical infinite divisibility on $\U$}As we will now see, the particularity of $\mathcal{ID}(\U,\circledast)$ is the existence of idempotent measures, a infinite class which has no equivalent in $\mathcal{ID}(\mathbb{R},\ast)$, $\mathcal{ID}(\mathbb{R},\boxplus)$ or $\mathcal{ID}(\U,\boxtimes)$. Our references in this section are~\cite{Chistyakov2008,Parthasarathy1967,Schatte1983}.

A probability measure $\mu$ on $\U$ is said to be idempotent if $\mu\circledast \mu=\mu$. Each compact subgroup of $\U$ leads to an idempotent measure given by its Haar measure. More concretely, let $m\in \N$. The $m$-th roots of unity form a subgroup of $\U$, whose Haar measure is denoted by $\lambda_m$. We have $\lambda_m\circledast \lambda_m=\lambda_m$ and consequently $\lambda_m\in \mathcal{ID}(\U,\circledast)$. We denote by $\lambda$, or $\lambda_\infty$, the Haar measure on $\U$, which is also {$\circledast$-infinitely} divisible. Fortunately, the measures $(\lambda_m)_{m\in\N\cup\{\infty\}}$ are the unique measures on $\U$ which are idempotent.

How to identify measures of $\mathcal{ID}(\U,\circledast)$ which are not idempotent? Recall that $ \Ms$ is the set of probability measures $\mu$ on $\U$ such that $\int_\U \zeta \diff \mu (\zeta)\neq0$. It is easy to see that measures in $ \Ms$ are not idempotent, with the exception of $\delta_1$. In fact, every measure in $\mathcal{ID}(\U,\circledast)$ factorizes into the product of an idempotent measure with a measure in $\mathcal{ID}(\U,\circledast)\cap \Ms$. For the study of $\mathcal{ID}(\U,\circledast)\cap \Ms$,
%A probability measure $\mu$ on $\U$ is a factor of another probability measure $\nu$ if $\nu\circledast \nu'=\mu$ for some measure $\nu'$. For all {$\circledast$-infinitely} divisible measure $\mu$, there always exists $m\in \N\cup\{\infty\}$ such that $\mu$ is the convolution of a measure $\lambda_m$ with another {$\circledast$-infinitely} divisible measure without idempotent factors. Consequently, we focus us on {$\circledast$-infinitely} divisible measures without idempotent factors.
 it is useful to introduce the \textit{characteristic function}: for  all probability measure $\mu$ on $\U$, it is the function $\widehat{\mu}:\Z \rightarrow \C$ defined for all $k\in \Z$ by
$$
\widehat{\mu}(k)=\int_\U \zeta^k  \diff \mu(\zeta).$$
It is multiplicative for the convolution $\circledast$ in the sense that, for all $\mu,\nu$ probability measures on $\U$, and all $k\in \Z$, we have
\begin{equation}\widehat{\mu \circledast\nu}(k)=\widehat{\mu}(k)\cdot \widehat{\nu}(k).\label{ch}
\end{equation}
For all $m\in \N^*$ and $k\in \Z$, we obviously have $\widehat{\lambda_m}(k)= 1$ if $k$ is divisible by $m$ and $0$ if not. Using the characteristic function, we can now characterize the measures in $\mathcal{ID}(\U,\circledast)\cap \Ms$. Let $\mu\in \mathcal{ID}(\U,\circledast)\cap \Ms$. There exists a finite measure $\nu $ on $\U$ and a real $\alpha\in \R$ such that, for all $k\in \Z$,  $$\widehat{\mu}(k)=e^{ik \alpha }\exp\Big(\int_{\U}\underbrace{\frac{\zeta^k-1-i k \Im(\zeta)}{1-\Re(\zeta)}}_{=-k^2\text{ if }\zeta=1}\diff \nu (\zeta)\Big).$$ 
Unfortunately, the pair $(e^{i\alpha},\nu)$ is not unique, in contrast to what~\cite{Schatte1983} can suggest at first reading (see the end of the current section). We say that $(e^{i\alpha},\nu)$ is \textit{a $\circledast$-characteristic pair} for~$\mu$. Conversely, for all pair $(\omega,\nu)$ such that $\omega\in \U$ and $\nu$ is a finite measure on $\U$, there exists a unique {$\circledast$-infinitely} divisible measure $\mu$ which admits $(\omega,\nu)$ as a $\circledast$-characteristic pair.

Similarly to the additive case, we introduce now the characteristic triplet. Let $\mu\in \mathcal{ID}(\U,\circledast)\cap \Ms$ and let $(\omega,\nu)$ be a $\circledast$-characteristic pair for $\mu$. We set \begin{equation}b=2 \nu (\{1\})\ \text{ and }\ \upsilon(\diff \zeta)=\frac{1}{1-\Re \zeta}\cdot 1_{\mathbb{U}\setminus \{1\}}(\zeta) \nu (\diff \zeta).\end{equation}We have, for all $k\in \Z$,  $$\widehat{\mu}(k)=\omega^k\exp\Big(-\frac{1}{2}bk^2+\int_{\U}(\zeta^k-1-i k \Im(\zeta))\diff \upsilon (\zeta)\Big).$$ 
We say that $(\omega,b,\upsilon)$ is \textit{a $\circledast$-characteristic triplet} for $\mu$. Let us remark that $\upsilon(\{1\})=0$ and $\int_{\U}(1+ \Re(\zeta))\diff \upsilon (\zeta)<+\infty$. Such a measure is called a \textit{Lévy measure} on $\U$. As expected, for all $(\omega,b,\upsilon)$ with $\omega\in \U$, $b\geq 0$ and $\upsilon$ a Lévy measure on $\U$, there exists a unique {$\circledast$-infinitely} divisible probability measure such that $(\omega,b,\upsilon)$ is a $\circledast$-characteristic triplet for $\mu$.
Moreover, for all $\mu_1$ and $\mu_2$ be two {$\circledast$-infinitely} divisible measures with $\circledast$-characteristic triplets $(\omega_1,b_1,\upsilon_1)$ and $(\omega_2,b_2,\upsilon_2)$, we see thanks to~\eqref{ch} that $\mu_1\circledast\mu_2\in \mathcal{ID}(\U,\circledast)\cap \Ms$ with $\circledast$-characteristic triplet $(\omega_1\omega_2,b_1+b_2,\upsilon_1+\upsilon_2)$.

To sum up the previous discussion, for all $\mu \in \mathcal{ID}(\U,\circledast)$,  there exist $m\in \N\cup\{\infty\}$, $\omega\in \U$ and $\nu$ a finite measure on $\U$ such that, for all $k\in \Z$,  $$\widehat{\mu}(k)=\widehat{\lambda_m}(k)\cdot \omega^k\exp\Big(\int_{\U}\underbrace{\frac{\zeta^k-1-i k \Im(\zeta)}{1-\Re(\zeta)}}_{=-k^2\text{ if }\zeta=1}\diff \nu (\zeta)\Big).$$ 
\begin{framed}
\begin{example}Here again, we can distinguish three classes of {$\circledast$-infinitely} divisible measures:\label{exmc}
\begin{enumerate}
\item For any constant $\omega\in \U$, $(\omega,0,0)$ is a $\circledast$-characteristic triplet of the Dirac distribution $\delta_\omega\in \mathcal{ID}(\U,\circledast)$;
\item For any constant $b>0$, the wrapped Gaussian distribution of parameter $b$ is $\e_\ast(\mathcal{N}_b)\in \mathcal{ID}(\U,\circledast)$ whose one $\ast$-characteristic triplet is $(1,b,0)$;
\item For any constant $\lambda>0 $ and any probability measure $\upsilon $ on $\U$, the compound Poisson distribution with rate $\lambda$ and jump distribution $\upsilon$ is $$\Poiss^\circledast_{\lambda,\upsilon}=e^{-\lambda}\sum_{n\in \N}\frac{\lambda^n}{n!}\upsilon^{\circledast n}\in \mathcal{ID}(\U,\circledast)$$ whose one $\circledast$-characteristic triplet is $(\exp(i\lambda \int_{\U}\Im \diff \upsilon),0,\lambda \upsilon_{|\U\setminus \{1\}})$.
\end{enumerate}
\end{example}
\end{framed}
We give now a case of {$\circledast$-infinitely} divisible measure which admits two different \mbox{$\circledast$-characteristic} pairs. Set $$\mu=e^{-\pi}\left(\frac{\cosh(\pi)+1}{2}\delta_1+\frac{\cosh(\pi)-1}{2}\delta_{-1}+\frac{\sinh(\pi)}{2}\delta_i+\frac{\sinh(\pi)}{2}\delta_{-i}\right).$$
For all $n\in \Z$, we have $\widehat{\mu}(4n)=1$, $\widehat{\mu}(4n+1)=\widehat{\mu}(4n+3)=e^{-\pi}$ and $\widehat{\mu}(4n+2)=e^{-2\pi}$. It is immediate that, for $\upsilon=\pi\delta_i$ or $\upsilon=\pi\delta_{-i}$, we have
$$\widehat{\mu}(n)=\exp\Big(\int_{\U}(\zeta^n-1-i n \Im(\zeta))\diff \upsilon (\zeta)\Big).$$
Thus, the measure $\mu$ admits $(1,0,\pi\delta_i)$ and $(1,0,\pi\delta_{-i})$ as $\circledast$-characteristic triplets. One can also see~\cite{Chistyakov2008} for others examples.

\subsection{The convolution $\boxtimes$ and the $S$-transform}\label{SSS}The free multiplicative convolution $\boxtimes$ can be described succinctly in terms of the $S$-transform. Let us explain how it works.

Let $\mu$ be a finite measure on $\U$. For all $k\in \N$, we set $m_k(\mu)=\int_\C \zeta^k\diff \mu (\zeta)$, which is finite, and we call $(m_k(\mu))_{k\in \N}$ the \textit{moments}  of $\mu$. We consider the formal power series $$M_\mu(z)=\sum_{k=0}^\infty m_k(\mu)z^k.$$
%For all $\mu$ and $\nu\in \Mt$, we define the classical multiplicative convolution of $\mu$ and $\nu$ to be the unique measure $\mu\circledast \nu\in \Mt$ such that, for all continuous and bounded complex function $f$,
%$$\int_{\U}f\ \diff (\mu \circledast\nu)=\int_{\U^2}f(\zeta\xi)\diff \mu(\zeta) \diff \nu(\xi) .$$
%As for the addictive case, the free multiplicative convolution can be defined by a concrete product of non-commutative random variables, but we prefer here to present an equivalent definition, using the $S$-transform.
Let us assume that $\mu\in \Ms$. We define $S_\mu$, the \textit{$S$-transform} of $\mu$,  to be the formal power series such that $z S_\mu(z)/(1+z)$ is the inverse under composition of $M_\mu(z)-1$. The $S$-transform is a $\boxtimes$-homomorphism (see~\cite{Bercovici1992}): for all $\mu$ and $\nu\in  \Ms$, $$S_{\mu\boxtimes\nu}=S_\mu\cdot S_\nu.$$
For all $\mu \in  \Ms$, the series $S_\mu(z)$ is convergent in a neighbourhood of $0$, and we can therefore identify $S_\mu$ with a function which is analytic in a neighborhood of zero. Conversely, a function which is analytic in a neighborhood of zero is the $S$-transform of a unique measure of $ \Ms$. Sometimes it will be convenient to use the function
$$ \Sigma_\mu(z) = S_\mu(z/(1-z))$$
which is also analytic in a neighborhood of $0$.

\subsection{Free infinite divisibility on $\U$}For the free multiplicative convolution, the existence of different proper subgroups of $\U$ does not imply the existence of different idempotent measures. Indeed, the Haar measure $\lambda$ and $\delta_1$ are the unique probability measures on $\U$ which are idempotent. Moreover, $\lambda$ is an absorbing element for $\boxtimes$ and it is the unique {$\boxtimes$-infinitely} divisible measure in $  \mathcal{ID}(\U,\boxtimes)\setminus \Ms$ according to~\cite{Bercovici1992}. Consequently, we will focus our study on $  \mathcal{ID}(\U,\boxtimes)\cap \Ms$.

Let $\mu\in  \Ms$ be a {$\boxtimes$-infinitely} divisible measure. From Theorem~6.7. of~\cite{Bercovici1992}, there exists a unique finite measure $\nu \in \Mt$ and a real $\alpha\in \R$ such that  $$\Sigma_\mu(z) =\exp\left(-i\alpha+\int_{\U}\frac{1+\zeta z}{1-\zeta z}\diff \nu (\zeta)\right).$$
The pair $(e^{i\alpha}, \nu)$ is called the $\boxtimes$-\textit{characteristic pair} for $\mu$, and, on the contrary to the classical case, it is uniquely determined by $\mu$. We have
\begin{equation}S_\mu(z) =e^{-i\alpha}\exp\left(\int_{\U}\frac{1+z+\zeta z}{1+z-\zeta z}\diff \nu (\zeta)\right).\label{Ssup}\end{equation}
We observe that, for $\zeta\neq 1$, we have
$$\frac{1+z+\zeta z}{1+z-\zeta z}=\frac{1}{1-\Re \zeta}\left(i \Im (\zeta)+\frac{1-\zeta}{1+z(1-\zeta)}\right) ,$$
which implies that, defining $\omega=e^{i\alpha}$, $b=2 \nu (\{1\})$ and  $\upsilon(\diff \zeta)=\frac{1}{1-\Re \zeta}\cdot 1_{\mathbb{U}\setminus \{1\}}(\zeta) \nu (\diff \zeta)$,
we have
\begin{equation}
S_\mu(z) =\omega^{-1}\exp\left(\frac{b}{2}+b z+\int_{\U}i \Im (\zeta)+\frac{1-\zeta}{1+z(1-\zeta)}\diff \upsilon (\zeta)\right).\label{Sinf}
\end{equation}
We will call $(\omega,b,\upsilon)$ the $\boxtimes$-\textit{characteristic triplet} for $\mu$. Conversely, for all triplet $(\omega,b,\upsilon)$ such that $\omega\in \U$, $b\in \R^+$ and $\upsilon$ is a Lévy measure on $\U$, there exists a unique {$\boxtimes$-infinitely} divisible measure $\mu$ whose $\boxtimes$-characteristic triplet is $(\omega,b,\upsilon)$. Indeed, the $S$-transform defined by~\eqref{Sinf} is analytic in a neighborhood of zero, so it is the $S$-transform of a unique measure $\mu\in  \Ms$. Moreover, if we define $$v(z)=-\Log(\omega)+\frac{b}{2}+b z+\int_{\U}i \Im (\zeta)+\frac{1-\zeta}{1+z(1-\zeta)}\diff \upsilon (\zeta)$$
using the principal value $\Log$, then we have $S_\mu(z)=\exp(v(z))$. According to Theorem~6.7. of~\cite{Bercovici1992}, $\mu$ is therefore {$\boxtimes$-infinitely} divisible.

Let $\mu_1,\mu_2\in  \Ms$ be two {$\boxtimes$-infinitely} divisible measures with respective $\boxtimes$-characteristic triplets $(\omega_1,b_1,\upsilon_1)$ and $(\omega_2,b_2,\upsilon_2)$. The measure $\mu_1\boxtimes\mu_2\in  \Ms$ is a {$\boxtimes$-infinitely} divisible measure with $\boxtimes$-characteristic triplet $(\omega_1\omega_2,b_1+b_2,\upsilon_1+\upsilon_2)$.
\begin{framed}
\begin{example}The three classes of {$\boxtimes$-infinitely} divisible measures are:\label{exmf}
\begin{enumerate}
\item For any constant $\omega\in \U$, $(\omega,0,0)$ is a $\boxtimes$-characteristic triplet of the Dirac distribution $\delta_\omega\in \mathcal{ID}(\U,\boxtimes)$;
\item For any constant $b>0$, the measure on $\U$ analogous to the Gaussian distribution law is the measure $\mathcal{B}_b\in \mathcal{ID}(\U,\circledast)$ whose $\ast$-characteristic triplet is $(1,b,0)$; it is the law of a free unitary Brownian motion at time $b$;
\item For any constant $\lambda>0 $ and any probability measure $\upsilon$ on $\U$, the free compound Poisson distribution with rate $\lambda$ and jump distribution $\upsilon$ is the measure $\Poiss^\boxtimes_{\lambda,\upsilon}\in \mathcal{ID}(\U,\boxtimes)$ whose $\boxtimes$-characteristic triplet is $(\exp(i\lambda \int_{\U}\Im \diff \upsilon),0,\lambda \upsilon_{|\U\setminus \{1\}})$.
\end{enumerate}
\end{example}
\end{framed}

\section{Homomorphisms between $\ID(\mathbb{R},\ast)$, $\ID(\U,\circledast)$, $\ID(\mathbb{R},\boxplus)$ and $\ID(\U,\boxtimes)$.}\label{secthree}

In this section, we define $\e_\boxplus$ and $\Gamma$ and prove Theorem~\ref{un}. The definitions and the commutativity of~\eqref{diag} is a routine program. The very difficulty consists in proving the first item of Theorem~\ref{un}, or equivalently Theorem~\ref{thone}. We shall do it in Section~\ref{limitt}. 

% Let us emphasis that $\Lambda$ and $\e_\ast$ are natural for at least one reason: they preserve theorems relative to infinitesimal convergence in the sense of the second item of Theorem~\ref{bercopata}. Our definition of $\e_\boxplus$ and $\Gamma$ have the same property (Theorem~\ref{thone} and~\ref{thtwo}), and the consequence is that the respective images of the Dirac distribution, the Gaussian distribution and the compound Poisson distribution by those morphisms are their analogous measures in each space $\ID(\mathbb{R},\ast)$, $\ID(\U,\circledast)$, $\ID(\mathbb{R},\boxplus)$ and $\ID(\U,\boxtimes)$ (see Section~\ref{DGP}).
\subsection{Definitions of $\e_\boxplus$ and $\Gamma$}In order to motivate the definition of $\e_\boxplus$, we start by indicating how a $\ast$-characteristic triplet is transformed by the homomorphism $\e_{\ast}$.

Let us recall that, for all measure $\mu$ on $\R$, $\e_{\ast}(\mu)$ denotes the push-forward measure of $\mu$ by the map $\e:x\mapsto e^{ix}$. Let us denote by $\e_{\ast}(\mu)_{| \U \setminus \{1\}}$ the measure induced by $\e_{\ast}(\mu)$ on $\U \setminus \{1\}$.
\begin{proposition}For all $\mu\in \ID(\mathbb{R},\ast)$ with $\ast$-characteristic triplet $(\eta, a, \rho)$,\label{exth}
\begin{equation*}
(\omega,b,\upsilon)=\left(\exp\left(i\eta+i\int_{\mathbb{R}}(\sin(x)-1_{[-1,1]}(x) x) \ \rho(\diff x)\right), a, \e_{\ast}(\rho)_{| \U \setminus \{1\}}\right)
\end{equation*}
is a $\circledast$-characteristic triplet of $\e_{\ast}(\mu)$. 
\end{proposition}
\begin{proof}First of all, the Fourier transform of a {$\ast$-infinitely} divisible measure is well-known (see \cite{Bercovici1999,Sato1999}): for all $\theta\in \R$, we have
\begin{equation*}
\int_\R e^{i\theta x}\diff\mu(x)=\exp\left(i\eta \theta -\frac{1}{2}a\theta^2+\int_\R (e^{i\theta x}-1-i\theta x 1_{[-1,1]}(x))\diff \rho(x)\right).\label{char}
\end{equation*}
%We set 
%$$(\omega,b,\upsilon)=\left(\exp\left(i\eta+i\int_{\mathbb{R}}(\sin(x)-1_{[-1,1]}(x) x) \ \rho(\diff x)\right), a, \e_{\ast}(\rho)_{| \U \setminus \{1\}}\right) .$$
Let $n\in \N$. We have
\begin{align*}
&\widehat{\e_{\ast}(\mu)}(n)=\int_\U \zeta^n \diff (\e_{\ast}(\mu))(\zeta)=\int_\R e^{inx}\diff \mu(x)\\
&\hspace{1cm} =\exp\left(i\eta n -\frac{1}{2}an^2+\int_\R (e^{in x}-1-in x 1_{[-1,1]}(x))\diff \rho(x)\right)\\
&\hspace{1cm} =\exp\left(i\eta n+in \int_{\mathbb{R}}(\sin(x)-1_{[-1,1]}(x) x) \ \rho(\diff x)-\frac{1}{2}an^2+\int_\R (e^{in x}-1-in \sin(x))\diff \rho(x)\right)\\
&\hspace{1cm} =\omega^n\exp\Big(-\frac{1}{2}bn^2+\int_{\U}(\zeta^n-1-i n \Im(\zeta))\diff \upsilon (\zeta)\Big),
\end{align*}
which proves that $(\omega,b,\upsilon)$ is a $\circledast$-characteristic triplet of $\e_{\ast}(\mu)$.
\end{proof}

We define $\e_{\boxplus}:\ID(\mathbb{R},\boxplus)\to \ID(\mathbb{U},\boxtimes)$ by analogy with the previous proposition.
\begin{definition}\label{edef}For all $\mu\in \ID(\mathbb{R},\boxplus)$ with $\boxplus$-characteristic triplet $(\eta, a, \rho)$, we define $\e_{\boxplus}(\mu)$ to be the {$\boxtimes$-infinitely} divisible measure on $\mathbb{U}$ with $\boxtimes$-characteristic triplet \begin{equation}
(\omega,b,\upsilon)=\left(\exp\left(i\eta+i\int_{\mathbb{R}}(\sin(x)-1_{[-1,1]}(x) x) \ \rho(\diff x)\right), a, \e_{\ast}(\rho)_{| \U \setminus \{1\}}\right).\label{fne}
\end{equation}
\end{definition}
\begin{proposition}
For all $\mu$ and $\nu\in \ID(\mathbb{R},\boxplus)$, we have $\e_{\boxplus}(\mu\boxplus \nu)=\e_{\boxplus}(\mu)\boxtimes \e_{\boxplus}( \nu) $.
\end{proposition}
\begin{proof}Let us denote by $(\eta_1, a_1, \rho_1)$ and $(\eta_2, a_2, \rho_2)$ the respective $\boxplus$-characteristic triplets of $\mu$ and $\nu$. The $\boxtimes$-characteristic triplet of $\mu\boxplus \nu$ is $(\eta_1+\eta_2, a_1+a_2, \rho_1+\rho_2)$. As a consequence, denoting by $(\omega_1, b_1,\upsilon_1)$ and $(\omega_2, b_2,\upsilon_2)$ the $\boxtimes$-characteristic triplets of $\e_{\boxplus}(\mu)$ and $\e_{\boxplus}(\nu)$ defined by~\eqref{fne}, $(\omega_1\omega_2, b_1+b_2,\upsilon_1+\upsilon_2)$  is the $\boxtimes$-characteristic triplet of both $\e_{\boxplus}(\mu\boxplus \nu)$ and $\e_{\boxplus}(\mu)\boxtimes \e_{\boxplus}( \nu)$.
\end{proof}
The definition of $\Gamma:\ID(\U,\boxtimes)\to \ID(\U,\circledast)$ is even simpler.
\begin{definition}For all $\mu\in \ID(\U,\boxtimes)\cap \Ms $ with characteristic triplet $(\omega,b,\upsilon)$, we define $\Gamma(\mu)$ to be the {$\circledast$-infinitely} divisible measure on $\mathbb{U}$ with characteristic triplet $(\omega,b,\upsilon)$. Moreover, for $\lambda$ being the Haar measure of $\mathbb{U}$, we set $\Gamma(\lambda)=\lambda$. \label{gammadef}
\end{definition}

\begin{proposition}
For all $\mu$ and $\nu\in \ID(\U,\boxtimes)$, we have $\Gamma(\mu\boxtimes \nu)=\Gamma(\mu)\circledast \Gamma( \nu)$ .
\end{proposition}
\begin{proof}
Let $\mu$ and $\nu\in \ID(\U,\boxtimes)$. If $\mu$ or $\nu$ is equal to $\lambda$, we have $\mu\boxtimes \nu=\lambda$. In this case, $\Gamma(\mu)$ or $\Gamma(\nu)$ is also equal to $\lambda$ and consequently, $\lambda=\Gamma(\mu\boxtimes \nu)=\Gamma(\mu)\circledast \Gamma( \nu).$

If $\mu,\nu\in \ID(\U,\boxtimes)\cap \Ms$ with respective $\boxtimes$-characteristic triplets $(\omega_1,b_1,\upsilon_1)$ and $(\omega_2,b_2,\upsilon_2)$, the measure $\mu\boxtimes\nu$ is a {$\boxtimes$-infinitely} divisible measure with $\boxtimes$-characteristic triplet $(\omega_1\omega_2,b_1+b_2,\upsilon_1+\upsilon_2)$. By consequence, $(\omega_1\omega_2,b_1+b_2,\upsilon_1+\upsilon_2)$ is a $\circledast$-characteristic triplet of both $\Gamma(\mu\boxtimes \nu)$ and $\Gamma(\mu)\circledast \Gamma( \nu)$.
\end{proof}

We can now verify the commutativity of the diagram~\eqref{diag} in the following proposition.
\begin{proposition}\label{thtwo}
We have $\Gamma\circ \e_{\boxplus}\circ \Lambda=\e_\ast$.
\end{proposition}
\begin{proof}
For all $\mu\in \ID(\mathbb{R},\ast)$ with $\ast$-characteristic triplet $(\eta, a, \rho)$,
\begin{equation*}
\left(\exp\left(i\eta+i\int_{\mathbb{R}}(\sin(x)-1_{[-1,1]}(x) x) \ \rho(\diff x)\right), a, \e_{\ast}(\rho)_{| \U \setminus \{1\}}\right)
\end{equation*}
is a $\circledast$-characteristic triplet of both $\e_{\ast}(\mu)$ and $\Gamma\circ \e_{\boxplus}\circ \Lambda(\mu)$.
\end{proof}
\begin{framed}
%\subsection{Dirac, Gaussian and Poisson distributions}
We summarize here the successive action of $\Lambda,\e_{\boxplus}, \Gamma $ and $\e_\ast$ on respectively a Dirac measure $\delta_\eta$ $(\eta\in \R)$, a Gaussian measure $\mathcal{N}_b$ $(b>0)$, and a compound Poisson distribution with rate $\lambda>0$ and jump distribution $\rho$ (Example~\ref{exac}). As expected, their images are respectively their free analogues on $\R$ (Example~\ref{exaf}), their free analogues on $\U$ (Example~\ref{exmf}), and their multiplicative analogues on $\U$ (Example~\ref{exmc}):\label{DGP}
$$\left.\begin{array}{ccccccc} & \Lambda &  & \e_\boxplus &  & \Gamma &  \\\delta_\eta & \longmapsto & \delta_\eta & \longmapsto & \delta_{e^{i\eta}} & \longmapsto & \delta_{e^{i\eta}} \\\mathcal{N}_b & \longmapsto & \mathcal{S}_b & \longmapsto & \mathcal{B}_b & \longmapsto & \e_\ast(\mathcal{N}_b) \\\Poiss^\ast_{\lambda,\rho} & \longmapsto & \Poiss^\boxplus_{\lambda,\rho} & \longmapsto & \Poiss^\boxtimes_{\lambda,\e_\ast(\rho)} & \longmapsto & \Poiss^\circledast_{\lambda,\e_\ast(\rho)}.\end{array}\right. $$
\end{framed}
\subsection{A limit theorem}The definition of $\Gamma$ is justified, if needed, by the following result of Chistyakov and Götze.
\label{limitt}

\begin{theorem}[\cite{Chistyakov2008}]For all $\mu\in \ID(\U,\boxtimes)\cap \Ms$, all natural numbers $k_1<k_2<\cdots$  and  all sequence $(\mu_n)_{n\in \N^*}$ of probability measures in $\Ms$ such that the measures $\mu_n^{\boxtimes k_n}$ converge weakly to $\mu$, the measures $\mu_n^{\circledast k_n}$ converge weakly to $\Gamma(\mu)$.
\end{theorem}

The rest of this section is devoted to proving an analogous theorem for $\e_{\boxplus}$. This goal is achieved in Theorem~\ref{thone}.
Let us start by a key result, interesting in its own, about the convergence towards a \mbox{$\boxtimes$-infinitely} divisible measure. The following proposition is the analogue of Theorem~\ref{critc} for the convolution~$\boxtimes$. We refer the reader to Theorem 4.3 of \cite{Bercovici2008} and  Theorem 2.3 of \cite{Chistyakov2008} for other similar criterions. The major difference between these results and ours is the shift of $\mu_n$ considered: in Proposition~\ref{shift}, we consider the angular part $\omega_n=m_1(\mu_n)/|m_1(\mu_n)|$ of the mean of $\mu_n$.

For all measure $\mu_n$ on $\U$, all $\omega_n\in \U$ and all $k_n\in \N$, we denote by $k_n (1-\Re( \zeta))\diff \mu_n(\omega_n \zeta)$ the measure such that, for all bounded Borel function $f$ on $\U$,
$$\int_\U f(\zeta) k_n (1-\Re( \zeta))\diff \mu_n(\omega_n \zeta)=k_n \int_\U f(\omega_n^{-1}\zeta)  (1-\Re(\omega_n^{-1} \zeta))\diff \mu_n( \zeta).$$

\begin{proposition}\label{shift}
Let $\mu\in \ID(\U,\boxtimes)$ with $\boxtimes$-characteristic pair $(\omega, \nu)$. Let $k_1<k_2<\cdots$ be a sequence of natural numbers. Let $(\mu_n)_{n\in \N}$ be a sequence of measures in $  \Ms$ and $(\omega_n)_{n\in \N}$ a sequence of elements of $\U$ such that, for all $n\in \N$,  $\omega_n=m_1(\mu_n)/|m_1(\mu_n)|$. The following assertions are equivalent:
\begin{enumerate}
\item the measures $\underbrace{\mu_n\boxtimes\cdots \boxtimes \mu_n}_{k_n \text{ times}}$ converge weakly to $\mu$;
\item the measures  $$\diff\nu_n(x)=k_n (1-\Re( \zeta))\diff \mu_n(\omega_n \zeta)$$
converge weakly to $\nu$ and
$$\lim_{n\to \infty}\omega_n^{k_n}=\omega.$$
\end{enumerate}
\end{proposition}
In concrete cases, the second item is often easier to verify. For example, it allows us to infer that, for any constant $\lambda>0 $ and any probability measure $\upsilon$ on $\U$, the measure $\Poiss^\boxtimes_{\lambda,\upsilon}$ is the weak limit of $\left((1-\lambda/n)\delta_1+(\lambda/n)\upsilon\right)^{\boxtimes n} $ as $n$ tends to $\infty$.

We would point out the recent work~\cite{Anshelevich2013} which proves that the convergence of Theorem~\ref{shift} implies local convergences of the probability densities.

\begin{proof}The weak convergence of finite measures on $\U$ is equivalent to the convergence of the moments, or equivalently, for measures in $\Ms$, to the convergence of the $S$-transform. Thus, it suffices to prove that the following assertions are equivalent:
\begin{enumerate}
\item $\lim_{n\to \infty}S_{\mu_n^{\boxtimes k_n}}=S_\mu$;
\item $\lim_{n\to \infty}M_{\nu_n}=M_\nu$ and $\lim_{n\to \infty}\omega_n^{k_n}=\omega$
\end{enumerate}
where the convergence of formal series is the convergence of each coefficient. Let us recall the useful information about the $S$-transform: it is a $\boxtimes$-homomorphism and $z S(z)/(1+z)$ is the inverse under composition of $M(z)-1$ (see Section~\ref{SSS}).

Let $n\in \N$. We set $r_n=|m_1(\mu_n)|$, so that $m_1(\mu_n)=r_n\omega_n$. We define also $\mu_n^{\circ}\in \Ms$ such that $\diff\mu_n^{\circ}(\zeta)=\diff \mu_n(\omega_n \zeta)$. The measure $\mu_n^{\circ}$ will be the link between $\mu_n$ and $\nu_n$. Observe that $M_{\mu_n^{\circ}}(z)=M_{\mu_n}(\omega_n^{-1} z)$, which implies that $S_{\mu_n^{\circ}}(z)=\omega_n S_{\mu_n}(z)$.
% We observe that $M_{\mu_n^{\circ}}(z)=M_{\mu_n}(\omega_n^{-1} z)$, and therefore $S_{\mu_n^{\circ}}(z)=\omega_n S_{\mu_n}(z)$. 
The first step of the proof is to write $M_{\nu_n}$ with the help of $M_{\mu_n^{\circ}}$. For all $\zeta\in \U$ and $z\in \mathbb{C}$ sufficiently small, we have
%$$2\frac{1-\Re\zeta}{1-\zeta z}=\frac{(1-\zeta)(1-z)}{1-\zeta z}+1-\bar{\zeta} $$ 
$$2\frac{1-\Re\zeta}{1-\zeta z}= (z-1)\left[(1-z)\frac{\zeta}{1-\zeta z}-1\right]+1-\bar{\zeta}.$$
%implies that
%
% \begin{align*}
% \int_{\U} \frac{2}{1-\zeta z}(1-\Re( \zeta))\diff\mu_n^{\circ}( \zeta) &= \int_{\U} \frac{(1-\zeta)(1-z)}{1-\zeta z}+\zeta-1\diff\mu_n^{\circ}( \zeta),
%\end{align*}
% \begin{align*}
% -\int_{\U}k_n \frac{2\zeta z}{1-\zeta z}(1-\Re( \zeta))\diff \mu_n(\omega_n \zeta)&= -\int_{\U}k_n \frac{2\zeta \omega_n^{-1} z}{1-\omega_n^{-1} \zeta z}(1-\Re(\omega_n^{-1} \zeta))\diff \mu_n( \zeta)\\
% &=k_n \int_{\U}- \frac{1+\zeta \omega_n^{-1} z}{1-\omega_n^{-1} \zeta z}(1-\Re(\omega_n^{-1} \zeta))+(1-\Re(\omega_n^{-1} \zeta))\diff \mu_n( \zeta)\\
% &=k_n \int_{\U} \frac{(\zeta \omega_n^{-1} -1)(1-z)}{1-\omega_n^{-1} \zeta z}+1-\omega_n^{-1} \zeta\diff \mu_n(\zeta),
%\end{align*}
Integrating with respect to $\mu_n^{\circ}$, and remarking that $\int_\U \bar{\zeta}\diff\mu_n^{\circ}(\zeta)=\overline{\int_\U \zeta\diff\mu_n(\zeta)/\omega_n} =\bar{r}_n=r_n$, we deduce that \begin{equation}
\frac{2}{k_n}M_{\nu_n}= (z-1)\left[\frac{1-z}{z}(M_{\mu_n^\circ}-1)-1\right]+(1-r_n).\label{mmu}
\end{equation}

Let us suppose that $\lim_{n\to \infty}S_{\mu_n^{\boxtimes k_n}}=S_\mu$. For all $n\in \N$, $S_{\mu_n^{\boxtimes k_n}}=(S_{\mu_n})^{k_n}$. Therefore, we have $\lim_{n\to \infty}m_1(\mu_n)^{-k_n}=\lim_{n\to \infty}S_{\mu_n}^{k_n}(0)=S_{\mu}(0)$. Thanks to~\eqref{Ssup}, we know that $S_{\mu}(0)=\omega^{-1} e^{\nu(\U)}$, which implies that $ \lim_{n\to \infty} \omega_n^{k_n}=\omega$ and $\lim_{n\to \infty}r_n^{k_n}= e^{-\nu(\U)}$. It remains now to prove $\lim_{n\to \infty}M_{\nu_n}=M_\nu$.
%We observe that $M_{\mu_n^{\circ}}(z)=M_{\mu_n}(\omega_n^{-1} z)$, which implies that $S_{\mu_n^{\circ}}(z)=\omega_n S_{\mu_n}(z)$. Consequently, we have $\lim_{n\to \infty}S_{\mu_n^{\circ}}(z)^{k_n} = \omega S_\mu$. 
At this stage of the proof, we need to inverse formal series, at least asymptotically, and instead of doing it term by term, we prefer to work in a quotient algebra where the negligible terms will be forgotten.

More precisely, let $\ell^\infty$ be the algebra of bounded complex sequences. We consider the ideal $\mathcal{I} \subset \ell^\infty$ composed of sequences $x_n$ such that $\lim_{n\to \infty}k_n  x_n = 0$; in other words, sequences which are $o(1/k_n)$. We find it convenient to work in the quotient algebra $\mathcal{B}=\ell^\infty/\mathcal{I}$. For example, $\lim_{n\to \infty}k_n \log(r_n)=- \nu(\U)$ can be rewritten $\log(r_n)\cong -\frac{1}{k_n}\nu(\U)$ in $\mathcal{B}$, which implies that $r_n\cong e^{-\frac{1}{k_n}\nu(\U)}\cong 1-\frac{1}{k_n}\nu(\U)$. We will view sequences of formal series as elements of $\ell^\infty[[z]]$, and we will naturally identify $\ell^\infty[[z]]/\mathcal{I}[[z]]$ with $\mathcal{B}[[z]]$. For simplicity, equality in $\mathcal{B}$ or $\mathcal{B}[[z]]$ will be denoted by the symbol $\cong$.
% Conversely, if $  \omega_n^{k_n}\rightarrow\omega$ and $S_{\mu_n^{\circ}}(z)\rightarrow  \omega S_\mu$ as $n$ tends to $\infty$, it follows that $S_{\mu_n}^{k_n}\rightarrow S_\mu$ as $n$ tends to $\infty$. Finally, it suffices to prove that under the condition that $  \omega_n^{k_n}\rightarrow\omega$ as $n$ tends to $\infty$, $k_n (1-\Re(z))\diff \mu_n^{\circ}(z) \xrightarrow{(w)} \diff \nu$ as $n$ tends to $\infty$ if and only if $S_{\mu_n^{\circ}}^n\rightarrow \omega S_\mu$ as $n$ tends to $\infty$. For the end of the proof, we will assume that $\omega_n^{k_n}\rightarrow\omega$ as $n$ tends to $\infty$.

Let us denote by $u(z)$ the series
$$u(z)=\int_{\U}\frac{1+z+\zeta z}{1+z-\zeta z}\diff \nu (\zeta)-\nu(\U).$$
Thanks to~\eqref{Ssup}, we have $m_1(\mu) S_\mu(z)=S_\mu(z)/S_\mu(0)=\exp(u(z))$, from which we deduce that $u(z)$ is equal to $\log (m_1(\mu) S_\mu)$, that is to say the series given by 
$ -\sum_{k=1}^{\infty}\frac{1}{k}\left(1-m_1(\mu) S_\mu(z)\right)^k.$
The formal series $k_n\log (m_1(\mu_n)S_{\mu_n})=\log ((m_1(\mu_n)S_{\mu_n})^{k_n})$ tends to $ \log (m_1(\mu) S_\mu)=u(z)$ as $n$ tends to infinity. Consequently, $\log (m_1(\mu_n)S_{\mu_n})\cong \frac{1}{k_n} u(z)$. Thus, we have

 \begin{align*}
S_{\mu_n^{\circ}}(z)={r_n^{-1} \cdot  m_1(\mu_n) S_{\mu_n}(z)}&\cong e^{\frac{1}{k_n}\nu(\U)} \cdot 
%e^{\frac{1}{k_n} u(z)}\\ 
\exp\left(\frac{1}{k_n} u(z)\right)\\ 
&\cong \left(1+\frac{1}{k_n}\nu(\U)\right)\left(1+\frac{1}{k_n} u(z)\right)\\ 
%\end{align*} \begin{align*}
&\cong 1+\frac{1}{k_n}\nu(\U)+ \frac{1}{k_n} u(z).\end{align*}
%We have $|m_1(\mu_n)|^{k_n}\rightarrow |m_1(\mu)|$ which implies that $|m_1(\mu_n)|\rightarrow 1$ as $n$ tends to $\infty$. Thus, 
With this new expression of $S_{\mu_n^{\circ}}$, it is easy to check that the inverse under composition in $\mathcal{B}[[z]]$  of $zS_{\mu_n^{\circ}}/(z+1)\cong z(1+\frac{1}{k_n}\nu(\U)+\frac{1}{k_n}u(z))/(z+1) $ is exactly 
%$ \frac{z}{(1-z)}\left[1-\frac{1}{k_n}\frac{1}{1-z}\left(\nu(\U)+u\left(\frac{z}{1-z}\right)\right)\right]$. Consequently,
 $$\frac{z}{(1-z)}\left[1-\frac{1}{k_n}\frac{1}{1-z}\left(\nu(\U)+u\left(\frac{z}{1-z}\right)\right)\right],$$
 which is then the expression of $M_{\mu_n^{\circ}}(z)-1$ in $\mathcal{B}[[z]]$.
%or $(1-z)\left[(z-1)M_{\mu_n^{\circ}}+1\right]\cong \frac{1}{k_n}zu\left(\frac{z}{1-z}\right)$. Using $1-r_n\cong \frac{1}{k_n}\nu(\U)$, 
Replacing $M_{\mu_n^{\circ}}(z)-1$ by this expression in~\eqref{mmu} yields
$$ \frac{2}{k_n} M_{\nu_n}\cong \frac{1}{k_n}\left(u\left(\frac{z}{1-z}\right)+2\nu(\U)\right).$$
But we have $u\left(z/(1-z)\right)+2\nu(\U)=\int_{\U}\frac{2}{1-\zeta z}\diff \nu (\zeta)=2 M_{\nu}$,  and finally, $ \frac{1}{k_n} M_{\nu_n}\cong \frac{1}{k_n} M_{\nu}$, or equivalently, $\lim_{n\to \infty}M_{\nu_n}=M_\nu$.

Conversely, if we suppose that $\lim_{n\to \infty}M_{\nu_n}=M_\nu$ and $\lim_{n\to \infty}\omega_n^{k_n}=\omega$,  we can basically retrace our steps in order to arrive at $\lim_{n\to \infty}S_{\mu_n^{\boxtimes k_n}}=S_\mu$.
\end{proof}

\begin{corollary}Let $\mu\in \ID(\U,\boxtimes)$ with $\boxtimes$-characteristic triplet $(\omega,b, \upsilon)$. Let $k_1<k_2<\cdots$ be a sequence of natural numbers. Let $(\mu_n)_{n\in \N}$ be a sequence of measures in $  \Ms$ and $(\omega_n)_{n\in \N}$ be such that, for all $n\in \N$, $\omega_n=m_1(\mu_n)/|m_1(\mu_n)|$. The following assertions are equivalent:\label{colshift}
\begin{enumerate}
\item the measures $\underbrace{\mu_n\boxtimes\cdots \boxtimes \mu_n}_{k_n \text{ times}}$ converge weakly to $\mu$;
\item $\lim_{n\to \infty} \omega_n^{k_n}=\omega$ and the measures  $\diff\nu_n(x)=k_n (1-\Re( \zeta))\diff \mu_n(\omega_n \zeta)$
converge weakly to \mbox{$(1-\Re( \zeta))\diff \upsilon(\zeta)+\frac{b}{2}\delta_1$.}

\end{enumerate}

\end{corollary}

We are now ready to prove the first main theorem of this paper.

\begin{theorem}\label{thone}
For all $\mu\in \ID(\R,\boxplus)$, all natural numbers $k_1<k_2<\cdots$  and all sequence $(\mu_n)_{n\in \N^*}$ of probability measures on $\R$ such that the measures $\mu_n^{\boxplus k_n}$ converge weakly to $\mu$, the measures $\e_{\ast}(\mu_n)^{\boxtimes k_n}$ converge weakly to $\e_{\boxplus}(\mu)$.

\end{theorem}
Let us derive right now some consequences of this theorem. It allows us to transfer limit theorems about $\boxplus$ into limit theorem about $\boxtimes$. For example, for all $b>0$, the semi-circular measure is such that $\mathcal{S}_{b/n}^{\boxplus n}=\mathcal{S}_{b}$. We deduce that $\mathcal{B}_b=\e_{\boxplus}(\mathcal{S}_{b})$, which is the law of a free unitary Brownian motion at time $b$, is the weak limit of the measures $\e_{\ast}(\mathcal{S}_{b/n})^{\boxtimes n}$. Using Theorem~\ref{bercopata}, we know also that the measures $\mathcal{N}_{b/n}^{\boxtimes n}$ converge weakly to $\mathcal{S}_{b}$. By consequence, $\mathcal{B}_b$  is also the weak limit of $\e_{\ast}(\mathcal{N}_{b/n})^{\boxtimes n}$ as $n$ tends to $\infty$.

\begin{proof}Let $(\eta, a, \rho)$ be the $\boxplus$-characteristic triplet of $\mu$, and $(\omega,b, \upsilon)$ be the $\boxtimes$-characteristic triplet of $\e_{\boxplus}(\mu)$ given by~\eqref{fne}. In order to use Corollary~\ref{colshift}, we first prove that $\e_{\ast}(\mu_n)\in  \Ms$ for $n$ sufficiently large.

Because $e^{ix}-1=ix1_{[-1,1]}(x)+(e^{ix}-1-ix1_{[-1,1]}(x))$, we have
$$
\left(\int_\R e^{ix}\diff \mu_n(x)-1\right)=i\int_\R x1_{[-1,1]}(x)\diff \mu_n(x)+\int_\R(e^{ix}-1-ix1_{[-1,1]}(x))\diff \mu_n(x).
$$
We use Lemma~\ref{lemconv}, and the fact that $e^{ix}-1-ix1_{[-1,1]}(x)\sim_{x\to 0}-\frac{1}{2}x^2,$ to deduce that
\begin{equation}
\lim_{n\to \infty}k_n \left(\int_\R e^{ix}\diff \mu_n(x)-1\right)= i\eta-\frac{a}{2}+\int (e^{ix}-1-1_{[-1,1]}(x) i x)\rho(\diff x).\label{auxk}
\end{equation}
Consequently, $m_1(\e_{\ast}(\mu_n))=\int_\R e^{ix}\diff \mu_n(x)$ tends to $1$ as $n$ tends to $\infty$, and $\e_{\ast}(\mu_n)\in   \Ms$ for $n$ sufficiently large. Without loss of generality, we assume that $\e_{\ast}(\mu_n)$ is in $  \Ms$ for all $n\in \N$. We set $(r_n,\omega_n)_{n\in \N}$ the sequence of $\R_+\times \U$ such that, for all $n\in \N$, we have $m_1(\mu_n)=r_n \omega_n$. Thanks to Corollary~\ref{colshift}, it suffices to prove that  $\lim_{n\to \infty}\omega_n^{k_n}=\omega$ and to prove that the measure $k_n (1-\Re( \zeta))\diff (\e_{\ast}(\mu_n))(\omega_n \zeta) $ converge weakly to $ (1-\Re( \zeta))\diff \upsilon+\frac{b}{2}\delta_1$ to conclude.

From~\eqref{auxk}, we deduce that
$$\lim_{n\to \infty}r_n^{k_n}\omega_n^{k_n}=\lim_{n\to \infty}\left(\int_\R e^{ix}\diff \mu_n(x)\right)^{k_n}=\exp\left( i\eta-\frac{a}{2}+\int (e^{ix}-1-1_{[-1,1]}(x) i x)\rho(\diff x)\right),$$
and this result can be split into
$$\lim_{n\to \infty}r_n^{k_n}=\exp\left( -\frac{a}{2}+\int (\cos(x)-1)\rho(\diff x)\right)$$and
$$\lim_{n\to \infty}\omega_n^{k_n}=\exp\left( i\eta+i \int (\sin(x)-1_{[-1,1]}(x) x)\rho(\diff x)\right)=\omega.$$
Using the real logarithm, we deduce that,  as $n$ tends to $\infty$,
\begin{equation}r_n^{-1}=1+\frac{1}{k_n}\left( \frac{a}{2}-\int (\cos(x)-1)\rho(\diff x)\right)+o\left(\frac{1}{k_n}\right).\label{auxr}\end{equation}
Using $\omega_n=r_n^{-1}\int_\R e^{ix}\diff \mu_n(x)$, \eqref{auxk} and~\eqref{auxr}, it follows that, as $n$ tends to $\infty$,
\begin{equation}\omega_n=1+\frac{i}{k_n}\left(  \eta+ \int (\sin(x)-1_{[-1,1]}(x) x)\rho(\diff x)\right)+o\left(\frac{1}{k_n}\right)\text{ . }\label{auxo}\end{equation}
In order to prove that the measures $k_n (1-\Re( \zeta))\diff \e_{\ast}\mu_n(\omega_n \zeta)$ converge weakly to $  (1-\Re( \zeta))\diff \upsilon+\frac{b}{2}\delta_1$, we shall use the method of moments and prove that, for all $m\in \N$,
$$\lim_{n\to \infty}k_n \int_\U\zeta^m(1-\Re( \zeta))\diff (\e_{\ast}(\mu_n))(\omega_n \zeta) = \int_\U\zeta^m(1-\Re( \zeta))\diff \upsilon( \zeta)+\frac{b}{2}.$$
Let $n\in \N$. We have
\begin{align*}
k_n \int_\U\zeta^m(1-\Re( \zeta))\diff (\e_{\ast}(\mu_n))(\omega_n \zeta) =&k_n \int_\U\omega_n^{-m}\zeta^m(1-\Re(\omega_n^{-1} \zeta))\diff (\e_{\ast}(\mu_n))(\zeta) \\
=&k_n \omega_n^{-m}\int_\R e^{imx}(1-\Re(\omega_n^{-1}e^{ix}))\diff \mu_n(x)\\
=&k_n \omega_n^{-m}\int_\R e^{imx}(1-\Re(\omega_n)\cos(x)-\Im(\omega_n)\sin(x))\diff \mu_n(x).
\end{align*}
Let us decompose the integral under study into four terms:\begin{align*}
k_n \int_\U\zeta^m(1-\Re( \zeta))\diff (\e_{\ast}(\mu_n))(\omega_n \zeta) =&k_n \omega_n^{-m}\Re(\omega_n)\int_\R e^{inx}(1-\cos(x))\diff \mu_n(x)\\
&+k_n (1-\Re(\omega_n))\omega_n^{-m}\int_\R e^{imx}\diff \mu_n(x)\\
&+k_n \Im(\omega_n)\omega_n^{-m}\int_\R(x1_{[-1,1]}(x)- e^{imx}\sin(x))\diff \mu_n(x)\\
&-k_n \Im(\omega_n)\omega_n^{-m}\int_\R x1_{[-1,1]}(x)\diff \mu_n(x).
%=& I+I\!I+I\!I\!I+IV
%\rightarrow&\int_\R e^{inx}(1-\cos(x))\diff \rho( x)+\frac{b}{2}=\int_\U\zeta^n(1-\Re( \zeta))\diff \upsilon( \zeta)+\frac{b}{2}.
\end{align*}
Thanks to Lemma~\ref{lemconv}, and because $\lim_{n\to \infty}\omega_n=1$, we know the limit of the first term:
\begin{align*}\lim_{n\to \infty}k_n \omega_n^{-m}\Re(\omega_n)\int_\R e^{imx}(1-\cos(x))\diff \mu_n(x) &=\int_\R e^{imx}(1-\cos(x))\diff \rho( x)+\frac{b}{2}\\&=\int_\U\zeta^n(1-\Re( \zeta))\diff \upsilon( \zeta)+\frac{b}{2}.\end{align*}
The three others terms tend to $0$. Indeed, \eqref{auxo} implies that $k_n (1-\Re(\omega_n))=o\left(1/k_n\right)$ and $ \Im(\omega_n)=O(1/k_n)$ when $n$ tends to $\infty$. We know that $\omega_n^{-m}=O(1)$ and $\int_\R e^{imx}\diff \mu_n(x)=O(1)$ when $n$ tends to $\infty$. Finally, Lemma~\ref{lemconv} tells us that $\int_\R(x1_{[-1,1]}(x)- e^{imx}\sin(x))\diff \mu_n(x)=O(1/k_n)$ and $\int_\R x1_{[-1,1]}(x)\diff \mu_n(x)=O(1/k_n)$ as $n$ tends to $\infty$. Thus,
$$k_n (1-\Re(\omega_n))\omega_n^{-m}\int_\R e^{imx}\diff \mu_n(x),$$
$$k_n \Im(\omega_n)\omega_n^{-m}\int_\R(x1_{[-1,1]}(x)- e^{imx}\sin(x))\diff \mu_n(x)$$ and $$
-k_n \Im(\omega_n)\omega_n^{-m}\int_\R x1_{[-1,1]}(x)\diff \mu_n(x)$$ are $o(1)$ as $n$ tends to $\infty$, and the result follows.
\end{proof}

\section{Free log-cumulants}\label{flc}
We are at the beginning of the second part of the paper, the aim of which is to prove Theorem~\ref{convwe} and Theorem~\ref{whole}. This goal is achieved in Section~\ref{secsix}. While Section~\ref{secfour} and Section~\ref{secfive} investigates the distributions of certain classes of random matrices, the current section is devoted to establish Proposition~\ref{taulogcumdeux} which is the result of free probability needed for the asymptotic  theorems proved in the last section of the paper. As a consequence, Section~\ref{flc} can be read independently of Section~\ref{secfour} and Section~\ref{secfive}.

Mastnak and Nica explain in~\cite{Mastnak2010} that, in order to treat the multidimensional free multiplicative convolution, it is preferable to work with a logarithmic version of the $S$-transform. This leads to a sequence of coefficients which in~\cite{Cebron2013} are called  the free log-cumulants. 
In this section, we use the theory of free log-cumulants to establish Proposition~\ref{taulogcumdeux} which links in an explicit formula the moments of a {$\boxtimes$-infinitely} divisible measure to its  {$\boxtimes$-characteristic} triplet. We start by stating Proposition~\ref{taulogcumdeux}, after which we introduce the free log-cumulants, which will be used only in the proof of Proposition~\ref{taulogcumdeux}.

\subsection{Moments of a {$\boxtimes$-infinitely} divisible measure}Proposition~\ref{taulogcumdeux} involves combinatorics on the symmetric group $\mathfrak{S}_n$. We first present the poset structure of $\mathfrak{S}_n$.

Let $n\in \N^*$. Let $\mathfrak{S}_n$ be the group of permutations of $\left\{1, \ldots , n\right\}$. For all permutation $\sigma\in\mathfrak{S}_n$, we denote by $\ell(\sigma)$ the numbers of cycles of $\sigma$ and we set $|\sigma|=n-\ell(\sigma)$. The minimal number of transpositions required to write $\sigma$ is $|\sigma|$ and we have $|\sigma|=0$ if and only if $\sigma$ is the identity $1_{\mathfrak{S}_n}$. We define a distance on $\mathfrak{S}_n$ by $d(\sigma_1,\sigma_2)=|\sigma^{-1}_1\sigma_2|$. The set $\mathfrak{S}_n$ can be endowed with a partial order by the relation $\sigma_1\preceq\sigma_2$ if $d(1_{\mathfrak{S}_n},\sigma_1)+ d(\sigma_1,\sigma_2)=d(1_{\mathfrak{S}_n},\sigma_2)$, or similarly if $\sigma_1$ is on a geodesic between $1_{\mathfrak{S}_n}$ and $\sigma_2$. The minimal element of $\mathfrak{S}_n$ is thus $1_{\mathfrak{S}_n}$.

For all $\sigma\in \mathfrak{S}_n$, we denote by $[1_{\mathfrak{S}_n},\sigma]$ the segment between $1_{\mathfrak{S}_n}$ and the $\sigma$, that is, the set $\{\pi\in \mathfrak{S}_n:\pi\preceq \sigma\}$. It is a lattice with respect to the partial order. A $(l+1)$-tuple $\Gamma = (\sigma_0,\ldots,\sigma_l)$ of $[1_{\mathfrak{S}_n},\sigma]$ such that $$ \sigma_0 \prec \sigma_1 \prec \cdots \prec \sigma_l\preceq \sigma$$ is called a \textit{simple chain} if and only if, for all $1\leq i\leq l$, $\sigma_{i-1}^{-1}\sigma_i$ is a non-trivial cycle. The length $k$ of a $k$-cycle $c$ will be denoted by $\sharp c$. We are now ready to state the main result of this section.\label{dist}

\begin{proposition}Let $\mu\in \ID(\U,\boxtimes)$ with $\boxtimes$-characteristic triplet $(\omega, b,\upsilon)$. For all $n\in \mathbb{N}^*$ and all $\sigma \in \mathfrak{S}_n$, we have\label{taulogcumdeux}
\begin{equation*}
\prod_{c\textup{ cycle of }\sigma}m_{\sharp c}(\mu)=e^{n L\kappa_1(\mu)}\cdot\displaystyle\sum_{\substack{\Gamma \textup{ simple chain in }[1_{\mathfrak{S}_n},\sigma]\\ \Gamma=(\sigma_0,\ldots,\sigma_{|\Gamma|}),\sigma_{|\Gamma|}=\sigma}} \frac{1}{|\Gamma| !}\displaystyle\prod_{i=1}^{|\Gamma|} L\kappa_{d(\sigma_i,\sigma_{i-1})+1}\left(\mu \right),
\end{equation*}
where
\begin{enumerate}
\item $L\kappa_1(\mu)=\Log(\omega)-b/2+\int_\U \left( \Re(\zeta)-1\right) \diff \upsilon(\zeta)$,
\item $L\kappa_2(\mu)=-b+\int_\U (\zeta -1)^2\diff \upsilon(\zeta)$
\item  and $L\kappa_n(\mu)=\int_\U (\zeta -1)^n\diff \upsilon(\zeta)$ for all $n\geq3$.
\end{enumerate} 
\end{proposition}

The proof of Proposition~\ref{taulogcumdeux} requires the notion of free log-cumulants and we postpone it until Section~\ref{proofprop}. In the mean time, we review the properties of the free log-cumulants that we shall use.

\subsection{The non-crossing partitions}The definition of the free log-cumulants involves combinatorial formulae which are related to non-crossing partitions.  We describe here the poset structure of the set of non-crossing partitions $NC(n)$, and we shall see that it is intimately linked to the poset structure of $\mathfrak{S}_n$.

%The condition for a $l$-tuple $\Gamma = (\sigma_0,\ldots,\sigma_l)$ of $[1,\sigma]$ such that $1\preceq \sigma_0 \preceq \sigma_1 \preceq \cdots \preceq \sigma_l\preceq \sigma$ to be a simple chain is that, for all $1\leq i\leq l$, $\sigma_{i-1}^{-1}\sigma_i$ is a non trivial cycle. 

A partition of the set $\left\{1, \ldots , n\right\}$ is said to have a crossing if there exist $1\leq i < j < k < l\leq n$, such that $i$ and $k$ belong to some block of the partition and $j$ and $l$ belong to another block. If a partition has no crossings, it is called non-crossing. The set of all non-crossing partitions of $\left\{1, \ldots , n\right\}$ is denoted by $NC(n)$. It is a lattice with respect to the relation of fineness defined as follows: for all $\pi_1$ and $\pi_2\in NC(n)$, we declare that $\pi_1 \preceq \pi_2$ if every block of $\pi_1$ is contained in a block of $\pi_2$. We denote respectively by $0_n$ and $1_n$ the minimal element  $\left\{\{1\}, \ldots , \{n\}\right\}$ of $NC (n)$, and the maximal element  $\left\{\{1, \ldots , n\}\right\}$ of $NC (n)$.

In~\cite{Biane1997b}, Biane describes an isomorphism between the posets $NC(n)$ and $[1_{\mathfrak{S}_n},(1\cdots n)]\subset \mathfrak{S}_n$. It consists simply in defining, from every partition $\pi \in NC(n)$, the permutation $\sigma_\pi$ which is the product, over all blocks $\{i_1<\cdots <i_k\}$ of $\pi$, of the $k$-cycle $(i_1\cdots i_k)$. In other words, take the cycles of $\sigma_\pi$ to be the blocks of $\pi$ with the cyclic order induced by the natural order of $\left\{1, \ldots , n\right\}$. Note that $\sigma_{0_n}=1_{\mathfrak{S}_n}$ and $\sigma_{1_n}=(1\cdots n)$.
\begin{lemma}The function $\pi\mapsto \sigma_\pi$ is a poset isomorphism between  $NC(n)$ and  $[1_{\mathfrak{S}_n},(1\cdots n)]$.\label{Bia}
\end{lemma}
Let $\pi \in NC(n)$. It is immediate that the map $\sigma\mapsto \sigma^{-1}\sigma_{\pi}$ is an order-reversing bijection of $[1_{\mathfrak{S}_n},\sigma_\pi]$. The corresponding decreasing bijection $K_\pi$ of $\{\pi' \in NC(n): \pi' \preceq \pi \}$ is called the Kreweras complementation map with respect to $\pi$. If $\pi=1_n$, we set $K(\sigma)=K_{1_n}(\sigma)$.\label{Kreweras}

Let $n\in \mathbb{N}$. A chain in the lattice $NC(n)$ is a $(l+1)$-tuple of the form $\Gamma = (\pi_0,\ldots,\pi_l)$ with $\pi_0,\ldots,\pi_l\in NC(n) $ such that $ \pi_0 \prec \pi_1 \prec \cdots \prec \pi_l$ (notice that we do not impose $\pi_0=0_n$ nor $\pi_l=1_n$, unlike in~\cite{Mastnak2010}). The positive integer $l$ appearing is called the length of the chain, and is denoted by $|\Gamma|$. If, for all $1\leq i\leq l$, $K_{\pi_i}(\pi_{i-1})$ has exactly one block which has more than two elements, we say that $\Gamma$ is a simple chain in $NC(n)$.
%We define similarly the multi-chains and the chains of the lattice $[1,\sigma]$. 
This way, we have an one-to-one correspondence between simple chains in $NC(n)$ and simple chains in $[1_{\mathfrak{S}_n},(1\cdots n)]$ via the isomorphism of Lemma~\ref{Bia}.

\subsection{Free log-cumulants}

Let $\mu\in  \Ms$. We denote by $W_\mu(z)$ the inverse under composition of $zM_\mu(z)$, and we denote by $C_\mu(z)$ the formal power series $M_\mu(W_\mu(z))$. The coefficients $(\kappa_k(\mu))_{k\in \mathbb{N}^*}$ of $$C_\mu(z)=1+\sum_{k=1}^\infty \kappa_k(\mu)z^k$$ are known as the \textit{free cumulants} of $\mu$. Let $\pi \in NC(n)$. We set $$\kappa\left[\pi\right](\mu)=\prod_{B\text{ block of }\pi}\kappa_{|B|}(\mu).$$
% Let us suppose first that $m_1(\mu)=1$. 
For all $n\geq 2$, we set
$$L\kappa_n(\mu)=m_1(\mu)^{-n}\displaystyle\sum_{\substack{\Gamma \text{ chain in }NC(n)\\ \Gamma=(\pi_0,\ldots,\pi_{|\Gamma|})\\\pi_0=0_n, \pi_{|\Gamma|}=1_n}} \frac{(-1)^{1+|\Gamma|}}{|\Gamma|} \displaystyle\prod_{i=1}^{|\Gamma|} \kappa\left[K_{\pi_i}(\pi_{i-1})\right](\mu).$$
We shall call the coefficients $(L\kappa_k(\mu))_{n\leq 2}$ the \textit{free log-cumulants} of $\mu$. We define also the \mbox{\textit{$LS$-transform}} of $\mu$ by
 $$LS_\mu(z)=\sum_{n=2}^\infty L\kappa_n(\mu)z^n.$$
 Let us define also $L\kappa_1(\mu)$, or $L\kappa(\mu)$, the free log-cumulant of order $1$ of $\mu$, by $\Log(m_1(\mu))$, where $\Log$ is the principal logarithm.

%(note that in~\cite{}, it would be denoted $LS_A(z)$ where $A$ is a random variable with distribution $\mu$). 
%In the general case, for all $n\geq 2$, we set
%$$L\kappa_n(\mu)=m_1(\mu)^{-n}L\kappa\left(\mu\right),$$
%and we call it the free log-cumulant of order $n$ of $\mu$. 

\begin{remark}From Proposition~4.5 of~\cite{Mastnak2010}, we see that this definition of $LS_\mu$ extends the definition of the $LS$-transform of $\mu$ given by Definition~1.4 of~\cite{Mastnak2010} in the case $m_1(\mu)\neq1$. The definition of the free log-cumulants $(L\kappa_n(\mu))_{n\in \N^*}$ follows~\cite{Cebron2013}, but we observe that $L\kappa_n(\mu)$ would be denoted by $L\kappa_n(A)$ in~\cite{Cebron2013},  where $A$  would be a random variable whose law is $\mu$.
\end{remark}
As the free cumulants linearise $\boxplus$, the free log-cumulants linearise $\boxtimes$.

\begin{proposition}[Corollary~1.5 of~\cite{Mastnak2010}, Proposition~2.11 of~\cite{Cebron2013})]For all $\mu,\nu\in \Ms$, we have $L\kappa_1(\mu\boxtimes \nu)\equiv L\kappa_1(\mu)+L\kappa_1(\nu) \pmod{2i\pi}$ and, for all $n\geq 2$,
$$L\kappa_n(\mu\boxtimes \nu)=L\kappa_n(\mu)+L\kappa_n(\nu) .$$
\end{proposition}

For concrete calculations, one would prefer to have an analytical description of the free log-cumulants.
We have $S_\mu(0)=1/m_1(\mu)$ and by consequence, we can define the formal logarithm of $m_1(\mu)\cdot S_\mu$ as the formal series $\log (m_1(\mu)\cdot S_\mu)=-\sum_{n=1}^\infty \frac{1}{n}(1-m_1(\mu)S_\mu(z))^n$.

\begin{proposition}[Corollary~6.12 of~\cite{Mastnak2010}]
Let $\mu\in  \Ms$. We have $$LS_\mu(z)=-z\log (m_1(\mu)\cdot S_\mu(z)).$$
\end{proposition}

%Thanks to Corollary~6.12 of~\cite{Mastnak2010}, we know that $$LS_\mu(z)=-z\log (m_1(\mu)\cdot S_\mu(z)).$$
\begin{remark}Technically, Corollary~6.12 of~\cite{Mastnak2010} only deals with measures, or more precisely linear functionals on $\C[X]$, such that $m_1(\mu)=1$. One can adapt the proof presented in~\cite{Mastnak2010}. Alternatively, argue as follows. From a measure $\mu\in \Ms$, we can define ${\varphi_\mu:\C[X]\to\C}$ such that $\varphi_\mu(X^k)=m_1(\mu)^{-k}m_k(\mu)$. Then, we observe that ${S_{\varphi_\mu}=m_1(\mu)\cdot S_\mu(z)}$ and $LS_{\varphi_\mu}=LS_\mu$. As a consequence,
$LS_\mu(z)=LS_{\varphi_\mu}=-z\log (S_{\varphi_\mu})=-z\log (m_1(\mu)\cdot S_\mu(z)).$
\end{remark}

Let $\pi \in NC(n)$ be such that $\pi$ has exactly one block which has at least two elements. Let $ \{j_1,\ldots ,j_N\}$ be this block of $\pi$, with $j_1 < \ldots < j_N$. Let us denote by $L\kappa\left[\pi\right](\mu)$ the free log-cumulant $L\kappa_N(\mu)$.

\begin{proposition}[Corollary~2.9 of~\cite{Cebron2013}]
Let $\mu\in  \Ms$ and $n\in \mathbb{N}^*$. We have
\begin{equation}
m_n(\mu)=e^{n L\kappa_1(\mu)}\cdot\displaystyle\sum_{\substack{\Gamma \textup{ simple chain in }NC(n)\\ \Gamma=(\pi_0,\ldots,\pi_{|\Gamma|}),\pi_0=0_n}} \frac{1}{|\Gamma| !} \displaystyle\prod_{i=1}^{|\Gamma|} L\kappa\Big[K_{\pi_i}(\pi_{i-1})\Big]\left(\mu \right).\label{taulogcum}
\end{equation}
\end{proposition}

%\begin{corollary}
%Let $(\mu_n)_{n\in \N}$ be a sequence of measure in $ \Mt\cap \Ms$ and $\mu\in  \Mt\cap \Ms$. The following assertions are equivalent:
%\begin{enumerate}
%\item $\mu_n\wn\mu$;
%\item $\lim_{n\rightarrow \infty }e^{L\kappa_1(\mu_n)}=e^{L\kappa_1(\mu)}$ and $ \lim_{n\rightarrow \infty } L\kappa_k(\mu_n)=L\kappa_k(\mu)$ for all $k\geq2$.
%\end{enumerate}
%\end{corollary}
%
%\begin{proof}Because all measures are probability measures, the weak convergence of $(\mu_k)_{k\in \N}$ to $\mu$ is equivalent from Lemma~\ref{metmom} to the convergences of the sequences of moments $(m_n(\mu_k))_{k\in \N}$ to $m_n(\mu)$ for all $n\in \N^*$, which is equivalent to the second condition of the corollary, using~\eqref{taulogcum}.
%\end{proof}

\subsection{Proof of Proposition~\ref{taulogcumdeux}}\label{proofprop}
Let us formulate a more general formula than~\eqref{taulogcum} with the help of the symmetric group.
\begin{lemma}
Let $\mu\in  \Ms$ and $n\in \mathbb{N}^*$. For all $\sigma \in \mathfrak{S}_n$, we have
\begin{equation}
\prod_{c\text{ cycle of }\sigma}m_{\sharp c}(\mu)=e^{n L\kappa_1(\mu)}\cdot\displaystyle\sum_{\substack{\Gamma \textup{ simple chain in }[1_{\mathfrak{S}_n},\sigma]\\ \Gamma=(\sigma_0,\ldots,\sigma_{|\Gamma|}),\sigma_{|\Gamma|}=\sigma}} \frac{1}{|\Gamma| !}\displaystyle\prod_{i=1}^{|\Gamma|} L\kappa_{d(\sigma_i,\sigma_{i-1})+1}\left(\mu \right).\label{taulogcumtrois}
\end{equation}
\end{lemma}

\begin{proof}The analogue formula of \eqref{taulogcum} for simple chains in $[1_{\mathfrak{S}_n},(1\cdots n)]$ is obtained via the isomorphism of Lemma~\ref{Bia}, remarking that, for a $l$-cycle $\sigma_1^{-1}\sigma_2$ of $[1_{\mathfrak{S}_n},(1\cdots n)]$, we have $l=n-\ell(\sigma_1^{-1}\sigma_2)+1=d(\sigma_1,\sigma_2)+1$. By consequence, we have
\begin{equation*}
m_n(\mu)=e^{n L\kappa_1(\mu)}\cdot\displaystyle\sum_{\substack{\Gamma \text{ simple chain in }[1_{\mathfrak{S}_n},(1\cdots n)]\\ \Gamma=(\sigma_0,\ldots,\sigma_{|\Gamma|}),\sigma_0=1}} \frac{1}{|\Gamma| !} \displaystyle\prod_{i=1}^{|\Gamma|} L\kappa_{d(\sigma_i,\sigma_{i-1})+1}\left(\mu \right).
\end{equation*}
Applying the Kreweras complementation $\sigma\mapsto \sigma^{-1}(1\cdots n)$ which is an isomorphism and preserves simple chains, we obtain
\begin{equation*}
m_n(\mu)=e^{n L\kappa_1(\mu)}\cdot\displaystyle\sum_{\substack{\Gamma \text{ simple chain in }[1_{\mathfrak{S}_n},(1\cdots n)]\\ \Gamma=(\sigma_0,\ldots,\sigma_{|\Gamma|}),\sigma_{|\Gamma|}=(1\cdots n)}} \frac{1}{|\Gamma| !} \displaystyle\prod_{i=1}^{|\Gamma|} L\kappa_{d(\sigma_i,\sigma_{i-1})+1}\left(\mu \right).
\end{equation*}
We now use the fact that for a cycle $c$ of length $\sharp c$, the segment $[1_{\mathfrak{S}_n},c]\subset \mathfrak{S}_n$ is isomorphic as a lattice to $[1_{\mathfrak{S}_{\sharp c}},(1\cdots \sharp c)]\subset \mathfrak{S}_{\sharp c}$, and by consequence, \eqref{taulogcumtrois} is true if $\sigma$ is a cycle.

For an arbitrary permutation $\sigma$, we decompose it into cycles $c_1,\ldots, c_{\ell(\sigma)}$. Constructing a simple chain of length $k$ ending at $\sigma$ is equivalent to constructing $\ell(\sigma)$ simple chains ending respectively at $c_1,\ldots, c_{\ell(\sigma)}$, whose lengths $l_1,\ldots, l_{\ell(\sigma)}$ add up to $k$, and shuffling the steps of these paths, that is choosing a sequence $(C_1,\ldots, C_{\ell(\sigma)})$ of subsets of $\{1,\ldots ,k\}$ which partition $\{1,\ldots ,k\}$ and whose cardinals are $l_1,\ldots, l_{\ell(\sigma)}$ respectively. Using the formula \eqref{taulogcumtrois} for cycles, this remark leads to \eqref{taulogcumtrois} for an arbitrary $\sigma\in \mathfrak{S}_n$.
\end{proof}

In order to conclude the proof of Proposition~\ref{taulogcumdeux}, it suffices to compute explicitly the free log-cumulants of a $\boxtimes$-infinitely divisible measure.

%Let us recall the following proposition (see 
%%Corollary~1.5 of
%~\cite{Mastnak2010}, or
%% Proposition~2.11 of
%~\cite{Cebron2013}).
%\begin{proposition}For all $\mu,\nu\in \Mt$, we have $L\kappa_1(\mu\boxtimes \nu)\equiv L\kappa_1(\mu)+L\kappa_1(\nu) \pmod{2i\pi}$ and, for all $n\geq 2$,
%$$L\kappa_n(\mu\boxtimes \nu)=L\kappa_n(\mu)+L\kappa_n(\nu) .$$
%\end{proposition}
%\begin{proof}
%We have $m_1(\mu\boxtimes \nu)=1/S_{\mu\boxtimes \nu}(0)=1/S_{\mu}(0)\cdot 1/S_{ \nu}(0)=m_1(\mu)\cdot m_1(\nu)$ and therefore $L\kappa_1(\mu\boxtimes \nu)\equiv L\kappa_1(\mu)+L\kappa_1(\nu) \pmod{2i\pi}$. We have also $\log (m_1(\mu\boxtimes \nu)\cdot S_{\mu\boxtimes \nu})=\log (m_1(\mu)\cdot S_\mu\cdot m_1(\nu)\cdot S_\nu)=\log (m_1(\mu)\cdot S_\mu)+\log (m_1(\nu)\cdot S_\nu)$ which implies that $LS_{\mu\boxtimes \nu}=LS_\mu+LS_\nu$ or equivalently that, for all $ n\geq 2$, $L\kappa_n(\mu\boxtimes \nu)=L\kappa_n(\mu)+L\kappa_n(\nu) .$
%\end{proof}

\begin{proposition}\label{loc}
Let $\mu\in \ID(\U,\boxtimes)$ with $\boxtimes$-characteristic triplet $(\omega, b,\upsilon)$. We have
\begin{enumerate}
\item $L\kappa_1(\mu)=\Log(\omega)-b/2+\int_\U \left( \Re(\zeta)-1\right) \diff \upsilon(\zeta)$,
\item $L\kappa_2(\mu)=-b+\int_\U (\zeta -1)^2\diff \upsilon(\zeta)$
\item and $L\kappa_n(\mu)=\int_\U (\zeta -1)^n\diff \upsilon(\zeta)$ for all $n\geq3$.
\end{enumerate} 
\end{proposition}

\begin{proof}The data of $S_\mu(z)$ is given by~\eqref{Sinf}. We first remark that $$m_1(\mu)=S_\mu(0)^{-1}=\omega e^{-b/2-\int_\U \left( i \Im(\zeta)+1-\zeta\right) \diff \upsilon(\zeta)},$$ from which we deduce that $L\kappa_1(\mu)=\Log(m_1(\mu))=\Log(\omega)-b/2+\int_\U \left( \Re(\zeta)-1\right) \diff \upsilon(\zeta)$. We also have $m_1(\mu) S_\mu(z)=S_\mu(z)/S_\mu(0)=\exp\left(bz+\int_{\U}\frac{1-\zeta}{1+z(1-\zeta) }-(1-\zeta)\diff \upsilon (\zeta)\right)$. Therefore, $$LS_\mu(z)=-z\log (m_1(\mu)\cdot S_\mu(z))=-bz^2+\int_{\U}\frac{z^2(\zeta-1)^2}{1-z(\zeta-1)}\diff \upsilon (\zeta).$$
We identify $(L\kappa_n(\mu))_{n\geq 2}$ as the coefficients of $LS_\mu(z)=\sum_{n=2}^\infty L\kappa_n(\mu)z^n.$
\end{proof}

%\begin{lemma}
%Let $(\mu_n)_{n\in \N}$ be a sequence of measure in $ \Mt\cap \Ms$. Let $a\in \C^*$ and $\nu\in  \Mt\cap \Ms$. The following assertions are equivalent:
%\begin{enumerate}
%\item $\lim_{n\rightarrow \infty } m_1(\mu_n)^n=a$ and $ \lim_{n\rightarrow \infty }n \cdot(m_k(\mu_n)-1)=a^k\int_{\U}(z^k-kz-1)\diff \nu(z)$ for all $k\geq2$;
%\item $\lim_{n\rightarrow \infty }e^{nL\kappa_1(\mu_n)}=a$ and $ \lim_{n\rightarrow \infty }n \cdot L\kappa_k(\mu_n)=\int_{\U}(z-1)^k\diff \nu(z)$ for all $k\geq2$.
%\end{enumerate}
%\end{lemma}
\section{Convolution semigroups on $U(N)$}\label{secfour}
In this section, we define and study the convolution semigroups on the unitary group $U(N)$. More precisely, we are interested in computing $\int_{U(N)}g^{\otimes n}\diff \mu(g)$ for $\mu$ arising from a convolution semigroup. In Proposition~\ref{momsemiu} and Proposition~\ref{momsemi}, we shall express this quantity in two different ways. The technique of proof is in the spirit of~\cite{LEVY2008}. It relies on a detailed comprehension of the generator of a semigroup of convolution on $U(N)$ (see~\cite{Liao2004}), and on the Schur-Weyl duality (see Section~\ref{SW}, and~\cite{Collins2003,Collins2004}).

Let $N\in \mathbb{N}$ and let $M_N(\mathbb{C})$ be the space of matrices of dimension $N$
%, with its canonical basis $(E_{a,b})_{1\leq a,b\leq N}$
. If $M\in M_N(\mathbb{C})$, we denote by 
%$M^t$ the transpose of $M$, by $\overline{M}$ the conjugate of $M$, and by 
$M^*$ the adjoint of $M$. Let us denote by $\Tr: M_N(\mathbb{C})\rightarrow \mathbb{C}$ the usual trace.
%, by $\tr: M_N(\mathbb{C})\rightarrow \mathbb{C}$ the normalized trace $\frac{1}{N}\Tr$ and by $\det:M_N(\mathbb{C})\rightarrow \mathbb{C}$ the determinant map. 
The identity matrix is denoted by $I_N$. We consider the unitary group $$U(N)=\{U\in M_N(\mathbb{C}):U^*U=I_N\}.$$
 The $\circledast$-convolution  of two probability measures $\mu$ and $\nu$ on $U(N)$ is defined to be the unique probability measure $\mu\circledast \nu$ on $U(N)$ such that $\int_{U(N)} f \diff (\mu\circledast \nu)=\int_{U(N)} f(gh) \ \mu(\diff g) \nu(\diff h)$ for all bounded Borel function $f$ on $U(N)$. Let us denote by $\ID(U(N),\circledast)$ the space of infinitely divisible probability measures on $U(N)$ and by $\ID_\inv(U(N),\circledast)$ the subspace of measures $\mu$ in $\ID(U(N),\circledast)$ which are invariant by unitary conjugation, that is, such that for all bounded Borel function $f$ on $U(N)$ and all $g\in U(N)$, we have
$$\int_{U(N)}f\diff \mu=\int_{U(N)}f(ghg^*)\diff \mu(h) .$$

\subsection{Generators of semigroups}Let $\mu=(\mu_t)_{t\in \mathbb{R}^+}$ be a weakly continuous semigroup of convolution on $U(N)$ starting at $\mu_0=\delta_e$. We define the  transition  semigroup $(P_t)_{t\in \mathbb{R}^+}$ as follows: for all $t\in \mathbb{R}^+$, all bounded Borel function $f$ on $U(N)$ and all $h\in U(N)$, we set $P_tf(h)=\int_{U(N)}f(hg)\mu_t(\diff g)$. The \textit{generator} of $\mu$, is defined to be the linear operator $L$ on $C(U(N))$ such as $Lf=\lim_{t\rightarrow 0}(P_tf-f)/t$ whenever this limit exists.

In order to describe the generator of a semigroup, we shall successively  introduce in the three next paragraphs the Lie algebra $\frak{u}(N)$ of $U(N)$, a scalar product on $\frak{u}(N)$ and the notion of Lévy measure on $U(N)$.

 The unitary group $U(N)$ is a compact real Lie group of dimension $N^2$, whose Lie algebra $\frak{u}(N)$ is the real vector space of skew-Hermitian matrices: $\frak{u}(N)=\{M\in M_N(\mathbb{C}): M^*+M=0\}.$ We consider also the special unitary group $SU(N)=\{U\in U(N):\det U=1\}$, whose Lie algebra is $\frak{su}(N)=\{M\in \frak{u}(N): \Tr(U)=0\}.$ We remark that $\frak{u}(N)=\frak{su}(N)\oplus(i \mathbb{R} I_N)$. Any $Y\in \frak{u}(N)$ induces a \textit{left invariant vector field }$Y^l$ on $U(N)$ defined for all $g\in U(N)$ by $Y^l(g)=DL_g(Y)$ where $DL_g$ is the differential map of $h\mapsto gh$.

We consider the following \textit{inner product} on $\frak{u}(N)$:
$$(X,Y) \mapsto \left\langle X,Y\right\rangle_{\frak{u}(N)}= \Tr (X^*Y)=-\Tr (XY).$$
It is a real scalar product on $\frak{u}(N)$ which is invariant by unitary conjugation, and its restriction to $\frak{su}(N)$ is also a real scalar product which is invariant by unitary conjugation. Let us fix an orthonormal basis $\left\{Y_1,\ldots,Y_{N^2-1}\right\}$ of $\frak{su}(N)$ and set $Y_{N^2}=\frac{i}{\sqrt{N}}I_N$. This way, $\left\{Y_1,\ldots,Y_{N^2}\right\}$ is an orthonormal basis of $\frak{u}(N)$.\label{bloubi}

It is convenient now to introduce an arbitrary auxiliary set of local coordinates around $I_N$. Let $\Re,\Im:U(N)\rightarrow M_N(\mathbb{C})$ be such that for all $U\in U(N)$, we have $\Re (U)=(U+U^*)/2$ and $\Im (U)=(U-U^*)/2i$. Note that $i \Im$ takes its values in $\frak{u}(N)$. A \textit{Lévy measure} $\Pi$ on $U(N)$ is a measure on $U(N)$ such that $\Pi(\{I_N\})=0$, for all neighborhood $V$ of $I_N$, we have $\Pi(V^c)<+\infty$ and $\int_{U(N)} \|i \Im(x)\|_{\frak{u}(N)}^2\ \Pi(\diff x)<\infty$.

The following theorem gives us a characterization of the generator of such semigroups.
\begin{theorem}[\cite{Applebaum1993,Liao2004}]Let $\mu=(\mu_t)_{t\in \mathbb{R}^+}$ be a weakly continuous semigroup of convolution on $U(N)$ starting at $\mu_0=\delta_e$.
There exist an element $Y_0\in \frak{u}(N)$, a symmetric positive semidefinite matrix $(y_{i,j})_{1\leq i,j \leq N^2}$ and a Lévy measure $\Pi$ on $U(N)$
%, and an independent Poisson process $ N$ on $(0,\infty)\times (U(N)\setminus \{I_N\})$, which has intensity measure $\lambda \otimes \Pi$, such that for all $f\in C^2(U(N))$, we have
%\begin{align*}
%f(U_t)=f(U_0)&+\int_0^tY_0f(U_{s^-})\diff s+\displaystyle\sum_{i=1}^rY_i f(U_{s^-})\circ \diff B^{(i)}_s \\
%&+\int_0^t\int_{U(N)}f(U_{s^-}g)-f(U_{s^-}) \ \tilde{N}(\diff s,\diff g)\\
%&+ \int_0^t\int_{U(N)}f(U_{s^-}g)-f(U_{s^-})-\left( i \Im(g)\right) f(X_{s-})\ \Pi(\diff g) \diff s.
%\end{align*}
%Then 
such that the generator $L$ of $\mu$ is the left-invariant differential operator given, for all $f\in C^2(U(N))$ and all $h\in U(N)$, by 
\begin{equation}
Lf(h)=Y_0^lf(h)+\frac{1}{2}\displaystyle\sum_{i,j=1}^{N^2}y_{i,j}Y_i^l Y_j^l f(h) +\int_{U(N)}f(hg)-f(h)- \left( i \Im(g)\right)^l f(h)\ \Pi(\diff g).\label{gen}
\end{equation}
Conversely, given such a triplet $(Y_0, (y_{i,j})_{1\leq i,j \leq N^2}, \Pi)$, it exists a unique weakly continuous semigroup of convolution on $U(N)$ starting at $\delta_e$ whose generator is given by \eqref{gen}.
\end{theorem}
The triplet $(Y_0, (y_{i,j})_{1\leq i,j \leq N^2}, \Pi)$ is called the \textit{characteristic triplet} of $(\mu_t)_{t\in \mathbb{R}^+}$, or of $L$. Let $\mu\in \ID(U(N),\circledast)$ be such that it exists a weakly continuous semigroup of convolution $(\mu_t)_{t\in \mathbb{R}^+}$ with $\mu_1=\mu$ and $\mu_0=\delta_e$. In this case, we say that the characteristic triplet of $(\mu_t)_{t\in \mathbb{R}^+}$ is \textit{a characteristic triplet} of $\mu$. It is not unique but it completely characterizes the measure $\mu$. Conversely, every triplet of this form is a characteristic triplet of a unique measure in  $\ID(U(N),\circledast)$.

 \subsection{Expected values of polynomials of the entries}Let $n\in \N^*$. In this section, we give a formula for $\int_{U(N)}g^{\otimes n}\diff \mu(g)$ when $\mu$ arises from a convolution semigroup.
Consider the representation $\rho^n_{U(N)}$ of $ U(N)$ on $(\mathbb{C}^N)^{\otimes n}$ given by
$$\rho^n_{U(N)}(g)=\underbrace{g\otimes\cdots \otimes g}_{n \text{ times}}.$$
We set $d\rho^n_{U(N)}(L)=L(\rho^n_{U(N)})(I_N)$, where $\rho^n_{U(N)}$ is seen as an element of $C^2(U(N))\otimes \End((\mathbb{C}^N)^{\otimes n})$.
% Let $\mathcal{U}(\frak{u}(N))$ denote the enveloping algebra of $\frak{u}(N)$, which is canonically isomorphic to the algebra of left invariant differential operators on $U(N)$. The representation $\rho^n_{U(N)}$ determines a homomorphism of associative algebras $d\rho^n_{U(N)}:\mathcal{U}(\frak{u}(N))\rightarrow \End((\mathbb{C}^N)^{\otimes n})$, given by

\begin{proposition}Let $(\mu_t)_{t\in \mathbb{R}^+}$ be a weakly continuous semigroup of convolution on $U(N)$ starting at $\mu_0=\delta_e$ with generator $L$ and characteristic triplet $(Y_0, (y_{i,j})_{1\leq i,j \leq N^2}, \Pi)$. For all $t\in \R_+$, we have the equality in $\End((\mathbb{C}^N)^{\otimes n})$\label{momsemiu}
$$\int_{U(N)}g^{\otimes n}\diff \mu_t(g)=\exp(t\ d\rho^n_{U(N)}(L))$$
with
\begin{eqnarray*}
d\rho^n_{U(N)}(L)&=& \displaystyle\sum_{1\leq k \leq n}\Id_N^{\otimes k-1}\otimes Y_0\otimes \Id_N^{\otimes n-k}\\
&&+\frac{1}{2}\displaystyle\sum_{i,j=1}^{N^2}y_{i,j}\cdot \displaystyle\sum_{1\leq k,l \leq n}\left(\Id_N^{\otimes k-1}\otimes Y_i\otimes \Id_N^{\otimes n-k}\right)\circ\left(\Id_N^{\otimes l-1}\otimes Y_j\otimes \Id_N^{\otimes n-l}\right)\\
&&+\int_{U(N)}\left(g^{\otimes n}-\Id_N^{\otimes n}-\displaystyle\sum_{1\leq k \leq n}\Id_N^{\otimes k-1}\otimes  i \Im(g)\otimes \Id_N^{\otimes n-k}\right) \ \Pi(\diff g).
\end{eqnarray*}

\end{proposition}

%\begin{proposition}Let $L$ be a generator with characteristic triplet $(Y_0, (y_{i,j})_{1\leq i,j \leq N^2}, \Pi)$. We have
%\begin{eqnarray*}
%d\rho^n_{U(N)}(L)&=& \displaystyle\sum_{1\leq k \leq n}\Id_N^{\otimes k-1}\otimes Y_0\otimes \Id_N^{\otimes n-k}\\
%&&+\frac{1}{2}\displaystyle\sum_{i,j=1}^{N^2}y_{i,j}\cdot \displaystyle\sum_{1\leq k,l \leq n}\left(\Id_N^{\otimes k-1}\otimes Y_i\otimes \Id_N^{\otimes n-k}\right)\circ\left(\Id_N^{\otimes l-1}\otimes Y_j\otimes \Id_N^{\otimes n-l}\right)\\
%&&+\int_{U(N)}\left(g^{\otimes n}-\Id_N^{\otimes n}-\displaystyle\sum_{1\leq k \leq n}\Id_N^{\otimes k-1}\otimes  i \Im(g)\otimes \Id_N^{\otimes n-k}\right) \ \Pi(\diff g).
%\end{eqnarray*}
%
%\end{proposition}

\begin{proof}Let denote by $U:U(N)\rightarrow M_N(\mathbb{C})$ the identity function of $U(N)$. We compute
\begin{eqnarray*}L\left(\rho^n_{U(N)}\right)&=&Y_0^l(U^{\otimes n})+\frac{1}{2}\displaystyle\sum_{i,j=1}^{N^2}y_{i,j}Y_i^l Y_j^l(U^{\otimes n})\\
&&+\int_{U(N)}(U g)^{\otimes n}-U^{\otimes n}- \left( i \Im(g)\right)^l (U^{\otimes n})\ \Pi(\diff g)
\end{eqnarray*}
and using that, for all $Y\in \mathfrak{u}(N)$, we have $Y^l(U^{\otimes n})=U^{\otimes n} \cdot \displaystyle\sum_{1\leq k \leq n}\Id_N^{\otimes k-1}\otimes Y\otimes \Id_N^{\otimes n-k}$,
\begin{eqnarray*}L\left(\rho^n_{U(N)}\right)&=& U^{\otimes n}\cdot \displaystyle\sum_{1\leq k \leq n}\Id_N^{\otimes k-1}\otimes Y_0\otimes \Id_N^{\otimes n-k}\\
&&+\frac{1}{2}U^{\otimes n}\cdot\displaystyle\sum_{i,j=1}^{N^2}y_{i,j} \displaystyle\sum_{1\leq k,l \leq n}\left(\Id_N^{\otimes k-1}\otimes Y_i\otimes \Id_N^{\otimes n-k}\right)\cdot\left(\Id_N^{\otimes l-1}\otimes Y_j\otimes \Id_N^{\otimes n-l}\right)\\
&&+U^{\otimes n}\cdot \int_{U(N)}\left(g^{\otimes n}-\Id_N^{\otimes n}-\displaystyle\sum_{1\leq k \leq n}\Id_N^{\otimes k-1}\otimes  i \Im(g)\otimes \Id_N^{\otimes n-k}\right) \ \Pi(\diff g).
\end{eqnarray*}
Hence, $d\rho^n_{U(N)}(L)=L(\rho^n_{U(N)})(e)$ leads to the expression of $d\rho^n_{U(N)}(L)$ given above. We conclude by remarking that $t\to \int_{U(N)}g^{\otimes n}\diff \mu_t(g)=\int_{U(N)}\rho^n_{U(N)}( g)\diff \mu_t(g)$ and $t\to \exp(t\ d\rho^n_{U(N)}(L))$ are both the unique solution to the differential equation
\begin{equation*}
\left\{
\begin{array}{r c l}
    y(0) &=& I_N^{\otimes n}, \\
    y' &=& y\cdot d\rho^n_{U(N)}(L). 
\end{array}
\right.\qedhere\end{equation*}
%$$\int_{U(N)}g^{\otimes n}\diff \mu_t(g)=\int_{U(N)}\rho^n_{U(N)}( g)\diff \mu_t(g)= \exp(d \rho^n_{U(N)}(L) ). \qedhere$$
\end{proof}

We now give an alternative expression of $d\rho^n_{U(N)}(L)$. Let $m\geq 0$. For all $1\leq k_1<\ldots < k_m\leq n$, let us denote by $\iota_{k_1,\ldots,k_m}^{M_N(\mathbb{C})^{\otimes n}}: M_N(\mathbb{C})^{\otimes m}\rightarrow  M_N(\mathbb{C})^{\otimes n}$ (or more simply $\iota_{k_1,\ldots,k_m})$ the mapping defined by
$$\iota_{k_1,\ldots,k_m}(X_1\otimes \cdots \otimes X_m)=I_N^{\otimes k_1-1}\otimes X_1 \otimes I_N^{\otimes k_2-k_1-1}\otimes X_2\otimes \cdots \otimes X_m\otimes I_N^{\otimes n-k_m},$$
that is to say in words that $\iota_{k_1,\ldots,k_m}(X_1\otimes \cdots \otimes X_m)$ is the tensor product of $X_1,\ldots,X_m$ at the places $k_1,\ldots,k_m$ and $I_N$ at the other places.

\begin{proposition}Let $L$ be a generator with characteristic triplet $(Y_0, (y_{i,j})_{1\leq i,j \leq N^2}, \Pi)$. We have\label{lulu}
\begin{eqnarray*}
d\rho^n_{U(N)}(L)&=& \displaystyle\sum_{1\leq k \leq n}\iota_{k}\left(Y_0+\int_{U(N)}(\Re(g)-I_N)\Pi(\diff g)\right)\\
&&+\frac{1}{2}\displaystyle\sum_{i,j=1}^{N^2}y_{i,j}\cdot \displaystyle\sum_{1\leq k,l \leq n}\iota_{k}( Y_i)\circ \iota_{l}( Y_j)\\
&&+\displaystyle\sum_{\substack{2\leq m \leq n\\1\leq k_1<\ldots<k_m \leq n}}\iota_{k_1,\ldots,k_m}\left(\int_{U(N)} (g-I_N)^{\otimes m}\Pi(\diff g)\right).
\end{eqnarray*}

\end{proposition}

\begin{proof}Our starting point is the expression of $d\rho^n_{U(N)}(L)$ given by Proposition~\ref{momsemi}. Let us remark that
\begin{align*}g^{\otimes n}&=(g-I_N+I_N)^{\otimes n}\\
&=I_N^{\otimes n}+\displaystyle\sum_{\substack{1\leq m \leq n\\1\leq k_1<\ldots<k_m \leq n}}\iota_{k_1,\ldots,k_m}((g-I_N)^{\otimes m}),
\end{align*}
from which we deduce that
\begin{eqnarray*}
&&\int_{U(N)}\left(g^{\otimes n}-\Id_N^{\otimes n}-\displaystyle\sum_{1\leq k \leq n}\Id_N^{\otimes k-1}\otimes  i \Im(g)\otimes \Id_N^{\otimes n-k}\right) \ \Pi(\diff g)\\
&=&\int_{U(N)}\left( \displaystyle\sum_{1\leq k \leq n} \iota_{k}(g-I_N-i\Im(g)) +\displaystyle\sum_{\substack{2\leq m \leq n\\1\leq k_1<\ldots<k_m \leq n}} \iota_{k_1,\ldots,k_m}((g-I_N)^{\otimes m})\right) \ \Pi(\diff g)\\
&=& \displaystyle\sum_{1\leq k \leq n} \iota_{k}\left(\int_{U(N)}(\Re(g)-I_N)\Pi(\diff g)\right) +\displaystyle\sum_{\substack{2\leq m \leq n\\1\leq k_1<\ldots<k_m \leq n}}\iota_{k_1,\ldots,k_m}\left(\int_{U(N)} (g-I_N)^{\otimes m}\Pi(\diff g)\right)
\end{eqnarray*}
because all the integrand are equivalent to $\|i \Im(g)\|_{\frak{u}(N)}^2$ in a neighborhood of $I_N$ and hence integrable with respect to $\Pi$. Replacing the last term by this new expression in Proposition~\ref{momsemiu} yields to the result.
\end{proof}

\subsection{Conjugate invariant semigroups on $U(N)$}

A weakly continuous  convolution semigroup $(\mu_t)_{t\in \mathbb{R}^+}$ on $U(N)$ starting at $\mu_0=\delta_e$ is said \textit{conjugate invariant} if all $\mu_t$ belong to $\ID_\inv(U(N),\circledast)$.
%are conjugate invariant by a unitary matrice: for all positive Borel function of $U(N)$ and all $g\in U(N)$, we have
%$$\int_{U(N)}f\diff \mu_t=\int_{U(N)}f(ghg^*)\diff \mu_t(h) .$$

\begin{proposition}
Let $(\mu_t)_{t\in \mathbb{R}^+}$ be a weakly continuous  convolution semigroup starting at $\mu_0=\delta_e$ which is conjugate invariant. Let $(Y_0, (y_{i,j})_{1\leq i,j \leq N}, \Pi)$ be its characteristic triplet. The differential operator $\frac{1}{2}\sum_{i,j=1}^{N^2}y_{i,j}Y_i^l Y_j^l$ and the measure $\Pi$ are both conjugate invariant. Moreover, there exists three constants $y_0, \alpha$ and  $\beta\in \mathbb{R}$ such that $Y_0=i y_0I_N$ and $$(y_{i,j})_{1\leq i,j \leq N^2}=\left(\begin{array}{cccc}\alpha &  &  & 0 \\ & \ddots &  &  \\ &  & \alpha &  \\0 &  &  & \beta\end{array}\right).$$

\end{proposition}
\begin{proof}
Thanks to~\cite{Liao2004}, if we denote by $(Y_0, (y_{i,j})_{1\leq i,j \leq N^2}, \Pi)$ the characteristic triplet of $\mu$, the differential operator $\frac{1}{2}\sum_{i,j=1}^{N^2}y_{i,j}Y_i^l Y_j^l$ and the measure $\Pi$ are both conjugate invariant. The mapping $i \Im$ has been chosen to be conjugate invariant and following the proof of Proposition 4.2.2 of \cite{LEVY2010}, we deduce that $Y_0$ is in the center of $\mathfrak{u}(N)$: there exists $y_0\in \mathbb{R}$ such that $Y_0=i y_0\Id_N$.

Because $\left\{Y_1,\ldots,Y_{N^2-1}\right\}$ is a basis of the conjugate invariant Lie subalgebra $\frak{su}(N)$, $\{y_{i,N},y_{N,i}:1\leq i \leq N^2\}=\{0\}$, and because $\frak{su}(N)$ is simple, there exists $\alpha\in \R$ such that $(y_{i,j})_{1\leq i,j \leq (N-1)^2}=\alpha I_{N-1}$. We set $\beta=y_{N,N}$.
\end{proof}
Thus, the invariance by conjugation of $\mu$ implies that its generator $L$ is a bi-invariant pseudo-differential operator. In this particular case, the expression of $d\rho^n_{U(N)}(L)$ can be described with the help of the symmetric group. It is the objet of the next section to use the Schur-Weyl duality in order to formulate a new expression of $d\rho^n_{U(N)}(L)$.

\subsection{Schur-Weyl duality}\label{SW}The Schur-Weyl duality is a deep relation between the actions of $U(N)$ and $\mathfrak{S}_n$ on $(\mathbb{C}^N)^{\otimes n}$ which allows one to transfer some elements relative to $U(N)$ to elements relative to $\mathfrak{S}_n$ (see \cite{Collins2003,Collins2004}, and also \cite{LEVY2008} and~\cite {Dahlqvist2012}). Let us spell out this fruitful duality.

Let $n\in \mathbb{N}$. Define the action $\rho_N^{\mathfrak{S}_n}$ of $\mathfrak{S}_n$ on $(\mathbb{C}^N)^{\otimes n}$ as follows: for all $\sigma \in \mathfrak{S}_n$ and $x_1,\ldots,x_n\in \mathbb{C}^N$, we set
$$(\rho_N^{\mathfrak{S}_n}(\sigma))(x_1\otimes \cdots \otimes x_n)=x_{\sigma^{-1}(1)}\otimes \cdots \otimes x_{\sigma^{-1}(n)}.$$
%Let $\sigma \in \mathfrak{S}_n$. For all $M \in M_N(\mathbb{C})$, we set
%$$p_{\sigma}(M)=\tr((M\otimes \cdots \otimes M)\circ(\rho(\sigma,\Id_N)) .$$
%We remark that for all $g\in U(N)$, we have $\rho_{n,N}(1\otimes g)=\rho^n_{U(N)}(g)$, and in the other side, the map $\sigma\mapsto \rho_{n,N}(\sigma\otimes 1)$ is a representation of $\mathfrak{S}_n$ on $(\mathbb{C}^N)^{\otimes n}$ which will be denoted by $\rho_N^{\mathfrak{S}_n}$. 
Let us denote by $\mathbb{C}[\mathfrak{S}_n]$ the group algebra of $\mathfrak{S}_n$. The action $\rho_N^{\mathfrak{S}_n}$ determines a homomorphism of associative algebra $d \rho_N^{\mathfrak{S}_n}:\mathbb{C}[\mathfrak{S}_n]\rightarrow \End((\mathbb{C}^N)^{\otimes n})$.
% The map $\sigma\mapsto \rho_{n,N}(\sigma\otimes 1)$ is a homomorphism of associative algebra from $\mathbb{C}[\mathfrak{S}_n]$ into $(\mathbb{C}^N)^{\otimes n}$ which will be denoted by $d \rho_N^{\mathfrak{S}_n}$.
The Schur-Weyl duality asserts that the subalgebras of $\End((\mathbb{C}^N)^{\otimes n})$ generated by the action of $U(N)$ and $\mathfrak{S}_n$ are each other's commutant. In particular, all element of $\End((\mathbb{C}^N)^{\otimes n})$ which commutes with $\rho^n_{U(N)} (g)$ for all $g\in U(N)$ is an element of the algebra generated by $\rho_N^{\mathfrak{S}_n}(\mathfrak{S}_n)$, that is to say an element of $d\rho_N^{\mathfrak{S}_n}(\mathbb{C}[\mathfrak{S}_n])$.

For all $\mathbf{A}\in \End((\mathbb{C}^N)^{\otimes n})$, we define
$$E(\mathbf{A})=\int_{U(N)}g^{\otimes n}\circ \mathbf{A} \circ (g^*)^{\otimes n} \diff g \in \End((\mathbb{C}^N)^{\otimes n})$$
where the integration is taken with respect to the Haar measure of $U(N)$. Obviously, $E(\mathbf{A})$ commutes with $\rho^n_{U(N)} (g)$ for all $g\in U(N)$, and due to the Schur-Weyl duality, $E(\mathbf{A})$ has to lie in $d\rho_N^{\mathfrak{S}_n}(\mathbb{C}[\mathfrak{S}_n])$. In Proposition 2.4 of \cite{Collins2004}, Collins and \'{S}niady answered the question of determining an element of $\mathbb{C}[\mathfrak{S}_n]$ which is mapped on $E(\mathbf{A})$, as follows. Set
$$\Phi(\mathbf{A})=\displaystyle\sum_{\sigma\in\mathfrak{S}_n }\Tr\left(\mathbf{A}\circ \rho_N^{\mathfrak{S}_n}(\sigma^{-1}) \right) \cdot \sigma \in \mathbb{C}[\mathfrak{S}_n]$$
and define $\Wg=\sum_{\sigma\in\mathfrak{S}_n}\Wg(\sigma)\cdot\sigma\in  \mathbb{C}[\mathfrak{S}_n]$ such that $d \rho_N^{\mathfrak{S}_n}( \Phi(\Id_N^{\otimes n})\cdot \Wg)=\Id_N^{\otimes n}$. If $n\leq N$, the element $ \Phi(\Id_N^{\otimes n})$ is invertible and $\Wg$ must be $ \Phi(\Id_N^{\otimes n})^{-1}$. If $N< n$, one can choose any pseudo-inverse of the symmetric element $ \Phi(\Id_N^{\otimes n})$ to be $\Wg$. Let us insist on the fact that $\Wg$ depends on both $n$ and $N$, even if for convenience, this dependence is not explicit in the notation.
\begin{proposition}[\cite{Collins2004}]For all $\mathbf{A}\in \End((\mathbb{C}^N)^{\otimes n})$, we have $E(\mathbf{A})=d\rho_N^{\mathfrak{S}_n}( \Phi(\mathbf{A}) \Wg).$\label{collins}
\end{proposition}
Very succinctly, the argument is as follows:
\[ \rho_N^{\mathfrak{S}_n}( \Phi(\mathbf{A}) )=\rho_N^{\mathfrak{S}_n}( \Phi(E(\mathbf{A})) )=\rho_N^{\mathfrak{S}_n}( \Phi( E(\mathbf{A})\cdot\Id_N^{\otimes n}) )=E(\mathbf{A})\cdot\rho_N^{\mathfrak{S}_n}( \Phi( \Id_N^{\otimes n}) ). \]
It allows us to write explicitly elements of the commutant of the algebra generated by $\rho^n_{U(N)}$ as elements of $d\rho_N^{\mathfrak{S}_n}(\mathbb{C}[\mathfrak{S}_n])$. Indeed, if $\mathbf{A}$ commutes with $\rho^n_{U(N)} (g)$ for all $g\in U(N)$, we have $$\mathbf{A}=E(\mathbf{A})=d\rho_N^{\mathfrak{S}_n}( \Phi(\mathbf{A}) \Wg).$$
Moreover, they give an asymptotic of the Weingarten function.
\begin{proposition}[\cite{Collins2004}]For all $\sigma\in \mathfrak{S}_n $, we have $\Wg(\sigma)=O(N^{-n-|\sigma|})$ when $N$ tends to $\infty$. We have also $\Wg(1_{\mathfrak{S}_n})=N^{-n}+O(N^{-n-2})$ when $N$ tends to $\infty$.\label{collins2}
\end{proposition}

\begin{example}\label{undeux}
\begin{enumerate}
\item for $n=1$: we have $\Wg=1_{\mathfrak{S}_n}/N$, and therefore, for all $A\in \End(\mathbb{C}^N)$,$$E(A)=\frac{1}{N}\Tr(A)I_N;$$
\item for $n=2$: we have $\Wg=\frac{1}{N^2-1}\left(1_{\mathfrak{S}_n}-\frac{1}{N}(1,2)\right)$, and thus, for all $A,B \in \End(\mathbb{C}^N)$,
$$E(A\otimes B)=\frac{1}{N^2-1}\Big((\Tr(A)\Tr(B)-\Tr(AB)/N)\cdot I_N^{\otimes 2}+(\Tr(AB)-\Tr(A)\Tr(B)/N)\cdot d\rho_N^{\mathfrak{S}_n}((1,2))\Big) .$$
\end{enumerate}
\end{example}

The generator $L$ of a conjugate invariant convolution semigroup is a bi-invariant pseudo-differential operator, and by consequence the element $d \rho^n_{U(N)}( L)$ commutes with $\rho^n_{U(N)} (g)$ for all $g\in U(N)$. Thus, it is an element of $d\rho_N^{\mathfrak{S}_n}(\mathbb{C}[\mathfrak{S}_n])$. Let $\mathcal{T}_n$ be the subset of $\mathfrak{S}_n$ consisting of all the transpositions. For all $1\leq k_1<\ldots < k_m\leq n$, let us denote by $\iota_{k_1,\ldots,k_m}^{\mathfrak{S}_n}:\mathfrak{S}_m\to \mathfrak{S}_n$ (or more simply $\iota_{k_1,\ldots,k_m}$) the mapping defined by
$$\iota_{k_1,\ldots,k_m}(\sigma):\left|\begin{array}{rll}k_i \mapsto & k_{\sigma(i)} &  \\i \mapsto & i & \text{ for }i\notin \{k_1,\ldots, k_m\}.\end{array}\right.$$
This map is such that $\rho_N^{\mathfrak{S}_n}\circ \iota_{k_1,\ldots,k_m}^{\mathfrak{S}_n}=\iota_{k_1,\ldots,k_m}^{M_N(\mathbb{C})^{\otimes n}}\circ \rho_N^{\mathfrak{S}_m}$. We are now ready to state the main result of this section.

\begin{proposition}\label{momsemi}Let $y_0, \alpha,\beta\in \mathbb{R}$ and $\Pi$ be a L\'{e}vy measure on $U(N)$ which is conjugate invariant. Let $\mu\in \ID(U(N),\circledast)$ with characteristic triplet $$\left(i y_0I_N,\left(\begin{array}{cccc}\alpha &  &  & 0 \\ & \ddots &  &  \\ &  & \alpha &  \\0 &  &  & \beta\end{array}\right), \Pi\right).$$
%Let $\tilde{L}\in \C[\mathfrak{S}_n]$ given by~\eqref{ltilde}. We have
We have
$
\int_{U(N)}g^{\otimes n}\diff \mu(g)=d\rho^{\mathfrak{S}_n}_N(e^{\tilde{L}}),$
where
\begin{multline*}
\tilde{L}= \left(niy_0-\frac{n^2}{N}\frac{\beta}{2}+\left(\frac{n^2}{N}-nN\right)\frac{\alpha}{2}+\frac{n}{N}\int_{U(N)}\Tr\left( \Re(g)-1\right) \Pi(\diff g)\right)1_{\mathfrak{S}_n} -\alpha \displaystyle\sum_{\tau\in \mathcal{T}_n}\tau\\
%&&+\displaystyle\sum_{i,j=1}^{N^2}y_{i,j}\left(\frac{\Tr(Y_iY_j)}{N^2-1}-\frac{\Tr(Y_i)\Tr(Y_j)}{N^3-N}\right)\cdot \displaystyle\sum_{\tau\in T_n}\tau\\
+\displaystyle\sum_{\substack{2\leq m \leq n\\1\leq k_1<\ldots<k_m \leq n}}\displaystyle\sum_{\sigma, \pi \in \mathfrak{S}_m}\Wg(\sigma^{-1} \pi)\cdot \int_{U(N)}\displaystyle\prod_{c \text{ cycle of }\sigma}\Tr\left((g-1)^{\sharp c}\right) \ \Pi(\diff g)\cdot \iota_{k_1,\ldots,k_m}(\pi).
\end{multline*}
\end{proposition}
\begin{proof}Let $(\mu_t)_{t\in \mathbb{R}^+}$ be the weakly continuous semigroup of convolution whose characteristic triplet is $$\left(i y_0I_N,\left(\begin{array}{cccc}\alpha &  &  & 0 \\ & \ddots &  &  \\ &  & \alpha &  \\0 &  &  & \beta\end{array}\right), \Pi\right),$$
and let $L$ be its generator. By definition, $\mu=\mu_1$, and thanks to Proposition~\ref{momsemiu}, we know that
$$
\int_{U(N)}g^{\otimes n}\diff \mu(g)=\exp(d\rho^n_{U(N)}(L)).$$
%We conclude using the following lemma.
%
To conclude, it suffices to prove that 
%
%
%
%\begin{lemma}Let $y_0, \alpha,\beta\in \mathbb{R}$ and $\Pi$ be a L\'{e}vy measure on $U(N)$ which is conjugate invariant.  Let $L$ be a generator with characteristic triplet $$\left(i y_0I_N,\left(\begin{array}{cccc}\alpha &  &  & 0 \\ & \ddots &  &  \\ &  & \alpha &  \\0 &  &  & \beta\end{array}\right), \Pi\right).$$
%We have
$d \rho^n_{U(N)}( L)=d \rho^{\mathfrak{S}_n}_N(\tilde{L})$.
%, where\label{exprL}
%\begin{multline}
%\tilde{L}= \left(niy_0-\frac{n^2}{N}\frac{\beta}{2}+\left(\frac{n^2}{N}-nN\right)\frac{\alpha}{2}+\frac{n}{N}\int_{U(N)}\Tr\left( \Re(g)-1\right) \Pi(\diff g)\right)1_{\mathfrak{S}_n} -\alpha \displaystyle\sum_{\tau\in \mathcal{T}_n}\tau\\
%%&&+\displaystyle\sum_{i,j=1}^{N^2}y_{i,j}\left(\frac{\Tr(Y_iY_j)}{N^2-1}-\frac{\Tr(Y_i)\Tr(Y_j)}{N^3-N}\right)\cdot \displaystyle\sum_{\tau\in T_n}\tau\\
%+\displaystyle\sum_{\substack{2\leq m \leq n\\1\leq k_1<\ldots<k_m \leq n}}\displaystyle\sum_{\sigma, \pi \in \mathfrak{S}_m}\Wg(\sigma^{-1} \pi)\int_{U(N)}\displaystyle\prod_{c \text{ cycle of }\sigma}\Tr\left((g-1)^{\sharp c}\right) \ \Pi(\diff g)\cdot \iota_{k_1,\ldots,k_m}(\pi).
%\end{multline}
%\end{lemma}
%
%\begin{proof}
We start from Proposition~\ref{lulu}. We have
\begin{eqnarray*}
d\rho^n_{U(N)}(L)&=& \displaystyle\sum_{1\leq k \leq n}\iota_{k}\left(iny_0+\int_{U(N)}(\Re(g)-I_N)\Pi(\diff g)\right)\\
&&+\frac{\alpha}{2} \displaystyle\sum_{i=1}^{N^2-1}\cdot \displaystyle\sum_{1\leq k,l \leq n}\iota_{k}( Y_i)\circ \iota_{l}( Y_i)+\frac{\beta}{2}\displaystyle\sum_{1\leq k,l \leq n}\iota_{k}( Y_{N^2})\circ \iota_{l}( Y_{N^2})\\
&&+\displaystyle\sum_{\substack{2\leq m \leq n\\1\leq k_1<\ldots<k_m \leq n}}\iota_{k_1,\ldots,k_m}\left(\int_{U(N)} (g-I_N)^{\otimes m}\Pi(\diff g)\right).
\end{eqnarray*}
Thanks to the invariance under conjugation of $\Pi$ and $\sum_{i=1}^{N^2-1} Y_i\otimes Y_i$, we know from Example~\ref{undeux} that
$$
\int_{U(N)}(\Re(g)-I_N)\Pi(\diff g)=\int_{U(N)}E(\Re(g)-I_N)\Pi(\diff g)=\int_{U(N)}\frac{1}{N}\Tr\left( \Re(g)-1\right) \Pi(\diff g)$$
and $$\displaystyle\sum_{i=1}^{N^2-1}Y_i\otimes Y_i= E\left(\displaystyle\sum_{i=1}^{N^2-1}Y_i\otimes Y_i\right)=\frac{1}{N}I_N^{\otimes 2}-\rho_N^{\mathfrak{S}_n}((1,2)).$$
We also deduce from Proposition~\ref{collins} that
\begin{align*}
\int_{U(N)} (g-I_N)^{\otimes m}\Pi(\diff g)&=\int_{U(N)}E\left( (g-I_N)^{\otimes m}\right)\Pi(\diff g)\\
&=\int_{U(N)} \displaystyle\sum_{\sigma, \pi \in \mathfrak{S}_m}\Wg(\sigma^{-1} \pi)\displaystyle\prod_{c \text{ cycle of }\sigma}\Tr\left((g-1)^{\sharp c}\right) \cdot  d\rho_N^{\mathfrak{S}_m}(\pi)\Pi(\diff g).
\end{align*}
Thus we have
\begin{eqnarray*}
\diff\rho^n_{U(N)}(L)&=&\left(niy_0-\frac{n^2}{N}\frac{\beta}{2}+\left(\frac{n^2}{N}-nN\right)\frac{\alpha}{2}+\frac{n}{N}\int_{U(N)}\Tr\left( \Re(g)-1\right) \Pi(\diff g)\right)I_N^{\otimes n}\\
&&-\alpha \displaystyle\sum_{i=1}^{N^2-1}\cdot \displaystyle\sum_{1\leq k<l \leq n}\iota_{k,l}\circ \rho_N^{\mathfrak{S}_2}((1,2))\\
&&+\displaystyle\sum_{\substack{2\leq m \leq n\\1\leq k_1<\ldots<k_m \leq n}}\displaystyle\sum_{\sigma, \pi \in \mathfrak{S}_m}\Wg(\sigma^{-1} \pi)\cdot \int_{U(N)}\displaystyle\prod_{c \text{ cycle of }\sigma}\Tr\left((g-1)^{\sharp c}\right) \ \Pi(\diff g)\\
&&\hspace{9cm}\cdot \iota_{k_1,\ldots,k_m}\circ \rho_N^{\mathfrak{S}_m}(\pi),
\end{eqnarray*}
from which we deduce that $d \rho^n_{U(N)}( L)=d \rho^{\mathfrak{S}_n}_N(\tilde{L})$.
%\end{proof}
\end{proof}

%\begin{corollary}\label{momsemi}Let $y_0, \alpha,\beta\in \mathbb{R}$ and $\Pi$ be a L\'{e}vy measure on $U(N)$ which is conjugate invariant. Let $\mu$ be a weakly continuous semigroup of convolution $\mu=(\mu_t)_{t\in \mathbb{R}^+}$ on $U(N)$ starting at $\mu_0=\delta_e$ with characteristic triplet $$(i y_0I_N,\left(\begin{array}{cccc}\alpha &  &  & 0 \\ & \ddots &  &  \\ &  & \alpha &  \\0 &  &  & \beta\end{array}\right), \Pi).$$
%Let $\tilde{L}\in \C[\mathfrak{S}_n]$ given by~\eqref{ltilde}. We have
%\begin{equation}
%\int_{U(N)}g^{\otimes n}\diff \mu_1(g)=d\rho^{\mathfrak{S}_n}_N(e^{\tilde{L}}).\label{power}
%\end{equation}
%
%\end{corollary}
%\begin{proof}Let $L$ be the generator of $\mu$. We just write
%$$\int_{U(N)}g^{\otimes n}\diff \mu_1(g)=\int_{U(N)}\rho^n_{U(N)}( g)\diff \mu_1(g)=d \rho^n_{U(N)}( e^L)=d\rho^{\mathfrak{S}_n}_N(e^{\tilde{L}}) . \qedhere$$
%\end{proof}

\section{The stochastic exponential $\mathcal{E}_N$}\label{secfive}

In this section, we shall describe $\mathcal{E}_N$, a map which connects the infinitely divisible measures on the space of Hermitian matrices $\mathcal{H}_N$ and the infinitely divisible measures on $U(N)$. We start by presenting $\mathcal{E}_N$ in Proposition-Definition~\ref{ththree}, and the rest of the section is devoted to the proof of Proposition-Definition~\ref{ththree}.

We consider the Hilbert space of Hermitian matrices $$\mathcal{H}_N=\{x\in M_N(\mathbb{C}):x^*=x\}.$$
We denote by $\ast$ the classical convolution on the vector space $\mathcal{H}_N$: given two probability measures $\mu$ and $\nu$ on $\mathcal{H}_N$, the convolution $\mu\ast \nu$ is such that $\int_{\mathcal{H}_N} f \diff (\mu\ast \nu)=\int_{\mathcal{H}_N}\int_{\mathcal{H}_N} f(x+y) \ \mu(\diff x) \nu(\diff y)$ for all bounded Borel function $f$ on $\mathcal{H}_N$. Let us denote by $\ID(\mathcal{H}_N,\ast)$ the space of infinitely divisible probability measures on $\mathcal{H}_N$ and by $\ID_\inv(\mathcal{H}_N,\ast)$ the subspace of measures $\mu$ in $\ID(\mathcal{H}_N,\ast)$ which are invariant by unitary conjugation, that is, such that for all bounded Borel function $f$ on $\mathcal{H}_N$ and all $g\in U(N)$, we have
$$\int_{\mathcal{H}_N}f\diff \mu=\int_{\mathcal{H}_N}f(gxg^*)\diff \mu(x) .$$

\subsection{Infinite divisibility on $\mathcal{H}_N$}The advantage of $\ID(\mathcal{H}_N,\ast)$ is that each infinitely divisible measures arises from a unique convolution semigroup, and by consequence, is characterized by a unique generator. In order to describe this generator, we introduce now an inner product on $\mathcal{H}_N$ and we define the notion of Lévy measure.\label{secfivebis}

We endow $\mathcal{H}_N$ with the following inner product:
$$(x,y) \mapsto \left\langle x,y\right\rangle_{\mathcal{H}_N}= \Tr (x^*y)=\Tr (xy).$$
It is a real scalar product on $\mathcal{H}_N$ which is invariant by unitary conjugation. We remark that $i\mathcal{H}_N=\mathfrak{u}(N)$. Thus, the family $\left\{X_1,\ldots,X_{N^2}\right\}=\left\{-iY_1,\ldots,-iY_{N^2}\right\}$ is an orthonormal basis of $\mathcal{H}_N$ such that $X_{N^2}=\frac{1}{\sqrt{N}}I_N$. It is now useful to fix one compact neighborhood $B$ of $0$: we choose to set $B=B(0,1)$, the closed unit ball of $\mathcal{H}_N$.
%, which is consistent with Section~\ref{adconv}.
A \textit{Lévy measure} $\Pi$ on $\mathcal{H}_N$ is a measure on $\mathcal{H}_N$ such that both $\Pi(\{0\})=0$ and such that $\int_{B} \|x\|_{\mathcal{H}_N}^2\ \Pi(\diff x)$ and $\Pi(B^c)$ are finite.
% Since all norms are equivalent in $\mathcal{H}_N$, the definition is independent of the norm chosen.

Let $C^2_b(\mathcal{H}_N)$ be the space of function $f\in C^2(\mathcal{H}_N)$ with bounded first and second-order partial derivatives.
%such that $\partial_{X}f$ and $\partial_{Y}\partial_{X}f$ are bounded for any $X,Y\in \mathcal{H}_N$.
%The Lévy-Itô representation of a Lévy process characterizes Lévy processes by stochastic integral equations with respect to Brownian motions and Poisson random measures.

\begin{theorem}[\cite{Sato1999,Liao2004}]
Let $\mu\in \ID(\mathcal{H}_N,\ast)$. There exists a unique weakly continuous semigroup $(\mu^{\ast t})_{t\in \mathbb{R}^+}$ such that $\mu^{\ast 0}=\delta_0$ 	and $\mu^{\ast 1}=\mu$.  There exist an element $X_0\in \mathcal{H}_N$, a symmetric positive semidefinite matrix $(y_{i,j})_{1\leq i,j \leq N^2}$ and a Lévy measure $\Pi$ on $\mathcal{H}_N$
%, and an independent Poisson process $ N$ on $(0,\infty)\times (U(N)\setminus \{I_N\})$, which has intensity measure $\lambda \otimes \Pi$, such that for all $f\in C^2(U(N))$, we have
%\begin{align*}
%f(U_t)=f(U_0)&+\int_0^tY_0f(U_{s^-})\diff s+\displaystyle\sum_{i=1}^rY_i f(U_{s^-})\circ \diff B^{(i)}_s \\
%&+\int_0^t\int_{U(N)}f(U_{s^-}g)-f(U_{s^-}) \ \tilde{N}(\diff s,\diff g)\\
%&+ \int_0^t\int_{U(N)}f(U_{s^-}g)-f(U_{s^-})-\left( i \Im(g)\right) f(X_{s-})\ \Pi(\diff g) \diff s.
%\end{align*}
%Then 
such that the generator $L$ of $(\mu^{\ast t})_{t\in \mathbb{R}^+}$ is given for all $f\in C^2_b(\mathcal{H}_N)$ and all $y\in \mathcal{H}_N$ by 
\begin{equation}
Lf(y)=\partial_{X_0}f(y)+\frac{1}{2}\displaystyle\sum_{i,j=1}^{N^2}y_{i,j}\partial_{X_i}\partial_{X_j} f(y) 
+\int_{\mathcal{H}_N}f(y+x)-f(y)- 1_B(x)\partial_{x} f(y)\ \Pi(\diff x) .\label{genu}
\end{equation}
\end{theorem}
%\begin{align*}
%Lf(y)=\partial_{Y_0}f(y)&+\frac{1}{2}\displaystyle\sum_{i=1}^r\partial_{Y_i}^2 f(y) \\
%&+\int_{U}f(y+x)-f(y)- \partial_{x} f(y)\ \Pi(\diff x) \\
%&+\int_{U^c}f(y+x)-f(y) \ \Pi(\diff y).
%\end{align*}
The triplet $(X_0, (y_{i,j})_{1\leq i,j \leq N^2}, \Pi)$ is called the \textit{characteristic triplet} of $\mu$, and  its associated generator $L$ is called the generator of $\mu$. Conversely, given such a triplet $(X_0, (y_{i,j})_{1\leq i,j \leq N^2}, \Pi)$, there exists a unique infinitely divisible  measure $\mu$ whose generator is given by \eqref{genu}.

Let us remark that the functions $\e$ and $\sin$ make sense on $\mathcal{H}_N$. For all $x\in \mathcal{H}_N$, we have
$$\e(x)=\exp(ix)\in U(N)\text{ and }\sin(x)=\Im\circ \e=(e^{ix}-e^{-ix})/2i\in \mathcal{H}_N. $$
As previously, for all measure $\Pi$ on $\mathcal{H}_N$, the measure $\e_{\ast} (\Pi)$ denotes the push-forward $\Pi$ on $\mathcal{H}_N$ by the mapping $\e:\mathcal{H}_N\to U(N)$, and the measure $\e_{\ast} (\Pi)_{|U(N)\setminus \{I_N\}}$ is the measure on $U(N)\setminus \{I_N\}$ induced by $\e_{\ast} (\Pi)$. We are now able to formulate the main result of this section.
\begin{propdef}\label{ththree}For all $\mu\in \ID(\mathcal{H}_N,\ast)$ with characteristic triplet $$(X_0, (y_{i,j})_{1\leq i,j \leq N^2}, \Pi),$$ we define $\mathcal{E}_N(\mu)$ to be the measure of $\ID(U(N),\circledast)$ with characteristic triplet $$\left(iX_0+i\int_{\mathcal{H}_N}(\sin(x)-1_B(x) x) \ \Pi(\diff x), (y_{i,j})_{1\leq i,j \leq N^2}, \e_{\ast} (\Pi)_{|U(N)\setminus \{I_N\}}\right).$$ The map $\mathcal{E}_N:\mathcal{ID}(\mathcal{H}_N,\ast)\rightarrow \ID(U(N),\circledast)$ is called the stochastic exponential and has the following properties
:
\begin{enumerate}\item For all $\mu\in \ID(\mathcal{H}_N,\ast)$, the measures $(\e_{\ast}(\mu^{\ast 1/n}))^{\circledast n}$ converge weakly to $\mathcal{E}_N(\mu)$;

\item the stochastic exponential maps $ \ID_{\inv}(\mathcal{H}_N,\ast)$ to $\ID_{\inv}(U(N),\circledast)$, and for all $\mu$, $\nu$ measures of $ \ID_{\inv}(\mathcal{H}_N,\ast)$, we have
$$\mathcal{E}_N(\mu\ast \nu)=\mathcal{E}_N(\mu)\circledast\mathcal{E}_N( \nu) .$$
%\item for all natural numbers $k_1<k_2<\cdots$,  probability measures $\mu_1 , \mu_2, \ldots$ and all $\mu\in  \ID_{\inv}(\mathcal{H}_N,\ast)$, the measures $\mu_n^{\ast k_n}$ converge weakly to $\mu$ implies that the measures $\e_{\ast}(\mu_n)^{\ast k_n}$ converge weakly to $\mathcal{E}_N(\mu)$;

\end{enumerate}
\end{propdef}
The tool used to prove this proposition is the Fourier transform of a measure on $U(N)$. Before proving Proposition~\ref{ththree} in Section~\ref{proofdef}, let us introduce this notion. 

\subsection{Fourier transform on $U(N)$}The set $\widehat{U(N)}$ of isomorphism classes of irreducible representations of $U(N)$ is in bijection with the set $\mathbb{Z}^N_{\downarrow}$ of non-increasing sequences of integers $\alpha=(\alpha_1\geq \ldots \geq \alpha_N)$. For all $\alpha \in \mathbb{Z}^N_{\downarrow}$, let $\pi^\alpha \in \widehat{U(N)}$ be a unitary representation in the corresponding class, acting on a vector space $E_\alpha$, and let $\chi_{\alpha}$ be its character, that is to say the function $\Tr\circ \pi^\alpha$. We will also consider the normalized character $\psi_\alpha(\cdot)=\chi_{\alpha}(\cdot)/\chi_{\alpha}(I_N)$.

Let $\mu$ be a probability measure on $U(N)$. The Fourier transform $\widehat{\mu}$ of $\mu$ is defined for all $\alpha \in \mathbb{Z}^N_{\downarrow}$ by $\widehat{\mu}(\alpha)=\int_{U(N)}\pi^\alpha(g)\ \mu(\diff g)\in \End( E_\alpha)$. Here are three properties of the Fourier transform.
%The Fourier transform allows us to compute convolution and to characterize weak convergence. More precisely, t
\begin{enumerate}
\item 
%The convolution of two probability measures $\mu$ and $\nu$ is defined to be the unique probability measure $\mu\ast \nu$ such that $\int f \diff (\mu\ast \nu)=\int f(gh) \ \mu(\diff g) \nu(\diff h)$ for all Borel and positive function on $U(N)$. By a straightforward computation, f
For all probability measures $\mu$ and $\nu$, and for all $\alpha \in \mathbb{Z}^N_{\downarrow}$ we have $\widehat{\mu\circledast \nu}(\alpha)=\widehat{\mu}(\alpha)\widehat{\nu}(\alpha)$.
\item A sequence of probability measures $(\mu_n)_{n\in \mathbb{N}}$ converges weakly to a measure $\mu$ if and only if for all $\alpha \in \mathbb{Z}^N_{\downarrow}$, the sequence $(\widehat{\mu_n}(\alpha))_{n\in \mathbb{N}}$ converges to $\widehat{\mu}(\alpha)$.
\item A probability measure $\mu$ is central, or conjugate invariant, if and only if for all $\alpha \in \mathbb{Z}^N_{\downarrow}$, 
$\widehat{\mu}(\alpha)$ is a homogeneous dilation, and in this case $\widehat{\mu}(\alpha)=(\int_{U(N)}\psi_{\alpha}(g)\ \mu(\diff g))\Id_{ E_\alpha}$.
\end{enumerate}
The following proposition gives the Fourier transform of a measure arising from a convolution semigroup.

\begin{proposition}
Let $(\mu_t)_{t\in \mathbb{R}^+}$ be a weakly continuous convolution semigroup on $U(N)$ starting at $\mu_0=\delta_e$ with generator $L$. For all $t\geq 0$, and all $\alpha \in \mathbb{Z}^N_{\downarrow}$, we have $\widehat{\mu_t}(\alpha)=e^{ tL\pi^\alpha(I_N)}$. Moreover, if $\mu$ is conjugate invariant, we have $\widehat{\mu_t}(\alpha)=e^{tL\psi_{\alpha}(I_N)}\Id_{E_\alpha}$.\label{invinvin}
\end{proposition}

%Let $D=\frac{1}{2}\displaystyle\sum_{i,j=1}^{N^2}y_{i,j}X_i X_j$. 
\begin{proof}For all $\alpha \in \mathbb{Z}^N_{\downarrow}$, we have
$\widehat{\mu_t}(\alpha)=\int_{U(N)}\pi^\alpha(g)\ \mu_t(\diff g)= \Id_{E_\alpha}+t\cdot L\pi^\alpha(I_N)+o_{t\rightarrow0}(t)$,
which implies that
$\widehat{\mu_t}(\alpha)=\lim_{s\rightarrow 0} \widehat{\mu_{s}}(\alpha)^{t/s}=e^{t L\pi^\alpha(I_N)}.$
If $\mu$ is conjugate invariant, then, for all $t\in \mathbb{R}^+$, $\mu_t$ is conjugate invariant, and we can replace $\pi^\alpha$ by $\psi_\alpha$ in the previous computation.
\end{proof}

\begin{corollary}Let $(\mu_t)_{t\in \mathbb{R}^+}$ and $(\nu_t)_{t\in \mathbb{R}^+}$ be two weakly continuous conjugate invariant convolution semigroups on $U(N)$ starting at $\mu_0=\delta_e$, with respective characteristic triplets $$(Y_0, (y_{i,j})_{1\leq i,j \leq N^2}, \Pi)\text{ and }(Y_0', (y_{i,j}')_{1\leq i,j \leq N^2}, \Pi').$$ Then, $(\mu_t\circledast \nu_t)_{t\in \mathbb{R}^+}$ is a weakly continuous convolution semigroup \label{sumwc}on $U(N)$ starting at $\mu_0=\delta_e$, with characteristic triplet $$(Y_0+Y_0', (y_{i,j}+y_{i,j}')_{1\leq i,j \leq N^2}, \Pi+\Pi').$$
\end{corollary}
\begin{proof}Remark that $ (y_{i,j}+y_{i,j}')_{1\leq i,j \leq N^2}$ is a symmetric positive semidefinite matrix and that $\Pi+\Pi'$ is a Lévy measure.  Let $L$ and $L'$ be the respective generators of $(\mu_t)_{t\in \mathbb{R}^+}$ and $(\nu_t)_{t\in \mathbb{R}^+}$ given by~\eqref{gen}. Thanks to Proposition~\ref{invinvin} and to the conjugation invariance, for all $\alpha \in \mathbb{Z}^N_{\downarrow}$, we have
$$\widehat{\mu_t\ast \nu_t}(\alpha)= \widehat{\mu_t}(\alpha)\cdot \widehat{\nu_t}(\alpha)=e^{tL\psi_{\alpha}(I_N)}e^{tL'\psi_{\alpha}(I_N)}\Id_{E_\alpha}=e^{t(L+L')\psi_{\alpha}(I_N)}\Id_{E_\alpha}$$
To conclude, observe that, for each time $t\in \R_+$, the measure at time $t$ of the weakly continuous semigroup whose characteristic triplet is $(Y_0+Y_0', (y_{i,j}+y_{i,j}')_{1\leq i,j \leq N^2}, \Pi+\Pi')$ has the same Fourier transform as $\mu_t\circledast \nu_t$.
\end{proof}

\begin{lemma}Let $\mu$ and $\nu \in \ID_\inv(U(N),\circledast)$ with characteristic triplet $(Y_0, (y_{i,j})_{1\leq i,j \leq N^2}, \Pi)$ and $(Y_0', (y_{i,j}')_{1\leq i,j \leq N^2}, \Pi')$. Then, $(Y_0+Y_0', (y_{i,j}+y_{i,j}')_{1\leq i,j \leq N^2}, \Pi+\Pi')$ is a characteristic triplet of $\mu\circledast\nu$. In particular, for all $k\in \Z$,
$(Y_0+2ik\pi I_N, (y_{i,j})_{1\leq i,j \leq N^2}, \Pi)$ is also a characteristic triplet of $\mu$.\label{tripletpi}
\end{lemma}
\begin{proof}
The first assertion follows from Corollary~\ref{sumwc}. For the second assertion, we remark that $ (\delta_{e^{2ikt\Pi}I_N})_{t\in \mathbb{R}^+}$ is a weakly continuous convolution semigroup with characteristic triplet $(2ik\pi, 0,0)$. By consequence,  $(Y_0+2ik\pi I_N, (y_{i,j})_{1\leq i,j \leq N^2}, \Pi)$ is a characteristic triplet of $\mu\circledast \delta_{e^{2ik\pi}I_N}=\mu$.
\end{proof}

We are now ready to prove Proposition-Definition~\ref{ththree}.

\subsection{Proof of Proposition-Definition~\ref{ththree}}\label{proofdef}First of all, we remark that the sine function is bounded and $\sin(x)- x\sim_{x\rightarrow 0} x^3/6$, which implies that $\int_{\mathcal{H}_N}(\sin(x)-1_B(x) x) \ \Pi(\diff x)$ exists.

We start by proving the first item. Let $\mu\in \ID(\mathcal{H}_N,\ast)$. Let us denote by $L_{\mu}$ the generator of $\mu$ and by $L_{\mathcal{E}_N(\mu)}$ the generator of $\mathcal{E}_N(\mu)$. Let $\alpha \in \mathbb{Z}^N_{\downarrow}$. We have $$\widehat{ \e_{\ast} (\mu^{*\frac{1}{n}})}(\alpha)=\int_{\mathcal{H}_N}\pi^\alpha(\e(x))\ \mu^{\ast 1/n}(\diff x)= \Id_{E_\alpha}+ L_{\mu}(\pi^\alpha\circ \e)(0)/n+o_{n\rightarrow \infty}(1/n),$$ which implies that $\lim_{n\rightarrow \infty} \widehat{(\e_\ast (\mu^{*\frac{1}{n}}))^{\circledast n}}(\alpha)=\lim_{n\rightarrow \infty} \left(\widehat{\e_{\ast} (\mu^{*\frac{1}{n}})}(\alpha)\right)^n=e^{L_{\mu}(\pi^\alpha\circ \e)(0)}$. Let us compute
\begin{align*}L_{\mu}(\pi^\alpha\circ \e)(0)=&\partial_{X_0}(\pi^\alpha\circ \e)(0)+\frac{1}{2}\displaystyle\sum_{i,j=1}^{N^2}y_{i,j}\partial_{X_i}\partial_{X_j} (\pi^\alpha\circ \e)(0)\\
 &+\int_{\mathcal{H}_N}\pi^\alpha(e^{i(x+0)})-\pi^\alpha(e^{i0})-  1_B(x)\partial_{x} (\pi^\alpha\circ \e)(0)\ \Pi(\diff x) .
 \end{align*}
Recall that, for all $Y\in \frak{u}(N)$, $Y^l$ is the left invariant vector field on $U(N)$ induced by $Y$. Using the fact that, for all $x\in \mathcal{H}_N$, $\partial_{x} (\pi^\alpha\circ \e)(0)=\left.\frac{\diff}{\diff t}\right|_{t=0}\pi^\alpha(e^{itx})=(ix)^l \pi^\alpha(I_N)$, we infer
\begin{align*}L_{\mu}(\pi^\alpha\circ \e)(0)=&(iX_0)^l(\pi^\alpha)(I_N)+\frac{1}{2}\displaystyle\sum_{i,j=1}^{N^2}y_{i,j}Y_i^lY_j^l(\pi^\alpha)(I_N) \\
 &+\int_{\mathcal{H}_N}\pi^\alpha(e^{ix})-\Id_{E_\alpha}- 1_B(x) (ix)^l \pi^\alpha(I_N)\ \Pi(\diff x) \\
% \end{align*}
% which can be rewritten
% 
%\begin{align*}L_{\mu}(\pi^\alpha\circ \e)(0)
=&(iX_0)^l(\pi^\alpha)(I_N)+\int_{\mathcal{H}_N} (i\sin(x)-i1_B(x) x)^l \pi^\alpha(I_N)\ \Pi(\diff x) \\
&+\frac{1}{2}\displaystyle\sum_{i,j=1}^{N^2}y_{i,j}Y_i^lY_j^l(\pi^\alpha)(I_N) +\int_{\mathcal{H}_N}\pi^\alpha(\e(x))-\Id_{E_\alpha}-\left(i \Im(\e(x))\right)^l\pi^\alpha(I_N)  \ \Pi(\diff x)\\
=&L_{\mathcal{E}_N(\mu)}\pi^\alpha(I_N).
\end{align*}
Finally, for all $\alpha \in \mathbb{Z}^N_{\downarrow}$, the sequence $\widehat{(\e_\ast (\mu^{*\frac{1}{n}}))^{\ast n}}(\alpha)$ converges to $e^{L_{\mathcal{E}_N(\mu)}\pi^\alpha(I_N)} =\widehat{\mathcal{E}_N(\mu)}(\alpha)$ and consequently the sequence $(\e_{\ast} (\mu^{*\frac{1}{n}}))^{\circledast n}$ converges to $\mathcal{E}_N(\mu)$.

For the proof of the second item, we use the Fourier transform of a measure in $ \ID(\mathcal{H}_N,\ast)$, which is given by the following proposition.
\begin{proposition}[\cite{Sato1999}]Let $\mu\in \ID(\mathcal{H}_N,\ast)$ with characteristic triplet $(X_0, (y_{i,j})_{1\leq i,j \leq N^2}, \Pi)$. We have $\int_{\mathcal{H}_N} e^{i\Tr(xy)} \mu^{\ast t}(\diff x)=\exp\left( t \varphi_\mu(y)\right)$ with\label{Fourier}
$$\varphi_\mu(y)=i\Tr(X_0y)-\frac{1}{2}\displaystyle\sum_{i,j=1}^{N^2}y_{i,j}\Tr(X_iy)\Tr(X_jy)+\int_{\mathcal{H}_N}e^{i\Tr(xy)}-1-i 1_B(x)\Tr(xy)\ \Pi(\diff x).  $$
\end{proposition}

Let $\mu\in  \ID_{\inv}(\mathcal{H}_N,\ast)$. We claim that, for all $t\geq 0$, $\mu^{\ast t}\in \ID_{\inv}(\mathcal{H}_N,\ast)$. Assuming for a moment that this claim is proved, let us explain how it leads to the result: in this case, each measure $(\e_{\ast}(\mu^{\ast 1/n}))^{\circledast n}$ is conjugate invariant and so is the limit $\mathcal{E}_N(\mu)$. In addition, for all $\mu$, $\nu\in \ID_{\inv}(\mathcal{H}_N,\ast)$, the characteristic triplets of $\mathcal{E}_N(\mu\ast \nu)$ and of $\mathcal{E}_N(\mu)\circledast\mathcal{E}_N( \nu)$ coincide thanks to Corollary~\ref{sumwc}, and thus $\mathcal{E}_N(\mu\ast \nu)=\mathcal{E}_N(\mu)\circledast\mathcal{E}_N( \nu) .$

Thus, it remains to prove that, for all $t\geq 0$, $\mu^{\ast t}\in \ID_{\inv}(\mathcal{H}_N,\ast)$. For this, we prove that the Fourier transform of $\mu^{\ast t}$ is conjugate invariant. 
%For all $y\in \mathcal{H}_N$, let us denote by $\varphi_y\in C^2(\mathcal{H}_N)$ the map defined by $\varphi_y(x)=e^{i\Tr(xy)}$. Fourier transform of $\mu$ for the scalar product $(x,y)\mapsto \Tr(xy)$ is given by
%$$\widehat{\mu}(y)=\int_{\mathcal{H}_N}\varphi_y \diff\mu.$$
Firstly, $\varphi_\mu$ is conjugate invariant. Indeed, for all $g\in U(N)$, we have
$$\exp\circ \varphi_\mu(gyg^*)=\int_{\mathcal{H}_N}e^{i\Tr(xgyg^*)} \diff\mu(x)=\int_{\mathcal{H}_N}e^{i\Tr(g^*xgy)} \diff\mu(x)=\int_{\mathcal{H}_N}e^{i\Tr(xy)} \diff\mu(x)=\exp\circ \varphi_\mu(y).$$
%Because $\widehat{\mu}(y)=\lim_{n\to 0}\widehat{(\mu^{\ast 1/n})^{\ast n}}(y)=\lim_{n\to 0}\left(\widehat{\mu^{\ast 1/n}}(y)\right)^n$ and $\widehat{\mu^{\ast 1/n}}(y)=1+L\varphi_y(0)/n+o(1/n)$, we have $$\widehat{\mu}(y)=\exp(L\varphi_y(0)). $$
We deduce that $\varphi_\mu$ is conjugate invariant since it is continuous and $\exp\circ \varphi_\mu$ is conjugate invariant.
Consequently, $\int_{\mathcal{H}_N}e^{i\Tr(xg\cdot g^*)} \diff\mu^{\ast t}(x)=\exp(t\varphi_\mu(\cdot))$ is conjugate invariant, which is sufficient to conclude.

\section{Random matrices}\label{secsix}
In this last section, we shall define the mappings $\Pi_N$ and $\Gamma_N$. Then we prove Theorem~\ref{convwe}, and in particular our main result, the weak convergence in expectation of the empirical spectral measures of random matrices distributed over $\Gamma_N(\mu)$ for some $\mu\in \mathcal{ID}(\mathbb{U},\boxtimes)$ (see Theorem~\ref{thsixp}). We finish the section by the proof of Theorem~\ref{whole}.

%The proof itself of Theorem~\ref{thsixp} is interesting at least for two reasons. It is the first time that the free log-cumulants, originated in~\cite{Mastnak2010},  are used for proving an asymptotic result of random matrices. Secondly, the proof relies upon a key object, the symmetric group, which is linked to both the combinatorics of free probability theory, and the computation of conjugate-invariant measures on $U(N)$: this proof reinforces the deep link between those two fields, present for example in~\cite{Collins2003,Collins2004,LEVY2008}.

\subsection{The matrix model $\Pi_N$}\label{pin}Recall that the covariance matrix, which corresponds to the diffuse part of an infinitely divisible measure, depends on the choice of a basis of $\mathcal{H}_N$ (see Section~\ref{secfivebis}). In this article, we fixed an orthonormal basis $\left\{X_1,\ldots,X_{N^2}\right\}$ of $\mathcal{H}_N$ such that $X_{N^2}=\frac{1}{\sqrt{N}}I_N$.
\begin{definition}
Let $\mu\in \ID(\R,\boxplus)$ and let $(\eta,a,\rho)$ be its \mbox{$\boxplus$-characteristic} triplet. The distribution $\Pi_N(\mu)\in \ID_{\inv}(\mathcal{H}_N,\ast)$ is defined to be the infinitely divisible measure with characteristic triplet $\left(\eta I_N, a_N,  \rho_N\right)$, where $a_N$ is the $N^2\times N^2$-matrix
$$a_N=\left(\begin{array}{cccc}\frac{a}{N+1} &  &  & 0 \\ & \ddots &  &  \\ &  & \frac{a}{N+1} &  \\0 &  &  & a\end{array}\right) ,$$ and $ \rho_N$ is the Lévy measure on $\mathcal{H}_N$ which is the push-forward measure of $N \rho\otimes \Haar$ by the mapping from $\R\times U(N)$ to $\mathcal{H}_N$ defined by $$(x,g)\mapsto g\left(\begin{array}{cccc}x & 0 & \cdots & 0 \\0 & 0 & \ddots & \vdots \\\vdots & \ddots & \ddots & 0 \\0 & \cdots & 0 & 0\end{array}\right)g^*.$$
\end{definition}
%
% defined as follows: for all Borel positive function $f$ on $\mathcal{H}_N$, we have
%$$\int _{\mathcal{H}_N}f\diff  \rho_N=N\int_\R \int_{U(N)} f\left(g\left(\begin{array}{cccc}x & 0 & \cdots & 0 \\0 & 0 & \ddots & \vdots \\\vdots & \ddots & \ddots & 0 \\0 & \cdots & 0 & 0\end{array}\right)g^*\right)\diff g\diff \rho(x) $$
%where the integration over $U(N)$ is taken with respect to the Haar measure of $U(N)$. In other words, $\rho_N$ is the push-forward measure of $N \rho\otimes \Haar$ by the function from $\R\times U(N)$ to $\mathcal{H}_N$ given by $(x,g)\mapsto g\left(\begin{array}{cccc}x & 0 & \cdots & 0 \\0 & 0 & \ddots & \vdots \\\vdots & \ddots & \ddots & 0 \\0 & \cdots & 0 & 0\end{array}\right)g^*$.
%The triplet$\left(\eta I_N, a_N,  \rho_N\right)$ is the characteristic triplet of unique infinitely divisible measure of $ \ID_{\inv}(\mathcal{H}_N,\ast)$, which will be denoted by $\Pi_N(\mu)$. 
The application $\Pi_N:\ID(\R,\boxplus)\to \ID_{\inv}(\mathcal{H}_N,\ast)$ is obviously a homomorphism of semigroups and we have $\Pi_1=\Lambda^{-1}$. Moreover, $\Pi_N$ is a matricial model for $\ID(\R,\boxplus)$ in the sense of the following theorem.
%
%If $(\eta,a,\rho)$ is the $\ast$-characteristic triplet of $\mu\in \ID(\R,\ast)$, this measure is denoted by $\Lambda_N(\mu)$.
%
%If $(\eta,a,\rho)$ is the $\boxplus$-characteristic triplet of $\mu\in \ID(\R,\boxplus)$, this measure is denoted by $\Pi_N(\mu)$.
%
\begin{theorem}[\cite{Benaych-Georges2005,Cabanal-Duvillard2005}]Let $\mu\in \mathcal{ID}(\R,\boxplus)$. For all $N\in \N^*$, let $H_N$ be a random matrix whose law is $\Pi_N(\mu)$, and let $\hat{\mu}_{H_N}$ be its empirical spectral measure, that is to say\label{thsix}
$$\hat{\mu}_{H_N}=\frac{1}{N}\sum_{\substack{\textup{eigenvalue }\lambda\textup{ of }H_N\\ \textup{(with multiplicity)} }}\delta_{\lambda} .$$
Then, the measures $\hat{\mu}_{H_N}$ converge weakly to $\mu$ in probability when $N$ tends to $\infty$.

\end{theorem}

In~\cite{Benaych-Georges2005,Cabanal-Duvillard2005}, the model is in fact defined starting from a measure $\mu\in \ID(\R,\ast)$. More precisely, for all $\mu\in \ID(\R,\ast)$ with $\ast$-characteristic triplet $(\eta,a,\rho)$ and Lévy exponent $$\varphi_\mu(\theta)=\left(i\eta \theta -\frac{1}{2}a\theta^2+\int_\R (e^{i\theta x}-1-i\theta x 1_{[-1,1]}(x))\diff \rho(x)\right),$$Benaych-Georges and Cabanal-Duvillard defined $\Lambda_N(\mu)\in \ID_{\inv}(\mathcal{H}_N,\ast)$ by its Fourier transform: for $x,y\in \mathcal{H}_N$, we have $$\int_{\mathcal{H}_N} e^{i\Tr(xy)} \Lambda_N(\mu)(\diff x)=\exp\left(  \varphi_{\Lambda_N(\mu)}(y)\right)$$where
$\varphi_{\Lambda_N(\mu)}(y)= N\E[\varphi_\mu(\langle u,y u\rangle)]$, with $u$ uniformly distributed on the unit sphere of $\C^N$. More explicitly,
$$ \varphi_{\Lambda_N(\mu)}(y)=i\eta \Tr(y)-\frac{a}{2(N+1)}\left(\Tr(y)\Tr(y)+\Tr(y)^2\right)+\int_{\mathcal{H}_N}e^{i\Tr(xy)}-1-i 1_B(x)\Tr(xy)\ \Pi(\diff x).$$ Using Proposition~\ref{Fourier}, we see that it is exactly the Fourier transform of 
%$i\eta \theta -\frac{1}{2}a\theta^2+\int_\R (e^{i\theta x}-1-i\theta x 1_{[-1,1]}(x))\diff \rho(x)$
the infinitely divisible measure of $ \ID_{\inv}(\mathcal{H}_N,\ast)$ with characteristic triplet $\left(\eta, a_N,  \rho_N\right)$. Consequently, we have $\Lambda_N=\Pi_N \circ \Lambda$, or $\Pi_N=\Lambda_N \circ \Lambda^{-1} $ which can be expressed as the commutativity of the following diagram
$$\xymatrix{  \ID(\R,\ast) \ar[r]^-*+{\Lambda_N}\ar@/_2pc/[rr]_-*+{\Lambda}  & \ID_{\inv}(\mathcal{H}_N,\ast)&  \ID(\R,\boxplus) \ar[l]_-*+{\Pi_N}  }.$$ Nevertheless, we prefer to use $\Pi_N$ which turns out to be more suitable for our present purposes (see~Theorem~\ref{convwe}). One can consult also~\cite{Dominguez2012,Dominguez2013} for further information about this model.

\subsection{The matrix model $\Gamma_N$}\label{gamman}Here again, observe that the data of a covariance matrix of $\frak{u}(N)$ depends on the basis chosen, and recall that we fixed an orthonormal basis $\left\{Y_1,\ldots,Y_{N^2}\right\}$ of $\frak{u}(N)$ such that $Y_{N^2}=\frac{i}{\sqrt{N}}I_N$ (see Section~\ref{bloubi}).
\begin{definition}\label{gammand}Let $\mu\in \mathcal{ID}(\mathbb{U},\boxtimes)$  and let $(\omega,b,\upsilon)$ be its \mbox{$\boxtimes$-characteristic} triplet. The distribution $\Gamma_N(\mu)\in \ID_{\inv}(U(N),\circledast)$ is defined to be the infinitely divisible measure with characteristic triplet $\left(\Log(\omega) I_N, b_N,  \upsilon_N\right)$, where $\Log$ is the principal logarithm, $b_N$ is the $N^2\times N^2$-matrix
$$b_N=\left(\begin{array}{cccc}\frac{b}{N+1} &  &  & 0 \\ & \ddots &  &  \\ &  & \frac{b}{N+1} &  \\0 &  &  & b\end{array}\right) ,$$ and $ \upsilon_N$ is the Lévy measure on $U(N)$ which is the push-forward measure of $N \upsilon\otimes \Haar$ by the mapping from $\U\times U(N)$ to $U(N)$ defined by $$(\zeta,g)\mapsto g\left(\begin{array}{cccc}\zeta & 0 & \cdots & 0 \\0 & 1 & \ddots & \vdots \\\vdots & \ddots & \ddots & 0 \\0 & \cdots & 0 & 1\end{array}\right)g^*.$$
We also define $\Gamma_N(\lambda)$ to be the Haar measure of $U(N)$ when $\lambda$ is the Haar measure of $\U$.

\end{definition}
%Let $\mu\in \mathcal{ID}(\mathbb{U},\boxtimes)$ with characteristic triplet $(\omega,b,\upsilon)$. We define $b_N$ to be the $N^2\times N^2$-matrix
%$$b_N=\left(\begin{array}{cccc}\frac{b}{N+1} &  &  & 0 \\ & \ddots &  &  \\ &  & \frac{b}{N+1} &  \\0 &  &  & b\end{array}\right) .$$ We denote by $ \upsilon_N$ the Lévy measure on $U(N)$ defined as follows: for all Borel positive function $f$ on $U(N)$, we have
%$$\int _{U(N)}f\diff  \upsilon_N=N \int_\U \int_{U(N)} f\left(g\left(\begin{array}{cccc}\zeta & 0 & \cdots & 0 \\0 & 1 & \ddots & \vdots \\\vdots & \ddots & \ddots & 0 \\0 & \cdots & 0 & 1\end{array}\right)g^*\right)\diff g\diff \upsilon(\zeta) $$
%where the integration over $U(N)$ is taken with respect to the Haar measure of $U(N)$.  In other words, $\upsilon_N$ is the push-forward measure of $N \upsilon \otimes \Haar$ by the function from $\U \times U(N)$ to $U(N)$ given by $(\zeta,g)\mapsto g\left(\begin{array}{cccc}\zeta & 0 & \cdots & 0 \\0 & 1 & \ddots & \vdots \\\vdots & \ddots & \ddots & 0 \\0 & \cdots & 0 & 1\end{array}\right)g^*$.
%
%
%We define $\Gamma_N(\mu)$ to be the measure of $\mathcal{ID}_\inv(U(N),\ast)$ with characteristic triplet $\left(\Log(\omega) I_N, b_N,  \upsilon_N\right)$, and we extend $\Gamma_N$ to $\mathcal{ID}(\mathbb{U},\boxtimes)$ by defining $\Gamma_N(\lambda)$ to be the Haar measure of $U(N)$ for $\lambda$ being the Haar measure of $\U$.

From this definition, we deduce right now the second half of Theorem~\ref{convwe}, as a consequence of the following propositions.

\begin{proposition}
For all $\mu$ and $\nu\in \ID(\U,\boxtimes)$, we have $\Gamma_N(\mu\boxtimes \nu)=\Gamma_N(\mu)\circledast \Gamma_N( \nu)$ .
\end{proposition}
\begin{proof}
Let $\mu$ and $\nu\in \ID(\U,\boxtimes)$. If $\mu$ or $\nu$ is equal to $\lambda$, we have $\mu\boxtimes \nu=\lambda$. In this case, $\Gamma_N(\mu)$ or $\Gamma_N(\nu)$ is the Haar measure on $U(N)$ and consequently, $\Haar=\Gamma_N(\mu\boxtimes \nu)=\Gamma_N(\mu)\circledast \Gamma_N( \nu).$

If $\mu,\nu\in \ID(\U,\boxtimes)\cap \Ms$, with respective $\boxtimes$-characteristic triplets $(\omega_1,b_1,\upsilon_1)$ and $(\omega_2,b_2,\upsilon_2)$, the measure $\mu\boxtimes\nu\in  \Ms$ is a {$\boxtimes$-infinitely} divisible measure with $\boxtimes$-characteristic triplet $(\omega_1\omega_2,b_1+b_2,\upsilon_1+\upsilon_2)$. We denote by $(Y_0,(y_{i,j})_{1\leq i,j \leq N^2},\Pi)$ and $(Y_0',(y_{i,j}')_{1\leq i,j \leq N^2},\Pi')$ the respective characteristic triplets of $\Gamma(\mu\boxtimes \nu)$ and $\Gamma(\mu)\circledast \Gamma( \nu)$. It is straightforward to verify that $((y_{i,j})_{1\leq i,j \leq N^2},\Pi)=((y'_{i,j})_{1\leq i,j \leq N^2},\Pi')$, and it remains to compare $Y_0$ and $Y_0'$. We have $Y_0=\Log(\omega_1\omega_2)I_N$ and $Y_0=(\Log(\omega_1)+\Log(\omega_2))I_N$. As a consequence, $Y_0$ and $Y_0'$ differ by a multiple of $2i\pi I_N$. Using Lemma~\ref{tripletpi}, we deduce that $(Y_0,(y_{i,j})_{1\leq i,j \leq N^2},\Pi)$ and $(Y_0',(y_{i,j}')_{1\leq i,j \leq N^2},\Pi')$ are characteristic triplets of the same measure. In other words, $\Gamma(\mu\boxtimes \nu)=\Gamma(\mu)\circledast \Gamma( \nu)$.
\end{proof}

\begin{proposition}
For all $\mu\in \ID(\R,\boxplus)$, we have $\Gamma_N\circ \e_{\boxplus}(\mu)=\mathcal{E}_N\circ \Pi_N(\mu)$.\label{thfive}
\end{proposition}

\begin{proof}Let $(\eta,a,\rho)$ be the $\boxplus$-characteristic triplet of $\mu$. We denote by $(Y_0,(y_{i,j})_{1\leq i,j \leq N^2},\Pi)$ and $(Y_0',(y_{i,j}')_{1\leq i,j \leq N^2},\Pi')$ the respective characteristic triplets of $\Gamma_N\circ \e_{\boxplus}(\mu)$ and $\mathcal{E}_N\circ \Pi_N(\mu)$. We remark first that, following the definitions, $$(y_{i,j})_{1\leq i,j \leq N^2}=(y_{i,j}')_{1\leq i,j \leq N^2}=\left(\begin{array}{cccc}\frac{a}{N+1} &  &  & 0 \\ & \ddots &  &  \\ &  & \frac{a}{N+1} &  \\0 &  &  & a\end{array}\right)$$ and $\Pi=\Pi'=M_{|U(N)\setminus \{I_N\}}$
where $M$ is the push-forward measure of $N \rho \otimes \Haar$ by the mapping from $\R \times U(N)$ to $U(N)$ given by $$(x,g)\mapsto g\left(\begin{array}{cccc}e^{ix} & 0 & \cdots & 0 \\0 & 1 & \ddots & \vdots \\\vdots & \ddots & \ddots & 0 \\0 & \cdots & 0 & 1\end{array}\right)g^*.$$
To conclude, it remains to compare $Y_0$ and $Y_0'$. We have
$$Y_0=\Log\circ \exp\left(i\eta+i\int_{\mathbb{R}}(\sin(x)-1_{[-1,1]}(x) x) \ \rho(\diff x)\right)I_N$$ and 
\begin{align*}
Y_0'&=i\eta I_N+i\int_{\mathcal{H}_N}(\sin(x)-1_U(x) x) \ \diff \rho_N( x) \\
&=i\eta I_N+iN\int_{\mathbb{R}}\int_{U(N)}g\left(\begin{array}{cccc}(\sin(x)-1_{[-1,1]}(x) x) &  &  & 0 \\ & 0 &  &  \\ & & \ddots &  \\0 &  & & 0\end{array}\right)g^* \ \diff g \rho(\diff x)\\
&=i\eta I_N+iN\int_{\mathbb{R}}\frac{1}{N}(\sin(x)-1_{[-1,1]}(x) x) \rho(\diff x)\\
&=\left(i\eta+i\int_{\mathbb{R}}(\sin(x)-1_{[-1,1]}(x) x) \ \rho(\diff x)\right)I_N,
\end{align*}
where we have used that $E(A)=\frac{1}{N}\Tr(A)I_N$ (see Example~\ref{undeux}) for the integration with respect to the Haar measure of $U(N)$. The difference between $Y_0$ and $Y_0'$ is a multiple of $2i\pi I_N$. Using Lemma~\ref{tripletpi}, we deduce that $(Y_0,(y_{i,j})_{1\leq i,j \leq N^2},\Pi)$ and $(Y_0',(y_{i,j}')_{1\leq i,j \leq N^2},\Pi')$ are characteristic triplets of the same measure. In other words, $\Gamma_N\circ \e_{\boxplus}(\mu)=\mathcal{E}_N\circ \Pi_N(\mu)$.
\end{proof}

\subsection{The large-$N$ limit}We are now ready to prove the first half of Theorem~\ref{convwe}, which is a corollary of the following theorem.

\begin{theorem}Let $\mu\in \mathcal{ID}(\mathbb{U},\boxtimes)$. For all $N\in \N^*$, let $U_N$ be a random matrix whose law is $\Gamma_N(\mu)$.\label{thsixp}
%, and whose empirical spectral measure is
%$$\hat{\mu}_{U^{(N)}}=\frac{1}{N}\sum_{\substack{\lambda\text{ eigenvalue of }U^{(N)}\\ \text{(with multiplicity)} }}\delta_{\lambda} $$
%Then, $\hat{\mu}_{U^{(N)}}\wn\mu$ a.s..
%
%More generally, f
For all polynomials $P_1,\ldots,P_k\in \C[X]$, we have,\label{fact}
$$\lim_{N\to \infty}\E\left[\frac{1}{N}\Tr\left(P_1(U_N)\right)\cdots \frac{1}{N}\Tr\left(P_k(U_N)\right) \right]=\int_\U P_1\diff\mu \cdots \int_\U P_k\diff\mu.$$
\end{theorem}

\begin{proof} If $\mu$ is the Haar measure $\lambda$ of $\U$, then $\Gamma_N(\mu)$ is the Haar measure on $U(N)$ for which the result is well-known. Let us assume that $\mu\in \mathcal{ID}(\mathbb{U},\boxtimes)\cap \Ms$, and let $(\omega,b, \upsilon)$ be its \mbox{$\boxtimes$-characteristic} triplet. Thanks to Definition~\ref{gammand}, we know that a characteristic triplet of $\Gamma_N(\mu)$ is given by $$\left(i y_0I_N,\left(\begin{array}{cccc}\alpha &  &  & 0 \\ & \ddots &  &  \\ &  & \alpha &  \\0 &  &  & \beta\end{array}\right), \Pi\right)$$ where $y_0=-i\Log(\omega)$, $\alpha=b/(N+1)$, $\beta=b$ and $\Pi$ is the Lévy measure obtained from $\upsilon$ as in Definition~\ref{gammand}.

By linearity, it suffices to prove the result for monomials. Let $l_1,\ldots, l_k\in \N$. We want to prove that
$$\lim_{N\to \infty}\E\left[\frac{1}{N}\Tr\left(U_N^{l_1}\right)\cdots \frac{1}{N}\Tr\left(U_N^{l_k}\right) \right]=m_{l_1}(\mu)\cdots m_{l_k}(\mu).$$
We will prove the result under the following form: for all $\sigma\in \mathfrak{S}_n$,
\begin{equation*}
\lim_{N\to \infty}\E\left[N^{-\ell(\sigma)}\prod_{c\text{ cycle of }\sigma}\Tr(U_N^{\sharp c})\right]=\prod_{c\text{ cycle of }\sigma}m_{\sharp c}(\mu).
\end{equation*}
We observe that, for all $U\in U(N)$ and $\sigma\in \mathfrak{S}_n$, we have
\begin{equation}
\prod_{c\text{ cycle of }\sigma}\Tr(U^{\sharp c})=\Tr_{(\C^N)^{\otimes n}}\left(U^{\otimes n}\circ\rho^{\mathfrak{S}_n}_N(\sigma) \right).\label{trtr}
\end{equation}
In order to use Proposition~\ref{momsemi}, we define $\tilde{L}_N\in \C[\mathfrak{S}_n]$ by
\begin{align*}
\tilde{L}_N=& \left(niy_0-\frac{n^2}{N}\frac{\beta}{2}+\left(\frac{n^2}{N}-nN\right)\frac{\alpha}{2}+\frac{n}{N}\int_{U(N)}\Tr\left( \Re(g)-1\right) \Pi(\diff g)\right)1_{\mathfrak{S}_n} -\alpha \displaystyle\sum_{\tau\in \mathcal{T}_n}\tau\\
%&&+\displaystyle\sum_{i,j=1}^{N^2}y_{i,j}\left(\frac{\Tr(Y_iY_j)}{N^2-1}-\frac{\Tr(Y_i)\Tr(Y_j)}{N^3-N}\right)\cdot \displaystyle\sum_{\tau\in T_n}\tau\\
&+\displaystyle\sum_{\substack{2\leq m \leq n\\1\leq k_1<\ldots<k_m \leq n}}\displaystyle\sum_{\pi', \pi \in \mathfrak{S}_m}\Wg(\pi'^{-1} \pi)\int_{U(N)}\displaystyle\prod_{c \text{ cycle of }\sigma}\Tr\left((g-1)^{\sharp c}\right) \ \Pi(\diff g)\cdot \iota_{k_1,\ldots,k_m}(\pi)\\
=&\left(n\Log(\omega)-\frac{n^2}{N}b+\left(\frac{n^2}{N}-nN\right)\frac{b}{2(N+1)}+n\int_{\U}\left( \Re(\zeta)-1\right) \upsilon(\diff \zeta)\right)1_{\mathfrak{S}_n} -\frac{b}{N+1} \displaystyle\sum_{\tau\in \mathcal{T}_n}\tau\\
%&&+\displaystyle\sum_{i,j=1}^{N^2}y_{i,j}\left(\frac{\Tr(Y_iY_j)}{N^2-1}-\frac{\Tr(Y_i)\Tr(Y_j)}{N^3-N}\right)\cdot \displaystyle\sum_{\tau\in T_n}\tau\\
&+\displaystyle\sum_{\substack{2\leq m \leq n\\1\leq k_1<\ldots<k_m \leq n}}\displaystyle\sum_{\pi', \pi \in \mathfrak{S}_m}\Wg(\pi'^{-1} \pi)N\int_{\U}\displaystyle(\zeta-1)^{m} \ \upsilon(\diff \zeta)\cdot \iota_{k_1,\ldots,k_m}(\pi).
\end{align*}
Using Proposition~\ref{momsemi}, we have
\begin{align*} \E\left[N^{-\ell(\sigma)}\prod_{c\text{ cycle of }\sigma}\Tr(U_N^{\sharp c})\right]&=N^{-\ell(\sigma)}\Tr_{(\C^N)^{\otimes n}}\left( \E\left[U_N^{\otimes n}\right]\circ\rho^{\mathfrak{S}_n}_N(\sigma) \right)\\
&=N^{-\ell(\sigma)}\Tr_{(\C^N)^{\otimes n}}\left(\rho^{\mathfrak{S}_n}_N(e^{\tilde{L}_N}\sigma)\right).\end{align*}
From~\eqref{trtr}, we deduce also that, for all $\sigma\in \mathfrak{S}_n$, we have
$$\Tr_{(\C^N)^\otimes n}\left(\rho^{\mathfrak{S}_n}_N(\sigma) \right)=N^{\ell(\sigma)}.$$
We denote by $N^\ell$ (resp. $N^{-\ell}$) the linear operator on $\C[\mathfrak{S}_n]$ defined by $N^\ell(\sigma)  =N^{\ell(\sigma)}\sigma$ (resp. $N^{-\ell}(\sigma)  =N^{-\ell(\sigma)}\sigma$) and by $\phi$ the linear functional defined by $\phi(\sigma)=1$. This way, we have $\Tr_{(\C^N)^{\otimes n}}\circ\rho^{\mathfrak{S}_n}_N=\phi\circ N^{\ell}$. Let us also denote by $T_N$ the linear operator on $\C[\mathfrak{S}_n]$ of multiplication by $\tilde{L}_N$, defined by $T_N(\sigma)  =\tilde{L}_N\sigma$.
%, & T(\sigma) & =\sum_{\substack{c\text{ cycle of }\mathfrak{S}_n\\c\sigma\preceq \sigma}} L\kappa_{d(\sigma,c\sigma)}\left(\mu \right)\cdot c\sigma,\end{array}
%$$
We can rewrite
\begin{align*}
 \E\left[N^{-\ell(\sigma)}\prod_{c\text{ cycle of }\sigma}\Tr(U_N^{\sharp c})\right]&=\Tr_{(\C^N)^{\otimes n}}\left(\rho^{\mathfrak{S}_n}_N(e^{\tilde{L}_N}N^{-\ell(\sigma)}\sigma)\right)\\
 &=\phi\left(N^{\ell}e^{T_N}N^{-\ell}(\sigma)\right)\\
 &=\phi\left(e^{N^{\ell}T_NN^{-\ell}}(\sigma)\right).
\end{align*}
We take the limit with the help of the following lemma. Recall that $(L\kappa_n(\mu))_{n\in \N^*}$ are the free log-cumulants of $\mu$ (see Section~\ref{flc}), which are given by\begin{enumerate}
\item $L\kappa_1(\mu)=\Log(\omega)-b/2+\int_\U \left( \Re(\zeta)-1\right) \diff \upsilon(\zeta)$,
\item $L\kappa_2(\mu)=-b+\int_\U (\zeta -1)^2\diff \upsilon(\zeta)$
\item  and $L\kappa_n(\mu)=\int_\U (\zeta -1)^n\diff \upsilon(\zeta)$ for all $n\geq2$.
\end{enumerate}
\begin{lemma}
When $N$ tends to $\infty$, the operator $N^{\ell}T_NN^{-\ell}$ converges to an operator $T$ which is such that, for all $\sigma\in \mathfrak{S}_n$,\label{superlemme}
$$T(\sigma)  =n L\kappa_1(\mu)\cdot \sigma+\displaystyle\sum_{\substack{2\leq m \leq n\\c\ m\textup{-cycle of }\mathfrak{S}_n\\c\sigma\preceq \sigma}} L\kappa_{m}\left(\mu \right)\cdot c\sigma. $$
\end{lemma}

\begin{proof}
%Proposition~\ref{loc} gives us the data of $(L\kappa_n(\mu))_{n\in \N^*}$. We have
%\begin{enumerate}
%\item $L\kappa_1(\mu)=\Log(\omega)-b/2+\int_\U \left( \Re(\zeta)-1\right) \diff \upsilon(\zeta)$,
%\item $L\kappa_2(\mu)=-b+\int_\U (\zeta -1)^2\diff \upsilon(\zeta)$ and
%\item $L\kappa_n(\mu)=\int_\U (\zeta -1)^n\diff \upsilon(\zeta)$ for all $n\geq2$.
%\end{enumerate}
We shall prove that, for a fixed $\sigma\in \mathfrak{S}_n$, $\lim_{N\to \infty}N^{\ell}T_NN^{-\ell}(\sigma)=T(\sigma)$. Let us compute
$$N^{\ell}T_NN^{-\ell}(\sigma)=N^{\ell(\sigma)}N^{\ell}(\tilde{L}\sigma).$$
%We recall~\eqref{ltilde}, where we replace $y_0,\alpha,\beta$ and $\Pi$ by their values:
%\begin{multline*}
%\tilde{L}= \left(ni\Log(\omega)-\frac{n^2}{N}b+\left(\frac{n^2}{N}-nN\right)\frac{b}{2(N+1)}+n\int_{\U}\left( \Re(\zeta)-1\right) \upsilon(\diff \zeta)\right)1_{\mathfrak{S}_n} -\frac{b}{N+1} \displaystyle\sum_{\tau\in \mathcal{T}_n}\tau\\
%%&&+\displaystyle\sum_{i,j=1}^{N^2}y_{i,j}\left(\frac{\Tr(Y_iY_j)}{N^2-1}-\frac{\Tr(Y_i)\Tr(Y_j)}{N^3-N}\right)\cdot \displaystyle\sum_{\tau\in T_n}\tau\\
%+\displaystyle\sum_{\substack{2\leq m \leq n\\1\leq k_1<\ldots<k_m \leq n}}\displaystyle\sum_{\pi', \pi \in \mathfrak{S}_m}\Wg({\pi'}^{-1} \pi)N\int_{\U}\displaystyle(\zeta-1)^{m} \ \upsilon(\diff \zeta)\cdot \iota_{k_1,\ldots,k_m}(\pi).
%\end{multline*}
Replacing $\tilde{L}$ by its value gives us $N^{\ell(\sigma)}N^{\ell}(\tilde{L}\sigma)=(I+I\!I+I\!I\!I)\sigma$, with
\begin{align*}
I&=\left(n\Log(\omega)-\frac{n^2}{N}b+\left(\frac{n^2}{N}-nN\right)\frac{b}{2(N+1)}+n\int_{\U}\left( \Re(\zeta)-1\right) \upsilon(\diff \zeta)\right)1_{\mathfrak{S}_n},\\
I\!I &=-\frac{b}{N+1} \displaystyle\sum_{\tau\in \mathcal{T}_n}N^{\ell(\tau\sigma)-\ell(\sigma)}\tau,
\end{align*}
and
\begin{multline*}
I\!I\!I=\displaystyle\sum_{\substack{2\leq m \leq n\\1\leq k_1<\ldots<k_m \leq n}}\displaystyle\sum_{\pi \in \mathfrak{S}_m}\int_{\U}\displaystyle(\zeta-1)^{m} \ \upsilon(\diff \zeta)\cdot \left(\displaystyle\sum_{\pi' \in \mathfrak{S}_m}\Wg({\pi'}^{-1} \pi)\right)\cdot N^{1+\ell\left(\iota_{k_1,\ldots,k_m}(\pi)\sigma\right)-\ell(\sigma)} \\
\cdot \iota_{k_1,\ldots,k_m}(\pi).
\end{multline*}
% I\!I\!I&=\displaystyle\sum_{\substack{2\leq m \leq n\\1\leq k_1<\ldots<k_m \leq n}}\displaystyle\sum_{\pi \in \mathfrak{S}_m}\int_{\U}\displaystyle(\zeta-1)^{m} \ \upsilon(\diff \zeta)\left(\displaystyle\sum_{\pi' \in \mathfrak{S}_m}\Wg({\pi'}^{-1} \pi)N^{1+\ell\left(\iota_{k_1,\ldots,k_m}(\pi)\sigma\right)-\ell(\sigma)} \right)\iota_{k_1,\ldots,k_m}(\pi)\sigma.
%\end{align*}
The first limit is immediate:  $$\lim_{N\to \infty}I=\left(n\Log(\omega)-\frac{n}{2}b+n\int_{\U}\left( \Re(\zeta)-1\right) \upsilon(\diff \zeta)\right)1_{\mathfrak{S}_n}=n L\kappa_1(\mu)1_{\mathfrak{S}_n}.$$
For the second and the third term, we recall that for all $ \pi\in \mathfrak{S}_n$ , we have $$d(1,\sigma)\leq d(1,\pi \sigma)+ d(\pi \sigma,\sigma)$$ with equality if and only if $\pi \sigma\preceq \sigma$ (see Section~\ref{dist}).

Let us focus on $I\!I$. We fix $\tau\in \mathcal{T}_n$. We know that $d(1,\sigma)\leq d(1,\tau \sigma)+ d(\tau \sigma,\sigma)$. In term of numbers of cycles, it means that $n-\ell(\sigma)\leq n-\ell(\tau \sigma)+n- \ell(\tau )$. Because $\ell(\tau )=n-1$, we have $ \ell(\tau \sigma)-\ell(\sigma)\leq 1$ with equality if and only if $\tau \sigma\preceq \sigma$. By consequence,
$$\lim_{N\to \infty}I\!I=-b\sum_{\substack{\tau\in \mathcal{T}_n\\\tau\sigma\preceq \sigma}} \tau\sigma.$$
A similar reasoning can be made for $I\!I\!I$. Let us fix $2\leq m \leq n,1\leq k_1<\ldots<k_m \leq n$ and $\pi\in \mathfrak{S}_m$. We denote by $c$ the permutation $\iota_{k_1,\ldots,k_m}(\pi).$ On one hand, Proposition~\ref{collins2} gives us $\Wg({\pi'}^{-1} \pi)=O(N^{-m-1})$ if $\pi\neq \pi'$ and $\Wg({\pi'}^{-1} \pi)=N^{-n}+O(N^{-n-1})$ if $\pi=\pi'$, and by consequence, $$\displaystyle\sum_{\pi' \in \mathfrak{S}_m}\Wg({\pi'}^{-1} \pi)=N^{-m}+O(N^{-m-1}).$$ On the other hand, we know that ${d(1,\sigma)\leq d(1,c \sigma)+ d(c \sigma,\sigma)}$. In terms of numbers of cycles, it means that $n-\ell(\sigma)\leq n-\ell(c \sigma)+n- \ell(c )$. Because $\ell(c )=\ell\left(\iota_{k_1,\ldots,k_m}(\pi)\right)=n-m+\ell(\pi)$, we have $1+\ell\left(c\sigma\right)-\ell(\sigma)\leq 1+m-\ell(\pi) $. Thus, we have,
$$1+\ell\left(c\sigma\right)-\ell(\sigma)\leq m$$
with equality if and only if  we have both $c \sigma\preceq \sigma$ and $\ell(\pi)=1$. Consequently, the term
$$\displaystyle\sum_{\pi' \in \mathfrak{S}_m}\Wg({\pi'}^{-1} \pi)N^{1+\ell\left(\iota_{k_1,\ldots,k_m}(\pi)\sigma\right)-\ell(\sigma)} $$
is equal to $1+O(N^{-1})$ if we have both $c \sigma\preceq \sigma$ and $\ell(\pi)=1$, but it is $O(N^{-1})$ if not. Finally,
\begin{align*}\lim_{N\to \infty}I\!I\!I&=\displaystyle\sum_{\substack{2\leq m \leq n\\1\leq k_1<\ldots<k_m \leq n}}\displaystyle\sum_{\substack{\pi\ m-\text{cycle of } \mathfrak{S}_m\\ \iota_{k_1,\ldots,k_m}(\pi)\sigma\preceq \sigma}}\int_{\U}\displaystyle(\zeta-1)^{m} \ \upsilon(\diff \zeta) \cdot \iota_{k_1,\ldots,k_m}(\pi)\sigma\\
&=\displaystyle\sum_{2\leq m \leq n}\ \ \displaystyle\sum_{\substack{c\ m-\text{cycle of } \mathfrak{S}_n\\ c\sigma\preceq \sigma}}\int_{\U}\displaystyle(\zeta-1)^{m} \ \upsilon(\diff \zeta) \cdot c\sigma . \end{align*}
Thus, we have
\begin{equation*}
\lim_{N\to \infty} I+I\!I+I\!I\!I=n L\kappa_1(\mu)\cdot \sigma+\displaystyle\sum_{\substack{2\leq m \leq n\\c\ m\text{-cycle of }\mathfrak{S}_n\\c\sigma\preceq \sigma}} L\kappa_{m}\left(\mu \right)\cdot c\sigma=T(\sigma).\qedhere
\end{equation*}
\end{proof}
As a consequence, we have
\begin{align*} &\lim_{N\to \infty}\E\left[N^{-\ell(\sigma)}\prod_{c\text{ cycle of }\sigma}\Tr(U_N^{\sharp c})\right]=\phi(e^{T}(\sigma))=\phi(e^{n L\kappa_1(\mu)}e^{T-nL\kappa_1(\mu)}(\sigma))\\
&\hspace{3cm}=\phi\left(e^{n L\kappa_1(\mu)}\displaystyle\sum_{\substack{\Gamma \text{ simple chain in }[1,\sigma]\\ \Gamma=(\sigma_0,\ldots,\sigma_{|\Gamma|}),\sigma_{|\Gamma|}=\sigma}} \frac{1}{|\Gamma| !}\displaystyle\prod_{i=1}^{|\Gamma|} L\kappa_{d(\sigma_i,\sigma_{i-1})+1}\left(\mu \right)\cdot \sigma_0\right)\\
&\hspace{3cm}=e^{n L\kappa_1(\mu)}\displaystyle\sum_{\substack{\Gamma \text{ simple chain in }[1,\sigma]\\ \Gamma=(\sigma_0,\ldots,\sigma_{|\Gamma|}),\sigma_{|\Gamma|}=\sigma}} \frac{1}{|\Gamma| !}\displaystyle\prod_{i=1}^{|\Gamma|} L\kappa_{d(\sigma_i,\sigma_{i-1})+1}\left(\mu \right).
\end{align*}
Using~\eqref{taulogcumdeux} on the right-hand side, we conclude that \begin{equation*}\lim_{N\to \infty} \E\left[N^{-\ell(\sigma)}\prod_{c\text{ cycle of }\sigma}\Tr(U_N^{\sharp c})\right]=\prod_{c\text{ cycle of }\sigma}m_{\sharp c}(\mu).\qedhere\end{equation*}

\end{proof}

\begin{corollary}Let $\mu\in \mathcal{ID}(\mathbb{U},\boxtimes)$. For all $N\in \N^*$, let $U_{N}$ be a random matrix whose law is $\Gamma_N(\mu)$, and whose empirical spectral measure is\label{thsix}
$$\hat{\mu}_{U_N}=\frac{1}{N}\sum_{\substack{\textup{eigenvalue }\lambda\textup{ of }U^{N}\\ \textup{(with multiplicity)} }}\delta_{\lambda} .$$
Then, the measures $\E[\hat{\mu}_{U_{N}}]$ converge weakly to $\mu$ when $N$ tends to $\infty$.
\end{corollary}

\begin{proof}We verify the convergence of moments. Let $n\in \N$. We have
$$
\int_\U\zeta^n\diff\E\left[\hat{\mu}_{U_{N}}\right]=\E\left[\int_\U\zeta^n\diff\hat{\mu}_{U_{N}}\right]
=\E\left[\frac{1}{N}\Tr((U_{N})^n)\right]
$$
which tends to $m_n(\mu)$ as $N$ tends to $\infty$.
\end{proof}

\begin{remark}
In fact, the proof can be easily extended to a more general situation. Let $\mu\in \mathcal{ID}(\mathbb{U},\boxtimes)$ and let $(\omega,b,\upsilon)$ be its $\boxtimes$-characteristic triplet. For all $N\in \N^*$, let $y_0, \alpha,\beta\in \mathbb{R}$ and $\Pi$ be a L\'{e}vy measure on $U(N)$ which is conjugate invariant. We suppose that
\begin{enumerate}
\item $\lim_{N\to \infty}e^{iy_0} =\omega$,  $\alpha\sim_{N\to \infty} \frac{b}{N}$ and $\beta=O(1)$ as $N$ tends to $\infty$;
\item for all $k_1,\ldots,k_n \in \N$, $$\lim_{N\to \infty}\frac{1}{N}\int_{U(N)}\Tr((g-I_N)^{k_1})\cdots \Tr((g-I_N)^{k_1})\Pi(\diff g)=m_{k_1}(\mu)\cdots m_{k_n}(\mu).$$
\end{enumerate}
Then, the conclusions of Theorem~\ref{thsixp} and Corollary~\ref{thsix} are still true whenever $U_{N}$ is a random matrix whose law is an infinitely divisible measure which admits $$\left(i y_0I_N,\left(\begin{array}{cccc}\alpha &  &  & 0 \\ & \ddots &  &  \\ &  & \alpha &  \\0 &  &  & \beta\end{array}\right), \Pi\right)$$
as a  characteristic triplet. \end{remark}

\subsection{Proof of Theorem~\ref{whole}}\label{wholly}We refer the reader to~\cite{Voiculescu1992} for the main definitions of free probability spaces.  We call {\it free unitary multiplicative Lévy process} a family $( U_t)_{t\in\R_+}$ of unitary elements of a non-commutative probability space $(\mathcal{A},\tau)$ such that\label{fmbm}
\begin{enumerate}
\item $U_0=1_{\mathcal{A}}$;
\item For all $0\leq s\leq t$, the distribution of $U_tU_s^{-1}$ depends only on $t-s$;
\item For all $0\leq t_1<\ldots <t_n$, the elements $U_{t_1},U_{t_2}U_{t_1}^{-1},\ldots U_{t_n}U_{t_{n-1}}^{-1}$ are freely independent;
\item The distribution of $U_t$ converge weakly to $\delta_1$ as $t$ tends to $0$.
\end{enumerate}
Notice that this definition differs from the definition in~\cite{Biane1998} by the first and the fourth items.

Let $( U_t)_{t\in\R_+}$ be a free unitary multiplicative Lévy process with marginal distributions $( \mu_t)_{t\in\R_+}$ in $\Ms$. Then, $( \mu_t)_{t\in\R_+}$ is a weakly continuous semigroup of measures for the convolution $\boxtimes$ on $\U$. Moreover, there exists $\alpha\in \R$ and $b\geq 0$ and $\upsilon$ a Lévy measure on $\U$ such that, for all $t\geq 0$, $(e^{i\alpha t},tb,t\upsilon)$ is a $\boxtimes$-characteristic triplet of $U_t$ (see~\cite{Bercovici1992}). Using Lemma~\ref{tripletpi}, it is straightforward to verify that the weakly continuous semigroup whose characteristic triplet is $(i\alpha I_N,b_N,\upsilon_N)$ coincides with
% As previously, we define $b_N$ to be the $N^2\times N^2$-matrix
%$$b_N=\left(\begin{array}{cccc}\frac{b}{N+1} &  &  & 0 \\ & \ddots &  &  \\ &  & \frac{b}{N+1} &  \\0 &  &  & b\end{array}\right) .$$ We denote by $ \upsilon_N$ the Lévy measure on $U(N)$ defined as follows: for all Borel positive function $f$ on $U(N)$, we have
%$$\int _{U(N)}f\diff  \upsilon_N=N \int_\U \int_{U(N)} f\left(g\left(\begin{array}{cccc}\zeta & 0 & \cdots & 0 \\0 & 1 & \ddots & \vdots \\\vdots & \ddots & \ddots & 0 \\0 & \cdots & 0 & 1\end{array}\right)g^*\right)\diff g\diff \upsilon(\zeta) $$
%where the integration over $U(N)$ is taken with respect to the Haar measure of $U(N)$. 
%The triplet $(i\alpha I_N,b_N,\upsilon_N)$ is the characteristic triplet of the weakly continuous semigroup 
$(\Gamma_N(\mu_t))_{t\in \R_+}$. Therefore, there exists a Lévy process $(U_t^{(N)})_{t\in \R_+}$ in $U(N)$ such that $\Gamma_N(\mu_t)$ is the distribution of $U_t^{(N)}$ for each $t\in \R_+$ (see~\cite{Liao2004}). We already know that, for each fixed $t\in \R_+$, the element $U_t^{(N)}$ converges to $U_t$ in non-commutative {$*$-distribution}, in the sense that, for each non-commutative polynomial $P$ in two variables, one has the convergence
$$\lim_{N\mapsto \infty}\frac{1}{N}\E\left[\Tr\left(P\left(U_t^{(N)},{U_t^{(N)}}^*\right)\right)\right]=\tau(P(U_t,U_t^*)). $$
Since the increments of $( U_t)_{t\in\R_+}$ are freely independent, to prove the convergence of the whole process, it suffices to prove that the increments of $(U_t^{(N)})_{t\in \R_+}$ are asymptotically free. This is a well-known consequence of the factorization property of Theorem~\ref{thsixp} and the fact that the increments of $(U_t^{(N)})_{t\in \R_+}$ are independent and invariant under conjugation by unitary matrices (see for example~\cite{Collins2003,Voiculescu1992,Voiculescu1191,Xu1997}, or the appendix of~\cite{LEVY2011} for a concise treatment).

%\subsection{Asymptotic freeness}
%We give some precision about Theorem~\ref{ssss}.
%\label{fmlp}

%\addtocontents{toc}{\protect\setcounter{tocdepth}{1}}

%\begin{appendix}
%
%\section*{Appendix. The construction of $\mathbb{C}\{X_i:i\in I\}$}
%\subsection*{sdv}
%
%\end{appendix}

\subsection*{Acknowlegments}
The author would like to gratefully thank Thierry Cabanal-Duvillard, Antoine Dahlqvist and Franck Gabriel for useful discussions, and his PhD advisor Thierry Lévy for his helpful comments which led to improvement in this manuscript. 

%\nocite{*}

\bibliographystyle{amsalpha}
\bibliography{/Users/guillaumecebron/Documents/Bibtex/freelevy}

\providecommand{\bysame}{\leavevmode\hbox to3em{\hrulefill}\thinspace}
\providecommand{\MR}{\relax\ifhmode\unskip\space\fi MR }
\providecommand{\MRhref}[2]{%
  \href{http://www.ams.org/mathscinet-getitem?mr=#1}{#2}
}
\providecommand{\href}[2]{#2}
\begin{thebibliography}{DMPARA13}

\bibitem[AK93]{Applebaum1993}
David Applebaum and Hiroshi Kunita, \emph{{L\'{e}vy flows on manifolds and
  L\'{e}vy processes on Lie groups}}, Kyoto Journal of Mathematics \textbf{33}
  (1993), no.~4, 1103--1123 (EN).

\bibitem[AWZ13]{Anshelevich2013}
Michael Anshelevich, Jiun-Chau Wang, and Ping Zhong, \emph{{Local limit
  theorems for multiplicative free convolutions}}, arXiv:1312.2487 (2013).

\bibitem[BG05]{Benaych-Georges2005}
Florent Benaych-Georges, \emph{{Classical and free infinitely divisible
  distributions and random matrices}}, The Annals of Probability \textbf{33}
  (2005), no.~3, 1134--1170.

\bibitem[Bia97a]{Biane1997a}
Philippe Biane, \emph{{Free Brownian motion, free stochastic calculus and
  random matrices}}, Free probability theory, (Waterloo ON, 1995), vol.~12,
  Amer. Math. Soc., Providence, RI, 1997, pp.~1--19.

\bibitem[Bia97b]{Biane1997b}
\bysame, \emph{{Some properties of crossings and partitions}}, Discrete
  Mathematics \textbf{175} (1997), no.~1, 41--53.

\bibitem[Bia98]{Biane1998}
\bysame, \emph{{Processes with free increments}}, Mathematische Zeitschrift
  \textbf{227} (1998), no.~1, 143--174.

\bibitem[BNT05]{Barndorff2005}
Ole~E. Barndorff-Nielsen and Steen Thorbj{\o}rnsen, \emph{{Classical and Free
  Infinite Divisibility and L\'{e}vy Processes}}, Quantum Independent Increment
  Processes II (Michael Sch\"{u}ermann and Uwe Franz, eds.), Lecture Notes in
  Mathematics, vol. 1866, Springer-Verlag, Berlin/Heidelberg, 2005,
  pp.~33--159.

\bibitem[BPB99]{Bercovici1999}
Hari Bercovici, Vittorino Pata, and Philippe Biane, \emph{{Stable laws and
  domains of attraction in free probability theory}}, Annals of Mathematics
  \textbf{149} (1999), no.~3, 1023--1060.

\bibitem[BV92]{Bercovici1992}
Hari Bercovici and Dan Voiculescu, \emph{{L\'{e}vy-Hin\v{c}in type theorems for
  multiplicative and additive free convolution.}}, Pacific Journal of
  Mathematics \textbf{153} (1992), no.~2, 217--248.

\bibitem[BW08]{Bercovici2008}
Hari Bercovici and Jiun-Chau Wang, \emph{{Limit theorems for free
  multiplicative convolutions}}, Trans. Amer. Math. Soc. \textbf{360} (2008),
  no.~11, 6089--6102.

\bibitem[C\'{e}13]{Cebron2013}
Guillaume C\'{e}bron, \emph{{Free convolution operators and free Hall
  transform}}, Journal of Functional Analysis \textbf{265} (2013), no.~11,
  2645--2708.

\bibitem[CD05]{Cabanal-Duvillard2005}
Thierry Cabanal-Duvillard, \emph{{A Matrix Representation of the Bercovici-Pata
  Bijection}}, Electronic Journal of Probability \textbf{10} (2005), 632--661.

\bibitem[CG08]{Chistyakov2008}
Gennadii~P. Chistyakov and Friedrich G\"{o}tze, \emph{{Limit theorems in free
  probability theory II}}, Central European Journal of Mathematics \textbf{6}
  (2008), no.~1, 87--117.

\bibitem[Col03]{Collins2003}
Beno\^{\i}t Collins, \emph{{Moments and cumulants of polynomial random
  variables on unitary groups, the Itzykson-Zuber integral, and free
  probability}}, Int. Math. Res. Not. (2003), no.~17, 953--982.

\bibitem[C\'{S}04]{Collins2004}
Benoit Collins and Piotr \'{S}niady, \emph{{Integration with respect to the
  Haar measure on unitary, orthogonal and symplectic group}}, Communications in
  Mathematical Physics \textbf{264} (2004), no.~3, 773--795.

\bibitem[Dah12]{Dahlqvist2012}
Antoine Dahlqvist, \emph{{Integration formula for Brownian motion on classical
  compact Lie groups}}, arXiv:1212.5107 (2012).

\bibitem[DMPARA13]{Dominguez2013}
J.~Armando Dom\'{\i}nguez-Molina, V\'{\i}ctor P\'{e}rez-Abreu, and Alfonso
  Rocha-Arteaga, \emph{{Covariation representations for Hermitian L\'{e}vy
  process ensembles of free infinitely divisible distributions}}, Electronic
  Communications in Probability \textbf{18} (2013), 1--14 (en).

\bibitem[DMRA12]{Dominguez2012}
J.~A. Dom\'{\i}nguez-Molina and A.~Rocha-Arteaga, \emph{{Random matrix models
  of stochastic integrals type for free infinitely divisible distributions}},
  Periodica Mathematica Hungarica \textbf{64} (2012), no.~2, 145--160.

\bibitem[Est92]{Estrade1992}
Anne Estrade, \emph{{Exponentielle stochastique et int\'{e}grale multiplicative
  discontinues.}}, Ann. Inst. Henri Poincar\'{e}, Probab. Stat. \textbf{28}
  (1992), no.~1, 107--129 (French).

\bibitem[L\'{e}08]{LEVY2008}
Thierry L\'{e}vy, \emph{{Schur–Weyl duality and the heat kernel measure on
  the unitary group}}, Advances in Mathematics \textbf{218} (2008), no.~2,
  537--575.

\bibitem[L\'{e}10]{LEVY2010}
\bysame, \emph{{Two-dimensional Markovian holonomy fields}}, 329 ed.,
  Ast\'{e}risque, 2010.

\bibitem[L\'{e}11]{LEVY2011}
\bysame, \emph{{The master field on the plane}}, arXiv:1112.2452 (2011).

\bibitem[Lia04]{Liao2004}
Ming Liao, \emph{{L\'{e}vy Processes in Lie Groups}}, Cambridge University
  Press, 2004.

\bibitem[MN10]{Mastnak2010}
Mitja Mastnak and Alexandru Nica, \emph{{Hopf algebras and the logarithm of the
  S-transform in free probability}}, Transactions of the American Mathematical
  Society \textbf{362} (2010), no.~07, 3705--3743.

\bibitem[Par67]{Parthasarathy1967}
K.~R. Parthasarathy, \emph{{Probability measures on metric spaces}}, Academic
  Press, New York, London, 1967.

\bibitem[Sat99]{Sato1999}
K.~I. Sato, \emph{{L\'{e}vy processes and infinitely divisible distributions}},
  Cambridge University Press, 1999.

\bibitem[Sch83]{Schatte1983}
P.~Schatte, \emph{{On sums modulo $2\pi$ of independent random variables}},
  Mathematische Nachrichten \textbf{110} (1983), no.~1, 243--262.

\bibitem[VDN92]{Voiculescu1992}
D~V Voiculescu, K~J Dykema, and A~Nica, \emph{{Free Random Variables}}, vol.~1,
  CRM Monograph Series, no.~X, American Mathematical Society, 1992.

\bibitem[Voi91]{Voiculescu1191}
Dan Voiculescu, \emph{{Limit laws for random matrices and free products}},
  Invent. Math. \textbf{104} (1991), no.~1, 201--220.

\bibitem[Xu97]{Xu1997}
Feng Xu, \emph{{A random matrix model from two-dimensional Yang-Mills theory}},
  Comm. Math. Phys. \textbf{190} (1997), no.~2, 287--307.

\end{thebibliography}

% \printindex

% \newpage

\end{document}